\RequirePackage{snapshot}
\documentclass[11pt,a4paper,twoside]{amsart}
\usepackage[top=1.5in, bottom=1.25in, marginparwidth=1in, inner=1.25in, outer=1.25in]{geometry}

\usepackage{lmodern,microtype,array,thm-restate}
\usepackage[T2A,T1]{fontenc} %% T2A for Cyrillic letters
\usepackage[utf8]{inputenc}
\usepackage[UKenglish]{babel}
\makeatletter
\newcommand*{\math@version@bold}{bold}
\DeclareMathOperator\DD{%
	\textrm{%
		\usefont{T2A}{cmr}{\ifx\math@version\math@version@bold bx\else m\fi}{n}%
		\CYRD
	}%
} 
\makeatother

\usepackage{xcolor}
\usepackage{color, colortbl}
\colorlet{darkgreen}{green!70!black}
\definecolor{darkblue}{RGB}{0,0,139}
\definecolor{lightblue}{RGB}{173,216,230}
\colorlet{lightviolet}{violet!30!white}
\colorlet{lightblue}{blue!30!white}
\colorlet{lightred}{red!30!white}
\usepackage{transparent}
\usepackage{caption}
\usepackage[labelfont=up,margin=5pt]{subcaption}
\usepackage{wrapfig}
\usepackage{tikz}
\usepackage{tikz-cd}
\usepackage{pinlabel}
\usepackage{booktabs}

\usepackage{setspace} 
\usepackage[bookmarks=true,pdfencoding=auto,hidelinks,pagebackref]{hyperref}
\usepackage[nameinlink]{cleveref}
\let\cref\Cref
\Crefname{equation}{}{}
\Crefname{subsection}{Subsection}{Subsection}
\Crefname{enumi}{}{}

\usepackage{amsmath}
\usepackage{amsfonts}
\usepackage{amssymb}
\usepackage{amsthm}
\usepackage{nicefrac}
\usepackage{mathtools}
\usepackage[shortlabels]{enumitem}
\usepackage{stmaryrd}
\usepackage{skak}
\mathtoolsset{mathic=true}
\usepackage{bm} %% for bold math

\newtheorem{theorem}{Theorem}[section]

\newtheorem{lemma}[theorem]{Lemma}

\newtheorem{corollary}[theorem]{Corollary}
\newtheorem{proposition}[theorem]{Proposition}
\theoremstyle{definition}

\newtheorem{definition}[theorem]{Definition}
\newtheorem{example}[theorem]{Example}
\newtheorem{remark}[theorem]{Remark}

\newcommand{\qua}{\hskip 0.4em \ignorespaces}
\def\arxiv#1{\relax\ifhmode\unskip\qua\fi
\href{http://arxiv.org/abs/#1}%
{\tt arXiv:\penalty -100\unskip#1}}

\def\MR#1{\relax\ifhmode\unskip\qua\fi
\href{http://www.ams.org/mathscinet-getitem?mr=#1}{\tt MR#1}}
\def\ZB#1{\relax\ifhmode\unskip\qua\fi
\href{https://zbmath.org/?q=an:#1}{\tt Zbl\:#1}}
\def\xox#1{\csname xx#1\endcsname}

\renewcommand*{\backrefalt}[4]{%
\ifcase #1 %
No citations.%
\or
Cited on page~#2.%
\else
Cited on pages~#2.%
\fi
}

\newcommand{\lambdau}{\lambda^0}
\newcommand{\myqed}{\pushQED{\qed}\qedhere}

\graphicspath{{figures/}}
\newcommand{\BNr}{{\operatorname{CBN}}}%
\newcommand{\Khr}{\widetilde{\operatorname{Kh}}}%
\newcommand{\id}{\operatorname{id}}%
\newcommand{\Cob}{\operatorname{Cob}}%
\newcommand{\Mor}{\operatorname{Mor}}%
\newcommand{\BNAlgH}{\mathcal{B}}%

\newcommand{\F}{\mathbb{F}}
\newcommand{\Z}{\mathbb{Z}}
\newcommand{\Q}{\mathbb{Q}}

\newcommand{\QGrad}[1]{{\textcolor{violet}{#1}}}

\newcommand{\HomGrad}[1]{{\textcolor{black}{#1}}}

\newcommand{\GGzqh}[4]{\prescript{\QGrad{#3}}{}{#1}_\HomGrad{#4}}

\newcommand{\conn}[1]{\operatorname{x}(#1)}

\renewcommand{\theta}{\ensuremath{\vartheta}}

\def\co{\colon\thinspace\relax}%

\newcommand{\NCr}{\textnormal{$\vcenter{\hbox{\def\svgwidth{11.6pt}%% Creator: Inkscape inkscape 0.92.5, www.inkscape.org
\begingroup%
  \makeatletter%
  \providecommand\color[2][]{%
    \errmessage{(Inkscape) Color is used for the text in Inkscape, but the package 'color.sty' is not loaded}%
    \renewcommand\color[2][]{}%
  }%
  \providecommand\transparent[1]{%
    \errmessage{(Inkscape) Transparency is used (non-zero) for the text in Inkscape, but the package 'transparent.sty' is not loaded}%
    \renewcommand\transparent[1]{}%
  }%
  \providecommand\rotatebox[2]{#2}%
  \newcommand*\fsize{\dimexpr\f@size pt\relax}%
  \newcommand*\lineheight[1]{\fontsize{\fsize}{#1\fsize}\selectfont}%
  \ifx\svgwidth\undefined%
    \setlength{\unitlength}{11.61026645bp}%
    \ifx\svgscale\undefined%
      \relax%
    \else%
      \setlength{\unitlength}{\unitlength * \real{\svgscale}}%
    \fi%
  \else%
    \setlength{\unitlength}{\svgwidth}%
  \fi%
  \global\let\svgwidth\undefined%
  \global\let\svgscale\undefined%
  \makeatother%
  \begin{picture}(1,0.99940785)%
    \lineheight{1}%
    \setlength\tabcolsep{0pt}%
    \put(0,0){\includegraphics[width=\unitlength,page=1]{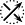}}%
  \end{picture}%
\endgroup%
}}$}}
\newcommand{\PCr}{\textnormal{$\vcenter{\hbox{\def\svgwidth{11.6pt}%% Creator: Inkscape inkscape 0.92.5, www.inkscape.org
\begingroup%
  \makeatletter%
  \providecommand\color[2][]{%
    \errmessage{(Inkscape) Color is used for the text in Inkscape, but the package 'color.sty' is not loaded}%
    \renewcommand\color[2][]{}%
  }%
  \providecommand\transparent[1]{%
    \errmessage{(Inkscape) Transparency is used (non-zero) for the text in Inkscape, but the package 'transparent.sty' is not loaded}%
    \renewcommand\transparent[1]{}%
  }%
  \providecommand\rotatebox[2]{#2}%
  \newcommand*\fsize{\dimexpr\f@size pt\relax}%
  \newcommand*\lineheight[1]{\fontsize{\fsize}{#1\fsize}\selectfont}%
  \ifx\svgwidth\undefined%
    \setlength{\unitlength}{11.61026587bp}%
    \ifx\svgscale\undefined%
      \relax%
    \else%
      \setlength{\unitlength}{\unitlength * \real{\svgscale}}%
    \fi%
  \else%
    \setlength{\unitlength}{\svgwidth}%
  \fi%
  \global\let\svgwidth\undefined%
  \global\let\svgscale\undefined%
  \makeatother%
  \begin{picture}(1,0.99940789)%
    \lineheight{1}%
    \setlength\tabcolsep{0pt}%
    \put(0,0){\includegraphics[width=\unitlength,page=1]{rational-1.pdf}}%
  \end{picture}%
\endgroup%
}}$}}
\newcommand{\cinfty}{\vcenter{\hbox{\def\svgwidth{8pt}%% Creator: Inkscape inkscape 0.92.5, www.inkscape.org
\begingroup%
  \makeatletter%
  \providecommand\color[2][]{%
    \errmessage{(Inkscape) Color is used for the text in Inkscape, but the package 'color.sty' is not loaded}%
    \renewcommand\color[2][]{}%
  }%
  \providecommand\transparent[1]{%
    \errmessage{(Inkscape) Transparency is used (non-zero) for the text in Inkscape, but the package 'transparent.sty' is not loaded}%
    \renewcommand\transparent[1]{}%
  }%
  \providecommand\rotatebox[2]{#2}%
  \newcommand*\fsize{\dimexpr\f@size pt\relax}%
  \newcommand*\lineheight[1]{\fontsize{\fsize}{#1\fsize}\selectfont}%
  \ifx\svgwidth\undefined%
    \setlength{\unitlength}{11.6198051bp}%
    \ifx\svgscale\undefined%
      \relax%
    \else%
      \setlength{\unitlength}{\unitlength * \real{\svgscale}}%
    \fi%
  \else%
    \setlength{\unitlength}{\svgwidth}%
  \fi%
  \global\let\svgwidth\undefined%
  \global\let\svgscale\undefined%
  \makeatother%
  \begin{picture}(1,0.99995636)%
    \lineheight{1}%
    \setlength\tabcolsep{0pt}%
    \put(0,0){\includegraphics[width=\unitlength,page=1]{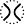}}%
  \end{picture}%
\endgroup%
}}}
\newcommand{\czero}{\vcenter{\hbox{\def\svgwidth{8pt}%% Creator: Inkscape inkscape 0.92.5, www.inkscape.org
\begingroup%
  \makeatletter%
  \providecommand\color[2][]{%
    \errmessage{(Inkscape) Color is used for the text in Inkscape, but the package 'color.sty' is not loaded}%
    \renewcommand\color[2][]{}%
  }%
  \providecommand\transparent[1]{%
    \errmessage{(Inkscape) Transparency is used (non-zero) for the text in Inkscape, but the package 'transparent.sty' is not loaded}%
    \renewcommand\transparent[1]{}%
  }%
  \providecommand\rotatebox[2]{#2}%
  \newcommand*\fsize{\dimexpr\f@size pt\relax}%
  \newcommand*\lineheight[1]{\fontsize{\fsize}{#1\fsize}\selectfont}%
  \ifx\svgwidth\undefined%
    \setlength{\unitlength}{11.61980321bp}%
    \ifx\svgscale\undefined%
      \relax%
    \else%
      \setlength{\unitlength}{\unitlength * \real{\svgscale}}%
    \fi%
  \else%
    \setlength{\unitlength}{\svgwidth}%
  \fi%
  \global\let\svgwidth\undefined%
  \global\let\svgscale\undefined%
  \makeatother%
  \begin{picture}(1,0.99995649)%
    \lineheight{1}%
    \setlength\tabcolsep{0pt}%
    \put(0,0){\includegraphics[width=\unitlength,page=1]{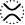}}%
  \end{picture}%
\endgroup%
}}}
\newcommand{\cx}{\vcenter{\hbox{\def\svgwidth{8pt}%% Creator: Inkscape inkscape 0.92.5, www.inkscape.org
\begingroup%
  \makeatletter%
  \providecommand\color[2][]{%
    \errmessage{(Inkscape) Color is used for the text in Inkscape, but the package 'color.sty' is not loaded}%
    \renewcommand\color[2][]{}%
  }%
  \providecommand\transparent[1]{%
    \errmessage{(Inkscape) Transparency is used (non-zero) for the text in Inkscape, but the package 'transparent.sty' is not loaded}%
    \renewcommand\transparent[1]{}%
  }%
  \providecommand\rotatebox[2]{#2}%
  \newcommand*\fsize{\dimexpr\f@size pt\relax}%
  \newcommand*\lineheight[1]{\fontsize{\fsize}{#1\fsize}\selectfont}%
  \ifx\svgwidth\undefined%
    \setlength{\unitlength}{11.61026645bp}%
    \ifx\svgscale\undefined%
      \relax%
    \else%
      \setlength{\unitlength}{\unitlength * \real{\svgscale}}%
    \fi%
  \else%
    \setlength{\unitlength}{\svgwidth}%
  \fi%
  \global\let\svgwidth\undefined%
  \global\let\svgscale\undefined%
  \makeatother%
  \begin{picture}(1,0.99940785)%
    \lineheight{1}%
    \setlength\tabcolsep{0pt}%
    \put(0,0){\includegraphics[width=\unitlength,page=1]{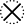}}%
  \end{picture}%
\endgroup%
}}}
\DeclareMathOperator{\Char}{char}

\usetikzlibrary{shapes.multipart, matrix, positioning, arrows,cd, decorations.pathmorphing, decorations.pathreplacing,calc,tikzmark,decorations.markings}

\tikzstyle{mor}=[%
->,
shorten >=4pt,
shorten <=4pt,
execute at begin node=$\scriptstyle,
execute at end node=$]

\newsavebox\justabox
\begin{lrbox}\justabox
\begin{tikzpicture}[scale=.60]
\draw (0,0) rectangle (1,1);
\end{tikzpicture}
\end{lrbox}

\newsavebox\OrangeBox
\begin{lrbox}{\OrangeBox}
\begin{tikzpicture}[scale=.60]
\draw[fill=orange!40] (0,0) rectangle (1,1);
\end{tikzpicture}
\end{lrbox}

\newsavebox\BlueBox
\begin{lrbox}{\BlueBox}
\begin{tikzpicture}[scale=.60]
\draw[fill=blue!30] (0,0) rectangle (1,1);
\end{tikzpicture}
\end{lrbox}

\newsavebox\GreenBox
\begin{lrbox}{\GreenBox}
\begin{tikzpicture}[scale=.60]
\draw[fill=green!30] (0,0) rectangle (1,1);
\end{tikzpicture}
\end{lrbox}

\newsavebox\bluepiece
\begin{lrbox}\bluepiece
\begin{tikzpicture}[scale=.60]
\node at (0,0){\usebox\BlueBox};
\node at (1,1){\usebox\BlueBox};
\end{tikzpicture}
\end{lrbox}

\newsavebox\greenpiece
\begin{lrbox}\greenpiece
\begin{tikzpicture}[scale=.60]
\node at (2,1){\usebox\GreenBox};
\node at (2,2){\usebox\justabox};
\node at (3,2){\usebox\justabox};
\node at (3,3){\usebox\GreenBox};
\end{tikzpicture}
\end{lrbox}

\newsavebox\theotherpieces
\begin{lrbox}\theotherpieces
\begin{tikzpicture}[scale=.50]
\draw[fill=white] (-2.6,-2.1) rectangle (6.8,5);
\node[scale=1.1] at (0,4){$n+1$ blue};
\node[scale=1.1] at (0,3){pieces};
\node at (0,0){\usebox\bluepiece};
\node[scale=1.1] at (4,4){$n$ green};
\node[scale=1.1] at (4,3){pieces};
\node at (4,0.5){\usebox\greenpiece};
\end{tikzpicture}
\end{lrbox}

\newsavebox\thepieces
\begin{lrbox}\thepieces
\begin{tikzpicture}[scale=.50]
\draw[fill=white] (-2,-2.1) rectangle (6,5);
\node[scale=1.1] at (0,4){$n$ blue};
\node[scale=1.1] at (0,3){pieces};
\node at (0,0){\usebox\bluepiece};
\node[scale=1.1] at (4,4){$n$ green};
\node[scale=1.1] at (4,3){pieces};
\node at (4,0.5){\usebox\greenpiece};
\end{tikzpicture}
\end{lrbox}

\newsavebox\piecesTwoN
\begin{lrbox}\piecesTwoN
\begin{tikzpicture}[scale=.50]
\draw[fill=white] (-2,-2.1) rectangle (1.8,5);
\node[scale=1.1] at (0,4){$m$ blue};
\node[scale=1.1] at (0,3){pieces};
\node at (0,0){\usebox\bluepiece};
\end{tikzpicture}
\end{lrbox}

\newsavebox\HomologyTorusThreeNPlusOne
\begin{lrbox}{\HomologyTorusThreeNPlusOne}
\begin{tikzpicture}[scale=.60]
\draw (3,2) grid (20,17);
\draw [fill,white] (13.5,1.5) rectangle (15.5,13.5);
\draw [fill,white] (2.5,11.5) rectangle (15.5,13.5);

\node at (2,1.5){\Large\nicefrac{\QGrad{q}}{t}};
\node at (1.5,2.5){\QGrad{6n}};
\node at (1.5,3.5){\QGrad{6n+2}};
\node at (1.5,16.5){\QGrad{12n+4}};
\node at (3.5,1.5){0};
\node at (4.5,1.5){1};
\node at (5.5,1.5){2};
\node at (6.5,1.5){3};
\node at (7.5,1.5){4};
\node at (8.5,1.5){5};
\node at (19.5,1.5){4n+1};

\node at (3.5,2.5){\usebox\OrangeBox};

\node at (5.5,4.5){\usebox\BlueBox};
\node at (6.5,5.5){\usebox\BlueBox};

\node at (7.5,5.5){\usebox\GreenBox};
\node at (8.5,7.5){\usebox\GreenBox};

\node at (9.5,7.5){\usebox\BlueBox};
\node at (10.5,8.5){\usebox\BlueBox};

\node at (11.5,8.5){\usebox\GreenBox};
\node at (12.5,10.5){\usebox\GreenBox};

\node[scale=2.5, very thick] at (14,12.5){\rotatebox[origin=c]{80}{$\ddots$}};

\node at (16.5,13.5){\usebox\BlueBox};
\node at (17.5,14.5){\usebox\BlueBox};

\node at (18.5,14.5){\usebox\GreenBox};
\node at (19.5,16.5){\usebox\GreenBox};

\draw [decorate,decoration={brace,amplitude=10pt,mirror,raise=4pt},yshift=0pt]
(8,3) -- (20,14) node [black,midway,xshift=0.8cm] {};

\node[scale=.75] at (17.2,5){\usebox\thepieces};
\end{tikzpicture}
\end{lrbox}

\newsavebox\HomologyTorusThreeNPlusTwo
\begin{lrbox}{\HomologyTorusThreeNPlusTwo}
\begin{tikzpicture}[scale=.60]
\draw (3,2) grid (22,18);
\draw [fill,white] (13.5,1.5) rectangle (15.5,13.5);
\draw [fill,white] (2.5,11.5) rectangle (15.5,13.5);

\node at (2,1.5){\Large\nicefrac{\QGrad{q}}{t}};
\node at (1.5,2.5){\QGrad{6n+2}};
\node at (1.5,3.5){\QGrad{6n+4}};
\node at (1.5,17.5){\QGrad{12n+8}};
\node at (3.5,1.5){0};
\node at (4.5,1.5){1};
\node at (5.5,1.5){2};
\node at (6.5,1.5){3};
\node at (7.5,1.5){4};
\node at (8.5,1.5){5};
\node at (21.5,1.5){4n+3};

\node at (3.5,2.5){\usebox\OrangeBox};

\node at (5.5,4.5){\usebox\BlueBox};
\node at (6.5,5.5){\usebox\BlueBox};

\node at (7.5,5.5){\usebox\GreenBox};
\node at (8.5,7.5){\usebox\GreenBox};

\node at (9.5,7.5){\usebox\BlueBox};
\node at (10.5,8.5){\usebox\BlueBox};

\node at (11.5,8.5){\usebox\GreenBox};
\node at (12.5,10.5){\usebox\GreenBox};

\node[scale=2.5, very thick] at (14,12.5){\rotatebox[origin=c]{80}{$\ddots$}};

\node at (16.5,13.5){\usebox\BlueBox};
\node at (17.5,14.5){\usebox\BlueBox};

\node at (18.5,14.5){\usebox\GreenBox};
\node at (19.5,16.5){\usebox\GreenBox};

\node at (20.5,16.5){\usebox\BlueBox};
\node at (21.5,17.5){\usebox\BlueBox};

\draw [decorate,decoration={brace,amplitude=10pt,mirror,raise=4pt},yshift=0pt]
(8,3) -- (22,14.8) node [black,midway,xshift=0.8cm] {};

\node[scale=.75] at (18.5,5){\usebox\theotherpieces};
\end{tikzpicture}
\end{lrbox}

\newsavebox\HomologyTorusTwoNPlusOne
\begin{lrbox}{\HomologyTorusTwoNPlusOne}
\begin{tikzpicture}[scale=.60]
\draw (3,2) grid (13,12);
\draw [fill,white] (9.5,1.5) rectangle (10.5,9.5);
\draw [fill,white] (2.5,8.5) rectangle (10.5,9.5);

\node at (2,1.5){\Large\nicefrac{\QGrad{q}}{t}};
\node at (1.5,2.5){\QGrad{2m}};
\node at (1.5,3.5){\QGrad{2m+2}};
\node at (1.5,11.5){\QGrad{6m+2}};
\node at (3.5,1.5){0};
\node at (4.5,1.5){1};
\node at (5.5,1.5){2};
\node at (6.5,1.5){3};
\node at (12.5,1.5){2m+1};

\node at (3.5,2.5){\usebox\OrangeBox};

\node at (5.5,4.5){\usebox\BlueBox};
\node at (6.5,5.5){\usebox\BlueBox};

\node at (7.5,6.5){\usebox\BlueBox};
\node at (8.5,7.5){\usebox\BlueBox};

\node[scale=2.5, very thick] at (9.5,9){\rotatebox[origin=c]{80}{$\ddots$}};

\node at (11.5,10.5){\usebox\BlueBox};
\node at (12.5,11.5){\usebox\BlueBox};

\draw [decorate,decoration={brace,amplitude=10pt,mirror,raise=4pt},yshift=0pt]
(6,3) -- (13,10) node [black,midway,xshift=0.8cm] {};

\node[scale=.75] at (11.7,4.5){\usebox\piecesTwoN};
\end{tikzpicture}
\end{lrbox}

\title[Khovanov homology and refined bounds for Gordian distances]{Khovanov homology \break and refined bounds for Gordian distances}

\newcommand{\myemail}[1]{\href{mailto:#1}{#1}}
\author{Lukas Lewark}
\address{ETH Z\"urich, R\"amistrasse 101, 8092 Z\"urich, Switzerland}
\email{\myemail{llewark@math.ethz.ch}}
\urladdr{\url{https://people.math.ethz.ch/~llewark/}}

\author{Laura Marino}
\address{Institut de Mathématiques de Jussieu -- Paris Rive Gauche (IMJ-PRG)}
\email{\myemail{marino@math.cnrs.fr}}
\urladdr{\url{https://sites.google.com/view/laura--marino/}}

\author{Claudius Zibrowius}
\address{Fakultät für Mathematik, Ruhr-Universität Bochum, Universitätsstraße 150, 44780 Bochum, Germany}
\email{\myemail{claudius.zibrowius@rub.de}}
\urladdr{\url{https://cbz20.raspberryip.com/}}

\newcommand{\vc}[1]{\vcenter{\hbox{#1}}}%
\newcommand{\mypic}[2]{%
  \newcommand{#2}{%
    \vc{%
      \includegraphics[page=#1]%
      {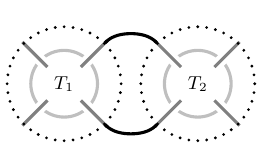}%
    }%
  }%
}%

\mypic{1}{\tanglesum}
\mypic{2}{\union}
\mypic{3}{\rationalclosure}
\mypic{4}{\Nib}
\mypic{5}{\Nob}
\mypic{6}{\Li}
\mypic{7}{\Lo}
\mypic{8}{\Ni}
\mypic{9}{\No}
\mypic{10}{\ConnectivityX}
\mypic{11}{\CrossingL}
\mypic{12}{\CrossingR}
\mypic{13}{\CrossingY}
\mypic{14}{\CrossingX}
\mypic{15}{\Crossing}
\mypic{16}{\CrossingLo}
\mypic{17}{\CrossingRo}
\mypic{18}{\ratixo}
\mypic{19}{\ratoxi}
\mypic{20}{\ratixi}
\mypic{21}{\ratiixi}
\mypic{22}{\ratmiiixi}
\mypic{23}{\muty}
\mypic{24}{\QuiverOnSurfaceDualWithQuiver}
\mypic{25}{\TorusknotFromPT}
\mypic{26}{\UnknotFromPT}

\hypersetup{pdfauthor={\authors},pdftitle={\shorttitle}}

\begin{document}

\begin{abstract}
From Khovanov homology, we extract a new lower bound for the Gordian distance of knots, which combines and strengthens the previously existing bounds coming from Rasmussen invariants and from torsion invariants. We also improve the bounds for the proper rational Gordian distance.
\end{abstract}

\maketitle

\section{Introduction}\label{sec:intro}

The historically first and still most famous geometric application of Khovanov homology is
the lower bound for the smooth slice genus
given by the Rasmussen invariant~\cite{RasmussenSlice}.
Recent years, however, have seen the rise of a new flavor of applications,
such as the lower bounds for the unknotting number given by Alishahi~\cite{zbMATH07178864} and Alishahi--Dowlin~\cite{zbMATH07005602}.
Let us take the perspective of the universal version of Khovanov homology introduced by Naot~\cite{naot} based on Bar-Natan's homology~\cite{BarNatanKhT},
which unifies the likes of Lee homology,
Bar-Natan homology, and both unreduced and reduced Khovanov homology.
This homology theory associates with a knot $K$ the homotopy class
of a graded chain complex $\BNr(K)$ over the graded polynomial ring~$\Z[G]$
in one variable $G$.
The Rasmussen invariant is then determined by the \emph{free} part of $H_*(\BNr(K) \otimes_{\Z[G]} \mathbb{Q}[G])$,
whereas Alishahi--Dowlin's invariant is determined by the \emph{torsion} part of $H_*(\BNr(K) \otimes_{\Z[G]} \mathbb{Q}[G])$.

Just as the Rasmussen invariant was modelled on the Floer theoretic $\tau$-invariant~\cite{osz10,phdrasmussen},
so the Khovanov `torsion invariants' were inspired by Alishahi--Eftekhary's torsion invariants from knot Floer homology~\cite{zbMATH07262229,zbMATH07305772}.
One fascinating feature of torsion invariants in both the Khovanov and the Floer setting
is their versatility. They provide lower bounds for a multitude of different geometric quantities:
not just for the unknotting number and the Gordian distance,
but also for the number of maxima in orientable cobordisms (and thus for the bridge index)~\cite{zbMATH07311838,zbMATH07527797},
for the number of maxima in non-orientable cobordisms~\cite{zbMATH07786943},
for the ribbon distance~\cite{zbMATH07195384,gujral},
for the Turaev genus~\cite{zbMATH07608601},
and for the proper rational unknotting number~\cite{lambda1,2203.09319}.

In this paper, we extract from $\BNr$ a new lower bound $\lambdau(K,J)$ for the
\emph{Gordian distance} $u(K,J)$ between knots $K$, $J$, which is the minimal number of crossing changes required to transform $K$ into $J$.
\begin{restatable}{theorem}{thmlambdagordian}\label{thm:lambda-gordian}
	For any two knots \(K\) and~\(J\),
	\[
	\lambdau(K,J) \leq u(K,J).
	\]
\end{restatable}
As we shall see, $\lambdau$ subsumes both the `free' Rasmussen invariant and the `torsion' Alishahi--Dowlin invariant, i.e.\ $\lambdau(K,J)$ is greater than or equal to the lower bounds for $u(K,J)$ given by either of those two invariants. In fact, $\lambdau$ appears to be the strongest known lower bound for $u$ coming from Khovanov homology.
For example, $\lambdau(T(3,4), T(2,9)) = 2$ (see \cref{thm:pairs_torus_knots}), whereas the Rasmussen invariant and all other previously known lower bounds from Khovanov homology only yield ${u(T(3,4), T(2,9)) \geq 1}$.

The essential ingredient in the proof of \cref{thm:lambda-gordian}
is the construction of the following chain maps:
Given two knots $K_+$~and~$K_-$ related by a crossing change,
there are chain maps 
\[
\begin{tikzcd}
	\BNr(K_+)
	\arrow[bend left=5]{r}{f}
	\arrow[bend right=5,leftarrow,swap]{r}{g}
	&
	\BNr(K_-)
\end{tikzcd}
\]
such that either composition is homotopic to $G$ times the identity chain map,
i.e.~$g\circ f \simeq G \cdot \id_{\BNr(K_+)}$ and $f\circ g \simeq G \cdot \id_{\BNr(K_-)}$.
Similar constructions are at the heart of all of the Khovanov `torsion' lower bounds for the Gordian distance, starting with~\cite{zbMATH07178864,zbMATH07005602}.
However, previous work ignored the behaviour of $f$ and $g$ vis-à-vis the gradings.
We prove that those chain maps indeed preserve the homological grading and change the quantum grading by $-2$ and~$0$, respectively.
This motivates the definition of  $\lambdau(K, J)$ as
the minimal $n\geq 0$ such that
there are homogeneous chain maps 
\[
\begin{tikzcd}
	\BNr(K)
	\arrow[bend left=5]{r}{f}
	\arrow[bend right=5,leftarrow,swap]{r}{g}
	&
	\BNr(J)
\end{tikzcd}
\]
preserving the homological grading and not increasing the quantum grading,
with $g\circ f \simeq G^n \cdot \id_{\BNr(K)}$ and $f\circ g \simeq G^n \cdot \id_{\BNr(J)}$.

Dropping the condition that $f$ and $g$ do not increase the quantum grading
results in our definition of $\lambda(K, J)$. 
Naturally, this integer is less than or equal to~$\lambdau(K, J)$.
However, it is not just a lower bound for the Gordian distance~$u(K, J)$,
but for the \emph{proper rational Gordian distance~$u_q(K,J)$},
i.e.~the minimal number of proper rational tangle replacements relating $K$ and~$J$
(see \cref{def:proper_rational_gordian}).

\begin{restatable}{theorem}{thmlambdarational}\label{thm:lambda_proper_rational_gordian}
    For any two knots \(K\) and~\(J\),
    \[
    \lambda(K,J) \leq u_q(K,J).
    \]
\end{restatable}

We show that~$\lambda(K,U)$, where $U$ denotes the unknot, is bounded from below by half of the difference between the Rasmussen invariants $s_{\mathbb{F}}(K)$ over different fields~$\mathbb{F}$.
In particular, this implies the inequality
\[
\frac{|s_{\mathbb{F}}(K) - s_{\mathbb{F}'}(K)|}2 \leq u_q(K,U)
\]
for any fields \(\mathbb{F}\) and \(\mathbb{F}'\),
which was previously only known for the unknotting number $u(K,U)$ in lieu of the proper rational unknotting number~$u_q(K,U)$.

The new invariant $\lambda$ is greater than or equal to the invariant (also called $\lambda$) introduced by the first two authors and Iltgen~\cite{lambda1}, which had a similar definition, but without requiring the chain maps $f$ and $g$ to be homogeneous.
As a consequence of the preservation of the homological grading by the chain maps in the definition of the new $\lambda$, the difference of torsion orders \emph{in each individual homological degree} can be seen to be a lower bound for $\lambda$.

\begin{proposition}\label{prop:localtorsion}
For a knot \(K\) and an integer \(i\), let \(\mathfrak{u}_i(K)\) be the maximal \(G\)-torsion order
of \(\BNr(K)\) in homological grading \(i\). Then for all knots \(K\) and \(J\), 
\[
\max_{i\in\mathbb{Z}} |\mathfrak{u}_i(K) - \mathfrak{u}_i(J)| \leq \lambda(K,J).
\]
\end{proposition}
This lower bound may be strict.
For example, we find $\lambda(T (3, 7), T (2, 13)) = 2$,
whereas $\max_{i\in\mathbb{Z}} |\mathfrak{u}_i(T (3, 7)) - \mathfrak{u}_i(T (2, 13))| = 1$
(and $s_{\mathbb{F}}(T (2, 13)) = s_{\mathbb{F}}(T (3, 7))$ for all fields~$\mathbb{F}$).
Let us showcase an application of the bound.
Writing $c_+(D)$ for the number of positive crossings of a knot diagram $D$, we find the following.
\begin{restatable}{theorem}{thmtorusPRTR}\label{thm:torus_PRTR}
    Let \(K\) be a knot related by a proper rational tangle replacement to \(T(3,3n+1)\) or \(T(3,3n+2)\), for some \(n\geq 1\). 
    Then \(c_+(D)\geq 4n+1\) for any diagram \(D\) of \(K\).
    Moreover, the bound on $c_+(D)$ is optimal for all $n$.
\end{restatable}

\begin{figure}[t]
\includegraphics{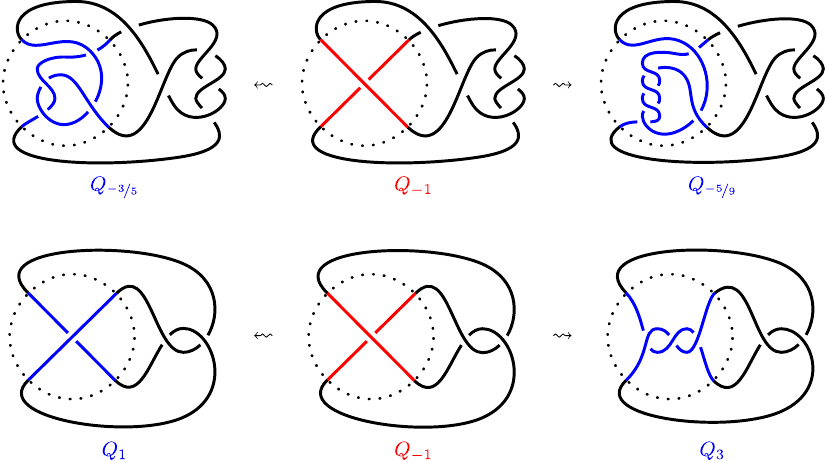}
\caption{Four examples of proper rational tangle replacements from the positive trefoil $T(2,3)$ (middle) to the unknot (left and right). The $0$-closures of the tangles replacing $Q_{-1}$ are
$Q_{\nicefrac{-3}{5}}(0) = -T(2,3)$,
$Q_{\nicefrac{-5}{9}}(0) = -T(2,5)$,
$Q_{1}(0) = U$, and
$Q_{3}(0) = -T(2,3)$.
\label{fig:LambdaExample}
}
\end{figure}
A more general variation $\Lambda$ of the invariants \(\lambda\) and \(\lambdau\) is obtained as follows: We may ask which pairs of integers are realized as the quantum gradings of chain maps \(f\) and \(g\) as above.  This leads to a set-valued invariant \(\Lambda(K,J)\subseteq\Z^2\), which subsumes both integer invariants \(\lambda(K,J)\) and \(\lambdau(K,J)\). 
The invariant \(\Lambda(K,J)\) allows for more refined statements about slopes in rational tangle replacements.
For example, the positive trefoil $T(2,3)$ may be transformed into the unknot
by a variety of different proper rational tangle replacements, four of which are shown in \cref{fig:LambdaExample}.
A glance at \cref{thm:obstruction_rational_repl_other_cr} and \(\Lambda(T(2,3),U)\)---see \cref{fig:Lambda_trefoil} on page \pageref{fig:Lambda_trefoil} for the latter---now reveals that if $T(2,3)$ is transformed into the unknot by the proper rational tangle replacement $Q_{-1} \leadsto Q_{\nicefrac{p}{q}}$, then the signature \(\sigma(Q_{\nicefrac{p}{q}}(0))\) of the $0$-closure of $Q_{\nicefrac{p}{q}}$ must equal $0$, $2$ or $4$. These three values of the signature are indeed realized, see \cref{fig:LambdaExample}.

\cref{thm:obstruction_rational_repl_other_cr} also implies a slightly stronger version of \cref{thm:lambda-gordian}, which we state in \cref{cor:strongerthm1}. Namely, in \cref{def:u0} we consider an intermediate Gordian distance \(u^0(K, J)\) that allows only certain proper rational replacements (among them crossing changes and Przytycki's \(\overline{t}_{2k}\)-moves~\cite{zbMATH04088518}), such that \(u_q(K, J) \leq u^0(K, J) \leq u(K, J)\). Then, $\lambda^0$ is a lower bound not just for $u$, but also for $u^0$.
We thus have the chain of inequalities
\[
\frac{|s_{\mathbb{F}}(K) - s_{\mathbb{F}}(J)|}2 \leq \lambdau(K,J)\leq u^0(K,J)\leq u(K, J),
\]
which in particular strengthens the previously known bound 
$|s_{\mathbb{F}}(K)|/2 \leq u(K)$.

\subsection*{Structure of the paper}
\cref{sec:prelim} introduces the necessary preliminaries
regarding the Bar-Natan homology $\BNr(K)$ of knots $K$, Bar-Natan's complexes for 4-ended tangles,
and more specifically Bar-Natan complexes of rational tangles.
In \cref{sec:small_lambda}, the new invariants $\lambda$ and $\lambdau$ are defined,
and their relationships with Gordian distances, with the Rasmussen invariant, and with torsion is analysed.
In particular, \cref{thm:lambda-gordian,thm:lambda_proper_rational_gordian} are proven.
\cref{sec:applications} contains applications to 2- and 3-stranded torus knots.
In \cref{sec:Lambda_set}, the invariant $\Lambda$ is defined and examined, which generalises $\lambda$ and $\lambdau$.
\cref{sec:appendix} establishes some required properties of the Bar-Natan complexes of rational tangles.

\subsection*{Acknowledgments}
The authors warmly thank Hoel Queffelec and Dirk Schütz for comments on drafts of this article and Peter Teichner for insightful remarks. 
We also thank Tuomas Kelomäki for pointing us to the computations in \cite{benheddi}. 
The authors are grateful to Peter Feller for inspiring conversations and funding an extended visit of L.M.~to the ETH Zurich, during which the paper advanced significantly, via the SNSF Grant~181199.
L.L.~gratefully acknowledges support by the DFG (German Research Foundation) via the Emmy Noether Programme, project no.~412851057.
L.M.~has received funding from the European Union's Horizon 2020 research and innovation programme under the Marie Skłodowska-Curie grant agreement No 945332 $\vc{\includegraphics[height=11pt]{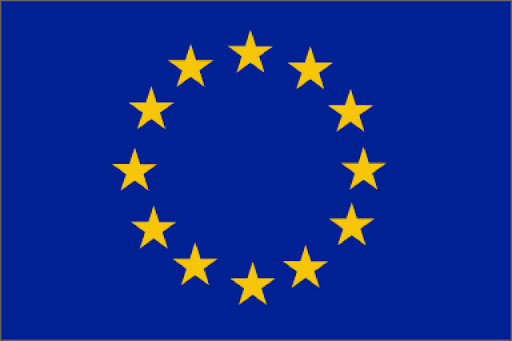}}$.
C.Z. is supported by an individual research grant (project no.~505125645) from the DFG and a Starting Grant (project no.~101115822) from the European Research Council (ERC). 

\section{A review of Bar-Natan homology}\label{sec:prelim}

For an arbitrary~$n\geq 1$, let $T$ be a $2n$-ended oriented tangle.
Bar-Natan \cite{BarNatanKhT} showed how to associate to a diagram $D_T$ of $T$ a bigraded chain complex~$\llbracket D_T \rrbracket$. This complex lives in an additive category $\text{Mat}(\Cob_{/l}(2n))$ generated by $2n$-ended crossingless tangle diagrams and cobordisms between them, modulo some relations, which are denoted by~$l$. The category $\text{Mat}(\Cob_{/l}(2n))$ carries a quantum grading~$\QGrad{q}$, defined by setting $\QGrad{q}(C)=\chi(C)-n$ for a cobordism~$C$.
Up to chain homotopy equivalence, $\llbracket D_T \rrbracket$ is an invariant of the tangle, and will therefore be denoted by~$\llbracket T \rrbracket$. 
Bar-Natan's construction is a key ingredient in the definition of two central objects of this paper, the 4-ended tangle invariant $\DD(T)$ and the invariant $\BNr(K)$ of knots~$K$.

\subsection{\texorpdfstring{$\Z[G]$-homology}{Z[G]-homology}}

There is a one-to-one correspondence between knots and 2-ended tangles without any closed components. By cutting a knot open at a point one obtains a 2-ended tangle and, vice-versa, a knot can be obtained from a 2-ended tangle without closed components by joining the two end-points with an arc.
Two-ended tangles are used in \cite[Section 2.3]{lambda1} to construct a universal variation of Khovanov homology, called $\Z[G]$-homology, for knots. Let us summarize this construction.
Let $G$ be a formal variable, and consider the ring~$\Z[G]$. We equip $\Z[G]$ with a \emph{quantum grading}~$\QGrad{q}$, determined by 
\[
\QGrad{q}(1)=0 \qquad \text{and} \qquad \QGrad{q}(G)=-2.
\]
Let $\mathcal{M}_{\Z[G]}$ be the category of graded, finitely generated, free $\Z[G]$-modules. 

\begin{proposition}[{\cite[Proposition 2.8]{lambda1}}]
\label{prop:equiv_cat_2-tangles}
    The functor 
    \[
    \mathcal{M}_{\Z[G]} \to \textnormal{Mat}(\Cob_{/l}(2))
    \]
    sending \(\Z[G]\) to the crossingless diagram \(\vcenter{\hbox{\def\svgwidth{7pt}\includegraphics[width=0.028\textwidth]{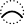}}}\) of the trivial 2-ended tangle and sending the linear map given by multiplication with~\(G^k\), for~\(k\geq 0\), to the connected cobordism of genus~\(k\), determines an equivalence of categories.
\end{proposition}

This equivalence allows to translate~$\llbracket T \rrbracket$, for any oriented 2-ended tangle~$T$, into a chain complex over~$\mathcal{M}_{\Z[G]}$. 
We denote this complex by~$\BNr(T)$. For an oriented 2-ended tangle $T$ with no closed components, let $K$ be the corresponding oriented knot.  Then $\BNr$ extends to a chain complex~$\BNr(K)$, which is a knot invariant up to chain homotopy equivalence. 
Given a graded module $M=\bigoplus_i M_i$ and an integer~$\ell$, we denote by $\QGrad{q^\ell}M$ the \emph{quantum grading shift by $\ell$}, i.e.~$(\QGrad{q^\ell}M)_i=M_{i-\ell}$. Similarly, $t^jM$ denotes the \emph{homological grading shift by $j\in\Z$}. 
As is common practice in Khovanov homology, we follow cohomological conventions: 
The \(t\)-grading increases by 1 along the differential. 
We say that a map $f\colon C \to C'$ between two chain complexes $(C,d)$ and $(C',d')$ is a \emph{chain map} if $f \circ d = d' \circ f$ and if $f$ preserves the homological grading\footnote{While this is the standard definition of chain map, it differs from the definition mostly used in \cite{lambda1}, where the chain maps were \emph{ungraded}, i.e.\ did not necessarily preserve the homological grading.}.

\begin{remark}\label{rem:universal}
    There is a second way to obtain the chain complex~$\BNr(K)$: This complex is chain homotopy equivalent to the reduced Khovanov chain complex associated to the Frobenius algebra $\Z[G][X]/(X^2+GX)$ over~$\Z[G]$. 

    It is stated in \cite{naot,naotphd} and shown explicitly in \cite[Theorem 2.17]{lambda1} that, for any rank 2 Frobenius system $\mathcal{F}=(R,A,\Delta,\varepsilon)$, the \emph{unreduced} Khovanov chain complex~$C^{\text{unred}}_{\mathcal{F}}$ corresponding to~$\mathcal{F}$ can be obtained from $\BNr$ by tensoring with the Frobenius algebra~$A$. 
    The correspondence between $\BNr$ and the reduced and unreduced Khovanov chain complexes is summarized in \cref{fig:chain_complexes_related}, when~$A=R[X]/(X^2)$, for $R=\Z$ or a field~$\F$.
In this paper, we almost exclusively use reduced chain complexes and homologies. So to simplify notation, we omit the tilde $\widetilde{\quad}$ that in the literature often signifies reducedness.
\end{remark}

\begin{figure}[t]
	\centering
	\[
	\begin{tikzcd}[column sep=huge, row sep=large]
		&
		&
		C_\Z(K)
		\arrow{r}{- \otimes_\Z \F}
		&
		C_\F(K)
		\\
		t^0\Z
		&
		\BNr({K;\Z[G]})
		\arrow[sloped]{l}{G=1}
		\arrow[sloped]{ru}{G=0}
		\arrow{r}{-\otimes_{\Z[G]}\F[G]}
		\arrow[swap, sloped]{rd}{-\otimes_{\Z[G]} A_\Z}
		&
		\BNr({K;\F[G]})
		\arrow[sloped]{ru}{G=0}
		\arrow{r}{G=1}
		\arrow[sloped]{rd}{-\otimes_{\F[G]} A_\F}
		&
		t^0\F
		\\
		&
		&
		C^{\text{unred}}_\Z(K)
		\arrow[swap]{r}{-\otimes_\Z \F}
		&
		C^{\text{unred}}_\F(K)
	\end{tikzcd}
	\]
	\caption{For any knot $K$ and field $\F$, this diagram summarizes, up to chain homotopy equivalence, the relations between $\BNr(K;\Z[G])$, $\BNr(K;\F[G])$, and the reduced and unreduced Khovanov chain complexes over $\Z$ and $\F$. Here, $A_R=\QGrad{q^{-1}}R \oplus \QGrad{q}R$, for $R=\Z$ or $\F$, and $A_R$ is a $R[G]$-module via $G$ acting as 
		$\bigl(\begin{smallmatrix} 0 & 2 \\ 0 & 0 \end{smallmatrix}\bigr)$.}
	\label{fig:chain_complexes_related}
\end{figure}

For any $\Z[G]$-module~$M$, we define a chain complex of $\Z[G]$-modules
\[
\BNr(K;M)\coloneqq\BNr(K)\otimes_{\Z[G]} M.
\]
This chain complex may inherit structure from $M$.
Namely,
$\BNr(K;M)$ is graded if $M$ is,
and for some commutative unital ring $R$,
$\BNr(K;M)$ is an $(R,\Z[G])$-bimodule if $M$ is.

\begin{remark}\label{rem:pawn_knight_pieces}
    When $\F$ is a field, the ring $\F[G]$ is a PID. 
    One can show that the complex $\BNr(K;\F[G])$ decomposes, up to chain homotopy, into a single grading-shifted copy of the base ring $t^0\QGrad{q^s}\F[G]$ in homological degree 0, and some summands of the form
    \[
    t^i\QGrad{q^\ell}\F[G] \overset{G^k}{\longrightarrow} t^{i+1}\QGrad{q^{2k+\ell}}\F[G]
    \]
    for $s,i,\ell\in \Z$ and~$k\in\Z_{>0}$. 
    The summand $t^0\QGrad{q^s}\F[G]$ is called the \emph{pawn piece} of the chain complex, and $s=s_{\F}(K)\in 2\Z$ is the Rasmussen invariant of $K$ over the field~$\F$. 
    We will refer to summands of the second kind as \emph{$G^k$-knight pieces}. 
    More details on the decomposition of $\BNr(K;\F[G])$ into pieces can be found in \cite[Section 3.3]{lambda1}.

    Consider the complex $\BNr(K;\F[G]/(G-1))=\BNr(K;\F[G]) \otimes_{\F[G]} \F[G]/(G-1)$. 
    The differential in each knight piece of $\BNr(K;\F[G])$ becomes an isomorphism in $\BNr(K;\F[G]/(G-1))$, and can be cancelled using Gaussian elimination. 
    Therefore, $\BNr(K;\F[G]/(G-1))\simeq t^0\F$.
\end{remark}

\subsection{Bar-Natan's invariant for Conway tangles}

Let $T$ be a \emph{Conway tangle}, i.e.\ a 4-ended tangle, with a \emph{basepoint}, i.e.\ a choice of a distinguished tangle end.
Using delooping, it was shown in \cite[Theorem 1.1]{KWZ} that $\llbracket T \rrbracket$ can be rewritten as a chain complex $\DD(T)$ over the quiver algebra:
\begin{equation}\label{eq:B_quiver}
\BNAlgH
\coloneqq
\Z\Big[
\begin{tikzcd}[row sep=2cm, column sep=1.5cm]
\bullet
\arrow[leftarrow,in=145, out=-145,looseness=5]{rl}[description]{D_{\bullet}}
\arrow[leftarrow,bend left]{r}[description]{S_{\circ}}
&
\circ
\arrow[leftarrow,bend left]{l}[description]{S_{\bullet}}
\arrow[leftarrow,in=35, out=-35,looseness=5]{rl}[description]{D_{\circ}}
\end{tikzcd}
\Big]\Big/\Big(
\parbox[c]{90pt}{\text{\footnotesize\centering
	$D_{\bullet} \cdot S_{\circ}=0=S_{\circ}\cdot D_{\circ}$} \\
	\text{\footnotesize\centering $D_{\circ}\cdot S_{\bullet}=0=S_{\bullet}\cdot D_{\bullet}$}
}\Big)
\end{equation}

As an algebra over $\Z$ (i.e.\ as a ring), $\BNAlgH$ consists of formal sums of paths in the quiver.
Multiplication is given by composing paths when possible and by 0 otherwise, modulo the relations given in \cref{eq:B_quiver}. We read elements of the algebra from right to left. Observe that the unit of the algebra is given by $1=\iota_\circ + \iota_\bullet$, where $\iota_\circ$ and $\iota_\bullet$ are the idempotents corresponding to the constant paths in the quiver. 
We set
\[
D\coloneqq D_{\circ}+D_{\bullet}
\quad\text{and}\quad
S\coloneqq S_{\circ}+S_{\bullet}.
\]
By setting
\[
    G=S_\circ S_\bullet+S_\bullet S_\circ-D_\bullet-D_\circ=S^2-D,
\]
the algebra $\BNAlgH$ can be considered as an algebra over the ring~$\Z[G]$.
The algebra \(\BNAlgH\) carries a quantum grading \(\QGrad{q}\), which is determined by 
\[
\QGrad{q}(D_{\bullet}) = \QGrad{q}(D_{\circ}) = -2
\qquad 
\text{and}
\qquad
\QGrad{q}(S_{\bullet}) = \QGrad{q}(S_{\circ}) = -1.
\] 
The complex $\DD(T)$ is a bigraded chain complex of right $\BNAlgH$-modules, and its differentials are homogeneous elements of~$\BNAlgH$. Each differential has homological degree 1 and quantum degree 0. Given~$i,\ell\in\Z$, we denote by $\GGzqh{\bullet}{0}{\ell}{i}$ (resp.\ by $\GGzqh{\circ}{0}{\ell}{i}$) the element $\bullet$ (resp.\ $\circ$) with homological grading shifted by $i$ and quantum grading shifted by~$\ell$.

Let us briefly describe the correspondence between $\DD(T)$ and Bar-Natan's tangle invariant~$\llbracket T \rrbracket$. It is based on an equivalence between the category $\text{Mat}(\Cob_{/l}(4))$ and a category $\mathcal{M}_{\BNAlgH}$ associated to the algebra~$\BNAlgH$.

The construction of the category $\mathcal{M}_{\BNAlgH}$ relies on the idempotents $\iota_\circ$ and~$\iota_\bullet$. The objects of the category are (sums of) the projective right $\BNAlgH$-modules $\iota_\circ\BNAlgH$ and~$\iota_\bullet\BNAlgH$, generated by the paths of the quiver that end at $\circ$ and $\bullet$, respectively. 
The morphisms of $\mathcal{M}_{\BNAlgH}$ are homomorphisms of right $\BNAlgH$-modules.
Consider an element~$\varphi\in \BNAlgH$. For any $x,y \in \{ \circ, \bullet \}$, $\varphi$ induces a morphism in~$\mathcal{M}_{\BNAlgH}$, denoted again by~$\varphi$, 
\[
\varphi \colon \iota_x\BNAlgH \to \iota_y\BNAlgH, \qquad p \mapsto \iota_y \cdot \varphi \cdot p \qquad \text{for } p\in \iota_x\BNAlgH.
\]

For instance, if $\varphi=S,\, x=\bullet$ and~$y=\circ$, one finds $\iota_\circ \cdot S \cdot p = S_\bullet \cdot p $ for all $p\in \iota_\bullet\BNAlgH$. Hence, as a morphism $\iota_\bullet\BNAlgH \to \iota_\circ\BNAlgH $ in~$\mathcal{M}_{\BNAlgH}$, one has~$S=S_\bullet$. On the other hand, one can easily check that $S=S_\circ$ as a morphism $\iota_\circ\BNAlgH \to \iota_\bullet\BNAlgH$ in~$\mathcal{M}_{\BNAlgH}$.
Composition in the category $\mathcal{M}_{\BNAlgH}$ is given by multiplication and, for~$x\in \{\circ,\bullet\}$,~$\id_x=\iota_x$.
For simplicity, we write $\circ=\iota_\circ\BNAlgH$ and~$\bullet=\iota_\bullet\BNAlgH$.

The following is a reformulation of \cite[Theorem 4.21]{KWZ}.

\begin{proposition}\label{prop:equiv_cat_4-tangles}
		For each choice of basepoint, 
		there exists an equivalence of categories
		\(\mathcal{M}_{\BNAlgH} \longrightarrow \textnormal{Mat}(\Cob_{/l}(4))\).
		If we choose the bottom left basepoint 
		\textnormal{\(\vcenter{\hbox{\def\svgwidth{12pt}%% Creator: Inkscape inkscape 0.92.5, www.inkscape.org
%% Accompanies image file 'basept.pdf' (pdf, eps, ps)
\begingroup%
  \makeatletter%
  \providecommand\color[2][]{%
    \errmessage{(Inkscape) Color is used for the text in Inkscape, but the package 'color.sty' is not loaded}%
    \renewcommand\color[2][]{}%
  }%
  \providecommand\transparent[1]{%
    \errmessage{(Inkscape) Transparency is used (non-zero) for the text in Inkscape, but the package 'transparent.sty' is not loaded}%
    \renewcommand\transparent[1]{}%
  }%
  \providecommand\rotatebox[2]{#2}%
  \newcommand*\fsize{\dimexpr\f@size pt\relax}%
  \newcommand*\lineheight[1]{\fontsize{\fsize}{#1\fsize}\selectfont}%
  \ifx\svgwidth\undefined%
    \setlength{\unitlength}{11.75406457bp}%
    \ifx\svgscale\undefined%
      \relax%
    \else%
      \setlength{\unitlength}{\unitlength * \real{\svgscale}}%
    \fi%
  \else%
    \setlength{\unitlength}{\svgwidth}%
  \fi%
  \global\let\svgwidth\undefined%
  \global\let\svgscale\undefined%
  \makeatother%
  \begin{picture}(1,1.09265902)%
    \lineheight{1}%
    \setlength\tabcolsep{0pt}%
    \put(0,0){\includegraphics[width=\unitlength,page=1]{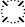}}%
    \put(0.11277773,-0.23927009){\color[rgb]{0,0,0}\rotatebox{13.15106325}{\makebox(0,0)[lt]{\lineheight{1.25}\smash{\begin{tabular}[t]{l}*\end{tabular}}}}}%
  \end{picture}%
\endgroup%
}}\)}, this equivalence is given by
    \begin{align*}
        \circ & \mapsto \vc{\rotatebox{90}{$\Lo$}}
        &
        S_\circ & \mapsto \vcenter{\hbox{\def\svgwidth{7pt}\includegraphics[width=0.05\textwidth]{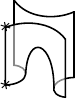}}}
        &
        G_\circ & \mapsto \vcenter{\hbox{\def\svgwidth{7pt}\includegraphics[width=0.054\textwidth]{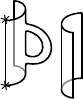}}}
        \\
        \bullet & \mapsto \vc{\rotatebox{90}{$\Li$}}
        &        
        S_\bullet & \mapsto \vcenter{\hbox{\def\svgwidth{7pt}\includegraphics[width=0.05\textwidth]{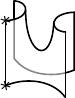}}} 
        &
        G_\bullet & \mapsto \vcenter{\hbox{\def\svgwidth{7pt}\includegraphics[width=0.057\textwidth]{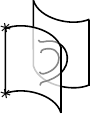}}}
    \end{align*}
    where \(G_\circ=S_\bullet S_\circ-D_\circ\) and~\(G_\bullet=S_\circ S_\bullet-D_\bullet\).
\qed
\end{proposition}

Throughout this paper, we will only use the equivalence corresponding to the choice of bottom left basepoint. 
Specifically, the distinguished tangle end in each Conway tangle is the bottom left tangle end. 

Below, we state some definitions and results about Conway tangles that will be used throughout the paper. In particular, \cref{thm:pairing_thm}, referred to as the \emph{Pairing Theorem}, gives a way of ``decomposing'' the chain complex of a knot $K$ into the chain complexes of two Conway tangles whose \emph{union} is~$K$.

\begin{definition}
    The \emph{connectivity} 
    \(\conn{T}\) of a Conway tangle \(T\) is an element of \(\{\Ni,\ConnectivityX,\No\}\) defined as follows:
	\[
	\conn{T}=
	\begin{cases*}
	\Ni & if \(T\) connects the tangle ends as in the tangle \(\Ni\) \\
	\ConnectivityX & if \(T\) connects the tangle ends as in the tangle \(\CrossingX\)\\
	\No & if \(T\) connects the tangle ends as in the tangle \(\No\)
	\end{cases*}
	\]
\end{definition}

\begin{definition}
    Given two Conway tangles \(T_1\) and \(T_2\), their \textit{union} \(T_1\cup T_2\) is the link defined by the diagram in \cref{fig:tangleunion}.

    \begin{figure}[!ht]
		\centering
		\(\union\)
		\caption{The union \(T_1\cup T_2\) of two Conway tangles \(T_1\) and \(T_2\).}\label{fig:tangleunion}
	\end{figure}
    
\end{definition}

Let $-C$ be the dual complex of a bigraded chain complex~$C$.
In practice, $-C$ may be obtained by reversing the direction of all differentials of~$C$, and by switching the sign of all the gradings. For a Conway tangle~$T$, let $-T$ denote its mirror image, i.e.\ the tangle obtained by switching over- and under-strands at all crossings in a diagram of~$T$. Similarly, we let $-K$ be the mirror image of a knot~$K$.

\begin{lemma}[{\cite[Proposition 4.26]{KWZ}}]
\label{lem:mirror_complex}
    For any Conway tangle \(T\) and knot~\(K\), 
    \[
    \DD(-T) \simeq -\DD(T) \qquad \text{and} \qquad \BNr(-K) \simeq -\BNr(K).
    \]
\end{lemma}

Given two bigraded chain complexes $X$ and $Y$ of right $\mathcal{B}$-modules,
we denote by $\Mor(X,Y)$ the graded $Z(\mathcal{B})$-module
of $\mathcal{B}$-homomorphisms $f\colon X\to Y$,
which are neither required to be homogeneous, nor to be chain maps.
Here, $Z(\mathcal{B})$ denotes the center of $\mathcal{B}$,
which is the commutative graded subring of $\mathcal{B}$ generated (as a ring) by
$D_\circ, D_\bullet, S$. Note that $\Z[G] \subset Z(\mathcal{B})$.
We introduce a chain complex structure on $\Mor(X,Y)$ as follows. 
The $k^\text{th}$ chain module of $\Mor(X,Y)$ is given by 
\[
\bigoplus_{j-i=k} \Mor(X_i,Y_j),
\]
and the differential sends $f\in  \Mor(X_i,Y_j)$ to 
\[
(d_{Y} \circ f, (-1)^{k+1}f \circ d_{X}) \in \Mor(X_i,Y_{j+1})\oplus \Mor(X_{i-1},Y_j).
\]
This construction is also known as the \emph{internal hom of chain complexes}.
Note that the $0$-cycles of $\Mor(X,Y)$ are precisely the chain maps $X\to Y$,
the $0$-boundaries are precisely the null-homotopic chain maps,
and so the graded $\Z[G]$-module of homotopy classes of chain maps $X\to Y$
is isomorphic to $H_0(\Mor(X,Y))$.

\begin{theorem}[{\cite[Proposition 4.31]{KWZ}}]
	\label{thm:pairing_thm}
	For any splitting \(T_1\cup T_2\) of an oriented knot into two oriented Conway tangles \(T_1\) and \(T_2\),
	\[
	\QGrad{q^{-1}}\BNr(T_1\cup T_2) \simeq \Mor(-\DD(T_1),\DD(T_2))
	\]
	as chain complexes of \(\Z[G]\)-modules.
\end{theorem}

\subsection{The chain complex of rational tangles}

When $T$ is a rational tangle, it was shown in \cite{thompson, lambda1} that the complex $\DD(T)$ has a particularly simple structure: It can be expressed in terms of a \emph{zigzag complex} and computed recursively. We outline its structure here. For the recursive construction of the chain complex and the proofs of \cref{lem:end_zz,lem:grading_end_zz} we refer the reader to the appendix.

\begin{definition}
    A 4-ended tangle \(T\) in a 3-ball \(B\) is called \emph{rational} if 
    the pair \((B,T)\) is homeomorphic to \((D^2\times [0,1], \{(-\tfrac12,0),(\tfrac12,0)\}\times [0,1])\),
    drawn in \cref{fig:trivrational}. The homeomorphism need not fix the boundary pointwise.
    
    \begin{figure}[ht]
    \centering
    \includegraphics[width=0.15\textwidth]{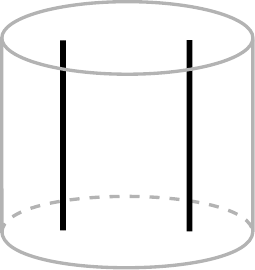}
    \caption{The rational tangle $(D^2\times [0,1], \{(-\tfrac12,0),(\tfrac12,0)\}\times [0,1])$.}
        \label{fig:trivrational}
    \end{figure}
\end{definition}

Using Conway's one-to-one correspondence between (unoriented) rational tangles and rational numbers, we denote the rational tangle corresponding to $r\in \Q\cup \{\infty\}$ by~$Q_r$. Any tangle obtained by adding an orientation to $Q_r$ will again be called~$Q_r$. 
As an example, $Q_0$ denotes the tangle~$\No$, while the negative crossing $\textnormal{$\vcenter{\hbox{\def\svgwidth{11.6pt}%% Creator: Inkscape inkscape 0.92.5, www.inkscape.org
%% Accompanies image file 'rational1.pdf' (pdf, eps, ps)
\begingroup%
  \makeatletter%
  \providecommand\color[2][]{%
    \errmessage{(Inkscape) Color is used for the text in Inkscape, but the package 'color.sty' is not loaded}%
    \renewcommand\color[2][]{}%
  }%
  \providecommand\transparent[1]{%
    \errmessage{(Inkscape) Transparency is used (non-zero) for the text in Inkscape, but the package 'transparent.sty' is not loaded}%
    \renewcommand\transparent[1]{}%
  }%
  \providecommand\rotatebox[2]{#2}%
  \newcommand*\fsize{\dimexpr\f@size pt\relax}%
  \newcommand*\lineheight[1]{\fontsize{\fsize}{#1\fsize}\selectfont}%
  \ifx\svgwidth\undefined%
    \setlength{\unitlength}{11.61026645bp}%
    \ifx\svgscale\undefined%
      \relax%
    \else%
      \setlength{\unitlength}{\unitlength * \real{\svgscale}}%
    \fi%
  \else%
    \setlength{\unitlength}{\svgwidth}%
  \fi%
  \global\let\svgwidth\undefined%
  \global\let\svgscale\undefined%
  \makeatother%
  \begin{picture}(1,0.99940785)%
    \lineheight{1}%
    \setlength\tabcolsep{0pt}%
    \put(0,0){\includegraphics[width=\unitlength,page=1]{rational1.pdf}}%
  \end{picture}%
\endgroup%
}}$}$ and the positive crossing $\textnormal{$\vcenter{\hbox{\def\svgwidth{11.6pt}%% Creator: Inkscape inkscape 0.92.5, www.inkscape.org
%% Accompanies image file 'rational-1.pdf' (pdf, eps, ps)
\begingroup%
  \makeatletter%
  \providecommand\color[2][]{%
    \errmessage{(Inkscape) Color is used for the text in Inkscape, but the package 'color.sty' is not loaded}%
    \renewcommand\color[2][]{}%
  }%
  \providecommand\transparent[1]{%
    \errmessage{(Inkscape) Transparency is used (non-zero) for the text in Inkscape, but the package 'transparent.sty' is not loaded}%
    \renewcommand\transparent[1]{}%
  }%
  \providecommand\rotatebox[2]{#2}%
  \newcommand*\fsize{\dimexpr\f@size pt\relax}%
  \newcommand*\lineheight[1]{\fontsize{\fsize}{#1\fsize}\selectfont}%
  \ifx\svgwidth\undefined%
    \setlength{\unitlength}{11.61026587bp}%
    \ifx\svgscale\undefined%
      \relax%
    \else%
      \setlength{\unitlength}{\unitlength * \real{\svgscale}}%
    \fi%
  \else%
    \setlength{\unitlength}{\svgwidth}%
  \fi%
  \global\let\svgwidth\undefined%
  \global\let\svgscale\undefined%
  \makeatother%
  \begin{picture}(1,0.99940789)%
    \lineheight{1}%
    \setlength\tabcolsep{0pt}%
    \put(0,0){\includegraphics[width=\unitlength,page=1]{rational-1.pdf}}%
  \end{picture}%
\endgroup%
}}$}$ represent oriented tangles corresponding to $Q_1$ and $Q_{-1}$ respectively.

By Conway's correspondence, the tangles $Q_{-r}$ and $-Q_r$ are isotopic, for all rational numbers~$r$. It follows that, when~$r\in\Q_{<0}$, the complex $\DD(Q_{r})$ can be obtained from $\DD(Q_{-r})$ by using \cref{lem:mirror_complex}: 
\[
\DD(Q_{r}) \simeq \DD(-Q_{-r}) \simeq -\DD(Q_{-r}).
\]
Therefore, we will limit our discussion to positive rational numbers~$r$.

\begin{definition}\label{def:zigzag}
A \emph{zigzag complex} is a graded chain complex $(\bigoplus_{i=0}^{n} A_i,\sum_{i=1}^{n} d_i)$ over~$\BNAlgH$
satisfying the following:
\begin{enumerate}[label=(\roman*)]
\item Each $A_i$ is either $\circ$ or~$\bullet$, with some homological and quantum shifts. We call $A_0$ and $A_n$ the \emph{ends} of the complex.
\label{eq:2objects}
\item Each $d_i$ has domain and target $A_{i-1}\to A_i$ or~$A_i\to A_{i-1}$,
and is one of the following five maps:

\[
S_\circ,\quad S_\circ S_\bullet, \quad S_\bullet S_\circ, \quad D_\circ, \quad D_\bullet.
\]
\label{eq:lineshapecomplex}
\item There is a partition of differentials into \emph{odd} and \emph{even}, such that if two differentials are consecutive (no matter what their domain and target are), then one of them is odd and the other one even. All differentials $S_\circ,S_\circ S_\bullet$ and $S_\bullet S_\circ$ are odd and all differentials $D_\circ$ and $D_\bullet$ are even. 
We call an end of the complex \emph{even} (resp.\ \emph{odd}) if the differential adjacent to it is even (resp.\ odd).
\label{eq:consecutived}
\item There is at least one differential~$S_\circ$.
\label{eq:saddleexists}
\end{enumerate}
\end{definition}

\begin{figure}[tb]
\[
\begin{tikzcd}[column sep=11ex,row sep=-.5ex]
  & & A_0 = t^2\QGrad{q^6}\circ \ar[rd,"d_0 = S",near start]      \\
  & &  & A_1 = t^3\QGrad{q^7}\bullet \\
  & & A_2 = t^2\QGrad{q^5}\bullet \ar[ru,"d_1 = D",swap,near start] \\
  & A_3 = t^1\QGrad{q^4}\circ \ar[ru,"d_2 = S",swap,near start]  \\
A_4 = t^0\QGrad{q^2}\circ \ar[ru,"d_3 = D",swap,near start] \\
\end{tikzcd}
\]
\caption{An example of a zigzag complex, which is homotopy equivalent to $\DD(Q_{\nicefrac{3}{2}})$.}
\label{fig:zigzag_example}
\end{figure}

See \cref{fig:zigzag_example} for an example of a zigzag complex.
Observe that, by the definition above, zigzag complexes do not contain any differential $S_\bullet$ nor (more generally) any differential oriented as~$\bullet \to \circ$.

\begin{remark}
	\label{rem:zigzag-differentials}
	Any zigzag complex is completely determined 
	by its underlying chain module and the directions of the differentials. 
	Indeed, any arrow \(\circ \to \bullet\) is labelled by \(S_\circ\) and must therefore be odd. 
	The existence of at least one such arrow determines the partition of the differentials into odd and even differentials. 
	This fixes the potential ambiguity of the labelling of 
	arrows 
	\(\circ \to \circ\) by \(D_\circ\)/\(S_\bullet S_\circ\) and 
	arrows
	\(\bullet \to \bullet\) by \(D_\bullet\)/\(S_\circ S_\bullet\). 
\end{remark}

\begin{proposition}[{\cite[Theorem 5.6]{lambda1}, \cite{thompson}, cf.\ \cref{prop:zigzag_equivalent}}]
   For~\(r\in\Q_{>0}\),\linebreak[3] let \(Q_r\) be an oriented rational tangle. Up to chain homotopy equivalence, the complex\linebreak[3] \(\DD(Q_r)\) has the structure of a zigzag complex.
\end{proposition}

The following lemma is well-known (see e.g.~\cite{zbMATH03960559,KWZ_linear}) and can easily be checked inductively.

\begin{lemma}\label{lem:parity}
    Let \(p\) and \(q\) be coprime integers. The connectivity of \(Q_{\nicefrac{p}{q}}\) is determined as follows:
    \[
        \conn{Q_{\nicefrac{p}{q}}}=
        \begin{cases*}
            \ConnectivityX & if $p$ and $q$ are odd, \\
            \No & if $p$ is even and $q$ is odd, \\
            \Ni & if $p$ is odd and $q$ is even.
        \end{cases*}
    \]%

\vspace{-1.15\baselineskip}
\mbox{}\qed
\end{lemma}

\begin{lemma}[{\cite[Lemma 5.12]{lambda1}}]
\label{lem:goodzigzag}
    Let \(r = \nicefrac{p}{q}\) with \(p,q\) positive coprime integers. Then the following hold.
    \begin{enumerate}[label=(\roman*)]
        \item \(\DD(Q_r)\) has an even end if and only if \(p\) or \(q\) is even.
        \item \(\DD(Q_r)\) has an odd \(\circ\)-end if and only if \(p\) is odd.
        \item \(\DD(Q_r)\) has an odd \(\bullet\)-end if and only if \(q\) is odd.
    \end{enumerate}
\end{lemma}

\begin{lemma}[\cref{lem:end_zzAppendix}]\label{lem:end_zz}
    Let \(p\) and \(q\) be odd coprime integers such that~\(\nicefrac{p}{q}>0\). By \cref{lem:goodzigzag}, the chain complex \(\DD(Q_{\nicefrac{p}{q}})\) has an odd \(\circ\)-end and an odd \(\bullet\)-end. In a neighborhood of the \(\bullet\)-end, the complex has the following shape (where we ignore the quantum and homological gradings):
    \begin{equation}\label{eq:ends_zz}
    \DD(Q_{\nicefrac{p}{q}})\simeq
    \begin{cases*}
        \cdots \; \circ \overset{D}{\longrightarrow} \circ \overset{S}{\longrightarrow} \bullet 
        & if 
        \(\nicefrac{p}{q} > 1\),
        \\
        \qquad\qquad\, \circ \overset{S}{\longrightarrow} \bullet 
        & 
        if \(\nicefrac{p}{q} = 1\),
        \\
        \cdots \; \circ \overset{D}{\longleftarrow} \circ \overset{S}{\longrightarrow} \bullet 
        & 
        if \(\nicefrac{1}{2} < \nicefrac{p}{q} < 1\),
        \\
        \cdots \; \bullet \overset{D}{\longrightarrow} \bullet \overset{S^2}{\longrightarrow} \bullet 
        & 
        if \(0 < \nicefrac{p}{q} < \nicefrac{1}{2}\).
    \end{cases*}
    \end{equation}
\end{lemma}

Given a Conway tangle~\(T\), the \emph{0-closure} of \(T\) is defined to be the link 
\[
T(0)=Q_{0}\cup T.
\]
For any tangle \(T\) without closed components, 
\(T\) has connectivity \(\ConnectivityX\) or \(\Ni\) 
if and only if 
the 0-closure $T(0)$ is a knot.

The bigrading of the complexes in \cref{eq:ends_zz} is determined by the following lemma.

\begin{lemma}[\cref{lem:grading_end_zzAppendix}]\label{lem:grading_end_zz}
	Let \(p\) and \(q\) be odd coprime integers. 
	We equip the rational tangle \(Q_{\nicefrac{p}{q}}\) with the same orientation at the tangle ends as the tangle \textnormal{$\vcenter{\hbox{\def\svgwidth{11.6pt}}}$}.
	(This is possible by \cref{lem:parity}.)
	Then the odd \(\bullet\)-end of \(\DD(Q_{\nicefrac{p}{q}})\) has bigrading~\(\GGzqh{\bullet}{0}{\ell}{0}\), where
	\[
	\ell = \begin{cases*}
		s_\Q(Q_{\nicefrac{p}{q}}(0))-1 
		& 
		if \(\nicefrac{p}{q}>0\),
		\\
		s_\Q(Q_{\nicefrac{p}{q}}(0))+1 
		& 
		if \(\nicefrac{p}{q}<0\).
	\end{cases*}
	\]
\end{lemma}

Observe that \cref{lem:grading_end_zz} determines the gradings of the complex $\DD(Q_{\nicefrac{p}{q}})$ completely: The degree shifts of all modules of the complex are uniquely given by the fact that the differentials must preserve the quantum grading and increase the homological grading by~1.
\begin{remark}\label{rem:s_of_2-bridge_knots}
	The expression \(s_{\Q}(Q_{\nicefrac{p}{q}}(0))\) in the formula for
        \(\ell\) in \cref{lem:grading_end_zz}
        can be replaced by $s_{\F}(Q_{\nicefrac{p}{q}}(0))$ over any field~$\F$. 
	In fact, 
	\(s_{\Q}(Q_{\nicefrac{p}{q}}(0))\) agrees with 	$-\sigma(Q_{\nicefrac{p}{q}}(0))$, since the knots $Q_{\nicefrac{p}{q}}(0)$ are alternating \cite{lee2002supportkhovanovsinvariantsalternating}.
	Below is a formula  for the signature $\sigma$ of 2-bridge knots $Q_{\nicefrac{p}{q}}(0)$ when $p$ and $q$ are coprime odd integers and $p > 0$:%
	\[
	s_{\Q}(Q_{\nicefrac{p}{q}}(0))
	=
	s_\F(Q_{\nicefrac{p}{q}}(0)) 
	= 
	-\sigma(Q_{\nicefrac{p}{q}}(0)) 
	= 
	-\sum_{i=1}^{p-1}(-1)^{\lfloor \nicefrac{iq}{p}\rfloor}.
	\]
	This formula is taken from \cite[Theorem 9.4]{murasugi}, which Murasugi attributes to Shinohara \cite{MR0563649}.%
	\footnote{%
		Our knot $Q_{\nicefrac{p}{q}}(0)$ corresponds to 
		the mirror of $B(p,q)$ in the notation of \cite{murasugi}. 
		Murasugi uses the convention that the signature of positive knots is positive,
		whereas we use the opposite convention.
		Moreover, he states the signature formula only for $p > q > 0$. 
		However, from that narrower case, 
		it may be deduced that the hypothesis $p > 0$ suffices, 
		by using that 
		$Q_{\nicefrac{p}{q}}(0) = Q_{\nicefrac{p}{q+2p}}(0)$, 
		$\sigma(Q_{\nicefrac{-p}{q}}(0)) = - \sigma(Q_{\nicefrac{p}{q}}(0))$, and 
		$\lfloor \nicefrac{iq}{p}\rfloor + \lfloor \nicefrac{-iq}{p}\rfloor = - 1$.
	}
	Alternative formulas, 
	based on continued fractions of~$\nicefrac{p}{q}$, 
	have been given by Murasugi \cite{Murasugi_2bridge} 
	and 
	Gallaspy--Jabuka~\cite{zbMATH06504434}.
\end{remark}

\section{\texorpdfstring{The invariants $\lambda$ and $\lambdau$}{The invariants λ and λ⁻}}\label{sec:small_lambda}

We start this section with the definition of our protagonists $\lambda$ and $\lambdau$.
Then, we prove \cref{thm:lambda-gordian,thm:lambda_proper_rational_gordian}
about $\lambda, \lambdau$, and Gordian distances.
We continue with a brief look at how $\lambda$ and $\lambdau$ change under non-rational proper tangle replacements, and then analyze their relationships with the Rasmussen invariants and with torsion.

\subsection{Definitions and basic properties}

Let $R$ denote a unital commutative ring. 

\begin{definition}\label{def:lambda}
	Given two knots \(K_1\) and \(K_2\), we define

    \[
    \lambda(K_1,K_2;R)
	=
	\min
	\left\{%
	n\in\Z_{\geq0}
	\left|\,
	\begin{minipage}{7cm}
    \centering
	\(
	\exists
    \)
    homogeneous chain maps
    \\
    \(
	\begin{tikzcd}
	\BNr(K_1;R[G])
	\arrow[bend left=5]{r}{f}
	\arrow[bend right=5,leftarrow,swap]{r}{g}
	&
	\BNr(K_2;R[G]):
	\end{tikzcd} 
    \)
    \\
    \(
	f\circ g\simeq G^n\cdot \id, \,
	g\circ f\simeq G^n\cdot \id
	\)
	\end{minipage}
	\right.
	\right\}
	\]
  and
	\[
	\lambdau(K_1,K_2;R)
	=
	\min
	\left\{%
	n\in\Z_{\geq0}
	\left|\,
    \begin{minipage}{7cm}
    \centering
	\(
	\exists
    \)
    homogeneous chain maps
    \\
    \(
	\begin{tikzcd}
	\BNr(K_1;R[G])
	\arrow[bend left=5]{r}{f}
	\arrow[bend right=5,leftarrow,swap]{r}{g}
	&
	\BNr(K_2;R[G]):
	\end{tikzcd} 
    \)
    \\
    \(
	f\circ g\simeq G^n\cdot \id, \,
	g\circ f\simeq G^n\cdot \id
	\)
    \\
    and $\QGrad{q}(f),\QGrad{q}(g)\leq 0$
	\end{minipage}
	\right.
	\right\}
	\]
	By convention, we set~$\min \varnothing = \infty$.
\end{definition}
Similar definitions appear in 
\cite[Def.~3.1]{zbMATH07305772},
\cite[Def.~7.1]{zbMATH07311838},
and \cite[Def.~1.7]{lambda1}.
When~$R=\Z$, we write $\lambda(K_1,K_2) = \lambda(K_1,K_2;\Z)$ and $\lambdau(K_1,K_2) = \lambdau(K_1,K_2;\Z)$. Moreover, for $K$ a knot and $U$ the unknot, we set 
\[
\lambda(K;R)=\lambda(K,U;R) \qquad \text{and} \qquad \lambdau(K;R)=\lambdau(K,U;R).
\]

\begin{remark}
    Observe that for any knots $K_1,\, K_2$ and any ring $R$
    \[
    \lambda(K_1,K_2;R) \leq \lambda(K_1,K_2) 
    \qquad \text{and} \qquad
    \lambda(K_1,K_2;R) \leq \lambdau(K_1,K_2;R) \leq \lambdau(K_1,K_2).
    \]
\end{remark}

\begin{remark}
Instead of working with the $\Z[G]$-chain complexes $\BNr(K)$ and homotopy classes of chain maps between them, as we do here, one could try to extract information from the homology $H_*(\BNr(K))$. 
However, this would be less practical than it might seem,
because $H_*(\BNr(K))$ is typically a non-free $\Z[G]$-module, and such modules over a non-PID such as $\Z[G]$ are not classified. More to the point, the isomorphism class of $H_*(\BNr(K))$ contains strictly less information than the homotopy class of $\BNr(K)$. Indeed, here is an example of two non-homotopy-equivalent $\Z[G]$-chain complexes $C_1, C_2$ with isomorphic graded homology modules.
\[
\begin{tikzcd}[column sep=7em,ampersand replacement=\&]
C_1 = \&[-7em]\&
(\QGrad{q^0} + \QGrad{q^{-2}})\Z[G]
\ar{r}{\begin{pmatrix}2&G\end{pmatrix}}
\& \QGrad{q^0}\Z[G] 
\\
C_2 = \&[-7em]
  \QGrad{q^{-2}}\Z[G] \ar{r}{\begin{pmatrix}G&-2&0\end{pmatrix}^{\top}}
\& (\QGrad{q^0}+\QGrad{q^{-2}}+\QGrad{q^{-2}})\Z[G] \ar{r}{\begin{pmatrix}2&G&0\end{pmatrix}}
\& \QGrad{q^0}\Z[G]
\end{tikzcd}
\]
In particular, for knots $K_i$ with $\BNr(K_i) \simeq C_i$, one has $\lambda(K_1, K_2) > 0$, but this would not be visible from homology.
\end{remark}

Below are some properties of the invariants $\lambda$ and~$\lambdau$. These properties show that both $\lambda$ and $\lambdau$ are pseudometrics on the set of isotopy classes of knots.

\begin{proposition}\label{prop:properties_lambda}
	Given knots \(K_1\), \(K_2\), and~\(K_3\), the following is true:
	\begin{enumerate}[label=(\roman*)]
		\item \label{item:properties_lambda:1} \(\lambda(K_1,K_2;R)=\lambda(K_2,K_1;R)\) and \(\lambdau(K_1,K_2;R)=\lambdau(K_2,K_1;R)\).
		\item \label{item:properties_lambda:2} \(\lambda(K_1,K_3;R)\leq \lambda(K_1,K_2;R) + \lambda(K_2,K_3;R)\), and\\
                      \(\lambdau(K_1,K_3;R)\leq \lambdau(K_1,K_2;R) + \lambdau(K_2,K_3;R)\).
		\item \label{item:properties_lambda:3} \(\lambda(K_1,K_2;R)=0\) if and only if \(\BNr(K_1;R[G]) \simeq \QGrad{q^\ell} \BNr(K_2;R[G])\), for some~\(\ell \in \Z\).
		\item \label{item:properties_lambda:4} \(\lambdau(K_1,K_2;R)=0\) if and only if \(\BNr(K_1;R[G]) \simeq \BNr(K_2;R[G])\).
	\end{enumerate}
\end{proposition}

\begin{proof}
  Property \cref{item:properties_lambda:1} follows directly from the definition of $\lambda$ and~$\lambdau$.
    
  For point \cref{item:properties_lambda:2}, 
  if \(\lambda(K_1,K_2;R)=\infty\) or \(\lambda(K_2,K_3;R)=\infty\), 
  then there is nothing to show. 
	So suppose \(\lambda(K_1,K_2;R)\) and \(\lambda(K_2,K_3;R)\) are finite. 
	Then there exist 
	maps \(f\) and \(g\) realizing \(\lambda(K_1,K_2;R)\) and 
	maps \(f'\) and \(g'\) realizing \(\lambda(K_2,K_3;R)\). 
	Clearly, \(f'\circ f\) and \(g\circ g'\) give rise to an upper bound of 
	\(\lambda(K_1,K_2;R)
	+
	\lambda(K_2,K_3;R)\) 
	for \(\lambda(K_1,K_3;R)\). The analogous argument applies to $\lambdau$.

    For property \cref{item:properties_lambda:3}, the implication $\Leftarrow$ is obvious, since $\lambda(K_1,K_2;R)=0$ can be realised by the maps $f=\id$ and~$g=\id$. For the other implication, $\lambda(K_1,K_2;R)=0$ implies that there exist homogeneous chain maps \(f\) and \(g\) with 
	\(f\circ g\simeq \id\) and \(g\circ f\simeq \id\). This, in turn, means that $\BNr(K_1;R[G])$ and $\BNr(K_2;R[G])$ are chain homotopy equivalent up to a shift in quantum gradings.

    The implication $\Leftarrow$ of \cref{item:properties_lambda:4} is obvious. The other direction can be seen as follows. 
	Suppose 
	\(
	\lambdau(K_1,K_2;R)=0
	\). 
	Then there exist homogeneous chain maps \(f\) and \(g\) with 
	\(f\circ g\simeq \id\), \(g\circ f\simeq \id\) and \(\QGrad{q}(f),\QGrad{q}(g)\leq0\). 
	It suffices to show that, in fact, \(\QGrad{q}(f)=\QGrad{q}(g)=0\). 
	Since neither \(\BNr(K_1;R[G])\) nor \(\BNr(K_2;R[G])\) are contractible, we know that \(f\circ g\neq0\). Hence \(\QGrad{q}(f)+\QGrad{q}(g) = \QGrad{q}(f\circ g)=\QGrad{q}(\id)=0\), which indeed implies \(\QGrad{q}(f)=\QGrad{q}(g)=0\).
\end{proof}

\subsection{Relationship with Gordian distances}

We now relate the invariants $\lambda$ and $\lambdau$ to two topological quantities: the \emph{Gordian distance} and the \emph{proper rational Gordian distance}.

\begin{definition} \label{def:gordian}
	The \emph{Gordian distance} $u(K_1,K_2)$ between two knots $K_1,\, K_2$ is defined as the minimal number of crossing changes relating $K_1$ to~$K_2$.
\end{definition}

\thmlambdagordian*

\begin{remark}
	In general, \cref{thm:lambda-gordian} does not hold for links with two or more components. 
	For instance, one may check that $\lambdau(L,U_2)=\infty$ if $L$ is the Hopf link and $U_2$ the two component unlink. 
	This follows from the fact that both \(\BNr(L;\Z[G])\) and \(\BNr(U_2;\Z[G])\) are free of rank 2, but the former is supported in distinct homological gradings whereas the latter is supported in just a single homological grading. 
\end{remark}

\begin{definition} \label{def:proper_rational_gordian}
    Two knots $K_1$ and $K_2$ are related by a  \emph{proper rational tangle replacement} if $K_2$ can be obtained from $K_1$ by replacing a rational tangle in $K_1$ by another rational tangle with the same connectivity. 
    The \emph{proper rational Gordian distance} or \emph{$u_q$-distance} $u_q(K_1,K_2)$ between two knots $K_1,\, K_2$ is defined as the minimal number of proper rational replacements relating $K_1$ to~$K_2$.
\end{definition}

\thmlambdarational*

\begin{remark}
    In general, $\lambdau$ is not a lower bound for~$u_q$;  counterexamples are given by $(2,n)$-torus knots. 
\end{remark}

\begin{remark}\label{remark:slopechange}
Any proper rational tangle replacement $Q_{\nicefrac{a}{b}} \rightsquigarrow Q_{\nicefrac{c}{d}}$
may be achieved by making instead a proper rational tangle replacement $Q_{-1} \rightsquigarrow Q_{\nicefrac{p}{q}}$.
To see this, pick a non-boundary-fixing homeomorphism $\varphi$ of the ambient 3-balls sending $Q_{\nicefrac{a}{b}}$ to $Q_{-1}$, and setting $Q_{\nicefrac{p}{q}} = \varphi(Q_{\nicefrac{c}{d}})$. That such a $\varphi$ exists follows immediately from the definition of rational tangle.
Properness then forces $p$ and $q$ to be odd (see \cref{lem:parity}).
If the rational tangles $Q_{\nicefrac{a}{b}}, Q_{\nicefrac{c}{d}}$ carry the same orientation,
then $\varphi$ may be chosen to preserve this orientation. Consequently,
every orientation-preserving rational tangle replacement may be achieved by making instead a replacement 
$\textnormal{$\vcenter{\hbox{\def\svgwidth{11.6pt}}}$} \rightsquigarrow Q_{\nicefrac{p}{q}}$.
In what follows, we will often focus on replacements $\textnormal{$\vcenter{\hbox{\def\svgwidth{11.6pt}}}$} \rightsquigarrow Q_{\nicefrac{p}{q}}$ without loss of generality.

Taking this a little further, one may after $\varphi$ apply another non-boundary fixing homeomorphism $\psi$ such that $\psi(Q_{-1}) = Q_{-1}$. A simple calculation in \cite[Lemma 5.15]{lambda1} shows that there exists such $\psi$ such that $\psi(Q_{\nicefrac{p}{q}}) = Q_{\nicefrac{p'}{q'}}$ with $\nicefrac{p'}{q'} > 0$. Hence w.l.o.g.~one may assume $\nicefrac{p}{q} > 0$. A similar calculation, the details of which we omit, shows that if the orientation of the tangles shall be preserved by $\psi$, one may still w.l.o.g.~assume $\nicefrac{p}{q} > -\nicefrac{1}{2}$. 
\end{remark}

In order to prove \cref{thm:lambda_proper_rational_gordian,thm:lambda-gordian}, 
we extend $\lambda$ and $\lambdau$ to Conway tangles.

\begin{definition}\label{def:lambda_tangle}
	Given two Conway tangles \(T_1\) and \(T_2\) and \(a\in\{D_\bullet,S^2,D_\circ,G\}\subset\mathcal{B}\), define

    \[
    \lambda_{a}(T_1,T_2)
	=
	\min
	\left\{%
	n\in\Z_{\geq0}
	\left|\,
    \begin{minipage}{5cm}
    \centering
	\(
	\exists
    \)
    homogeneous chain maps
    \\
    \(
	\begin{tikzcd}
	\DD(T_1)
	\arrow[bend left=5]{r}{f}
	\arrow[bend right=5,leftarrow,swap]{r}{g}
	&
	\DD(T_2):
	\end{tikzcd} 
    \)
    \\
    \(
	f\circ g\simeq a^n\cdot \id, \,
	g\circ f\simeq a^n\cdot \id
	\)
	\end{minipage}
	\right.
	\right\}
	\]
   and
	\[
	\lambdau_{a}(T_1,T_2)
	=
	\min
	\left\{%
	n\in\Z_{\geq0}
	\left|\,
	\begin{minipage}{5cm}
    \centering
	\(
	\exists
    \)
    homogeneous chain maps
    \\
    \(
	\begin{tikzcd}
	\DD(T_1)
	\arrow[bend left=5]{r}{f}
	\arrow[bend right=5,leftarrow,swap]{r}{g}
	&
	\DD(T_2):
	\end{tikzcd} 
    \)
    \\
    \(
	f\circ g\simeq a^n\cdot \id, \,
	g\circ f\simeq a^n\cdot \id
	\)
    \\
    and $\QGrad{q}(f),\QGrad{q}(g)\leq 0$
	\end{minipage}
	\right.
	\right\}
	\]
\end{definition}
We take the stance that for $f$ the zero map, $\QGrad{q}(f) \leq 0$ holds.
So, if $a^n\cdot \id$ is nullhomotopic for both $T_1$ and $T_2$, then
$\lambdau_{a}(T_1,T_2) \leq n$.

\begin{remark}\label{rem:basic_properties_lambda_tangles}
    Observe that \cref{prop:properties_lambda} still holds
    with $\lambda_a,\, \lambdau_a$ and tangles $T_1,\, T_2,\, T_3$ in lieu of
    $\lambda,\, \lambdau$ and knots.
\end{remark}

\begin{definition}
	For \(x\in \{\Ni, \ConnectivityX, \No\}\), define 
	\[
	a_x
	\coloneqq
	\begin{cases*}
	D_\bullet & if \(x=\Ni\)\\
	S^2 & if \(x=\ConnectivityX\)\\
	D_\circ & if \(x=\No\)
	\end{cases*}
	\]
	Conversely, for
	\(a\in\{D_\bullet,S^2,D_\circ\}\subset\mathcal{B}\), define 
	\(x_a\in \{\Ni, \ConnectivityX, \No\}\) by \(a_{x_a}=a\). 
\end{definition}

\begin{lemma}\label{lem:nullhomotopic_basepoint_action_from_connectivity}
	For any Conway tangle \(T\) with a basepoint, \(a_{\conn{T}}\cdot \id_{\DD(T)}\simeq 0\). 
\end{lemma}

For example, if \(T=\No\), then \(\DD(T)=\bullet\), and hence \(D_\circ\cdot \id_{\DD(T)}\) is in fact equal to 0.

\begin{proof}
	This follows from Bar-Natan's basepoint moving lemma \cite[Lemma 5.11]{KWZ_linear}.
	The proof is identical to the proof of \cite[Proposition 5.9]{KWZ_linear}, except that we consider \(\DD(T)\) instead of \(\Khr(T)\), which is the mapping cone of \(G\cdot \id_{\DD(T)}\).
\end{proof}

\begin{lemma}\label{lem:lambda-01}
	Let \(T_1\) and \(T_2\) be two Conway tangles with a basepoint, with the same connectivity \(x\in \{\Ni, \ConnectivityX, \No\}\). 
	Then 
    \[
	\lambda_{a_x}(T_1,T_2)
	=
	\begin{cases*}
	0 & if \(\DD(T_1) \simeq \QGrad{q^\ell}\DD(T_2)\) for some~$\ell\in\Z$,
	\\
	1 & else.
	\end{cases*}
	\]
    and
    \[
	\lambdau_{a_x}(T_1,T_2)
	=
	\begin{cases*}
	0 & if \(\DD(T_1) \simeq \DD(T_2)\),
	\\
	1 & else.
	\end{cases*}
	\]
\end{lemma}

\begin{proof}
	By \cref{lem:nullhomotopic_basepoint_action_from_connectivity}, \(a_x\cdot \id_{\DD(T_i)}\simeq 0\) for \(i=1,2\), so the maps \(f=g=0\) give rise to the upper bounds \(\lambda_{a_x}(T_1,T_2)\leq 1\) and \(\lambdau_{a_x}(T_1,T_2)\leq 1\). We conclude by \cref{item:properties_lambda:3} and \cref{item:properties_lambda:4} of \cref{prop:properties_lambda} and \cref{rem:basic_properties_lambda_tangles}. 
\end{proof}

\begin{lemma} \label{lem:lambda_crossing_change}
	We have
	\[
    \lambda_{a}(\textnormal{$\vcenter{\hbox{\def\svgwidth{11.6pt}}}$},\textnormal{$\vcenter{\hbox{\def\svgwidth{11.6pt}}}$})
    =
	\lambdau_{a}(\textnormal{$\vcenter{\hbox{\def\svgwidth{11.6pt}}}$},\textnormal{$\vcenter{\hbox{\def\svgwidth{11.6pt}}}$})
	=
	\begin{cases*}
	\infty & if \(a=D_\circ\), and
	\\
	1 & if \(a=D_\bullet\) or \(S^2\).
	\end{cases*}
	\]
\end{lemma}

\begin{proof}
	Since \(\DD(\textnormal{$\vcenter{\hbox{\def\svgwidth{11.6pt}}}$})\not\simeq \QGrad{q^\ell}\DD(\textnormal{$\vcenter{\hbox{\def\svgwidth{11.6pt}}}$})\) for all~$\ell\in\Z$, \cref{lem:lambda-01} yields \(\lambda_{S^2}(\textnormal{$\vcenter{\hbox{\def\svgwidth{11.6pt}}}$},\textnormal{$\vcenter{\hbox{\def\svgwidth{11.6pt}}}$})=\lambdau_{S^2}(\textnormal{$\vcenter{\hbox{\def\svgwidth{11.6pt}}}$},\textnormal{$\vcenter{\hbox{\def\svgwidth{11.6pt}}}$})=1\).
	The following maps \(f\) and \(g\) witness the identity \(\lambda_{D_\bullet}(\textnormal{$\vcenter{\hbox{\def\svgwidth{11.6pt}}}$},\textnormal{$\vcenter{\hbox{\def\svgwidth{11.6pt}}}$})=\lambdau_{D_\bullet}(\textnormal{$\vcenter{\hbox{\def\svgwidth{11.6pt}}}$},\textnormal{$\vcenter{\hbox{\def\svgwidth{11.6pt}}}$})=1\):
	\[
	\begin{tikzcd}
	\DD(\textnormal{$\vcenter{\hbox{\def\svgwidth{11.6pt}}}$})
	\arrow[bend left=10]{d}{f}
	\arrow[bend right=10,leftarrow,swap]{d}{g}
	&
	\GGzqh{\circ}{0}{-2}{-1}
	\arrow{r}{S}
	&
	\GGzqh{\bullet}{0}{-1}{0}
	\arrow[bend left=10]{d}{D_\bullet}
	\arrow[bend right=10,leftarrow,swap]{d}{1}
    &
	\\
	\DD(\textnormal{$\vcenter{\hbox{\def\svgwidth{11.6pt}}}$})
	&
	&
	\GGzqh{\bullet}{0}{1}{0}
    \arrow{r}{S}
    &
    \GGzqh{\circ}{0}{2}{1}
	\end{tikzcd}
	\]
	As for $\lambda_{D_\circ}(\textnormal{$\vcenter{\hbox{\def\svgwidth{11.6pt}}}$},\textnormal{$\vcenter{\hbox{\def\svgwidth{11.6pt}}}$})$ and $\lambdau_{D_\circ}(\textnormal{$\vcenter{\hbox{\def\svgwidth{11.6pt}}}$},\textnormal{$\vcenter{\hbox{\def\svgwidth{11.6pt}}}$})$,
  observe that for chain maps $f$ and $g$ preserving the homological grading,
  \(g\circ f\) is zero on the \(\circ\)-end of \(\DD(\textnormal{$\vcenter{\hbox{\def\svgwidth{11.6pt}}}$})\). 
  Moreover, for any homotopy $h$ of \(\DD(\textnormal{$\vcenter{\hbox{\def\svgwidth{11.6pt}}}$})\), $dh + hd$ equals $h S_{\circ}$ on the \(\circ\)-end. 
  Since $D^k_\circ$ is not of the form $h S_{\circ}$ for any $k \geq 0$, 
  it follows that \(g\circ f \not\simeq D^k_\circ \cdot \id_{\DD(\textnormal{$\vcenter{\hbox{\def\svgwidth{10pt}}}$})}\). 
  So there are no desired maps $f$ and $g$.
\end{proof}

\begin{proposition}\label{prop:lambda-tangles-to-links}
	Let \(T_1\), \(T_2\), \(T\) be three Conway tangles with \(\conn{T_1}=\conn{T_2}\neq \conn{T}\). 
	Let \(a\in\{D_\bullet,S^2,D_\circ\}\subset\mathcal{B}\) be such that 
	\(\{\conn{T_1}, \conn{T},x_a\}=\{\Ni, \ConnectivityX, \No\}\). 
	Then
	\[
	\lambda(T_1\cup T,T_2\cup T)
	\leq
	\lambda_{a}(T_1,T_2)
	\qquad \text{and} \qquad
	\lambdau(T_1\cup T,T_2\cup T)
	\leq
	\lambdau_{a}(T_1,T_2). 
	\]
\end{proposition}

The proof essentially boils down to the following lemma.

\begin{lemma}\label{lem:maps-tangles-to-links}
    Let \(T_1\), \(T_2\) be two Conway tangles with \(\conn{T_1}=\conn{T_2}\). 
	Let \(a\in\{D_\bullet,S^2,D_\circ\}\subset\mathcal{B}\) be such that 
	\(\conn{T_i}\ne x_a\). 
	Consider a pair of homogeneous maps
    \[
    \begin{tikzcd}
	\DD(T_1)
	\arrow[bend left=5]{r}{f}
	\arrow[bend right=5,leftarrow,swap]{r}{g}
	&
	\DD(T_2)
	\end{tikzcd}
    \]
    such that
	\(
	f\circ g\simeq a^n\cdot \id
	\), 
	\(
	g\circ f\simeq a^n\cdot \id
	\)
    for some~$n\geq 0$.
    Then, given any tangle $T$ such that \(\{\conn{T_1}, \conn{T},x_a\}=\{\Ni, \ConnectivityX, \No\}\), the maps $f$ and $g$ induce homogeneous chain maps
    \[
    \begin{tikzcd}
	\BNr(T_1\cup T)
	\arrow[bend left=5]{r}{f'}
	\arrow[bend right=5,leftarrow,swap]{r}{g'}
	&
	\BNr(T_2\cup T)
	\end{tikzcd}
    \]
    with 
	\(
	f'\circ g'\simeq G^n\cdot \id
	\), 
	\(
	g'\circ f'\simeq G^n\cdot \id
	\)
    and \(\QGrad{q}(f')=\QGrad{q}(f),\, \QGrad{q}(g')=\QGrad{q}(g)\).
\end{lemma}    

\begin{proof}
	Given $x\in \{\Ni,\No,\ConnectivityX\}$, let
    \begin{equation}\label{eq:lem:maps-tangles-to-links}
    b_x=
    \begin{cases*}
    -a_x & for $x=\Ni$ and $x=\No$, \\
    a_x & for $x=\ConnectivityX$.
    \end{cases*}
  \end{equation}
    Then $G$ can be expressed in terms of the $b_x$ as $G=b_{
    \vcenter{\hbox{\def\svgwidth{8pt}%% Creator: Inkscape inkscape 0.92.5, www.inkscape.org
%% Accompanies image file 'conn_infty.pdf' (pdf, eps, ps)
\begingroup%
  \makeatletter%
  \providecommand\color[2][]{%
    \errmessage{(Inkscape) Color is used for the text in Inkscape, but the package 'color.sty' is not loaded}%
    \renewcommand\color[2][]{}%
  }%
  \providecommand\transparent[1]{%
    \errmessage{(Inkscape) Transparency is used (non-zero) for the text in Inkscape, but the package 'transparent.sty' is not loaded}%
    \renewcommand\transparent[1]{}%
  }%
  \providecommand\rotatebox[2]{#2}%
  \newcommand*\fsize{\dimexpr\f@size pt\relax}%
  \newcommand*\lineheight[1]{\fontsize{\fsize}{#1\fsize}\selectfont}%
  \ifx\svgwidth\undefined%
    \setlength{\unitlength}{11.6198051bp}%
    \ifx\svgscale\undefined%
      \relax%
    \else%
      \setlength{\unitlength}{\unitlength * \real{\svgscale}}%
    \fi%
  \else%
    \setlength{\unitlength}{\svgwidth}%
  \fi%
  \global\let\svgwidth\undefined%
  \global\let\svgscale\undefined%
  \makeatother%
  \begin{picture}(1,0.99995636)%
    \lineheight{1}%
    \setlength\tabcolsep{0pt}%
    \put(0,0){\includegraphics[width=\unitlength,page=1]{conn_infty.pdf}}%
  \end{picture}%
\endgroup%
}}
    }+b_{
    \vcenter{\hbox{\def\svgwidth{8pt}%% Creator: Inkscape inkscape 0.92.5, www.inkscape.org
%% Accompanies image file 'conn_0.pdf' (pdf, eps, ps)
\begingroup%
  \makeatletter%
  \providecommand\color[2][]{%
    \errmessage{(Inkscape) Color is used for the text in Inkscape, but the package 'color.sty' is not loaded}%
    \renewcommand\color[2][]{}%
  }%
  \providecommand\transparent[1]{%
    \errmessage{(Inkscape) Transparency is used (non-zero) for the text in Inkscape, but the package 'transparent.sty' is not loaded}%
    \renewcommand\transparent[1]{}%
  }%
  \providecommand\rotatebox[2]{#2}%
  \newcommand*\fsize{\dimexpr\f@size pt\relax}%
  \newcommand*\lineheight[1]{\fontsize{\fsize}{#1\fsize}\selectfont}%
  \ifx\svgwidth\undefined%
    \setlength{\unitlength}{11.61980321bp}%
    \ifx\svgscale\undefined%
      \relax%
    \else%
      \setlength{\unitlength}{\unitlength * \real{\svgscale}}%
    \fi%
  \else%
    \setlength{\unitlength}{\svgwidth}%
  \fi%
  \global\let\svgwidth\undefined%
  \global\let\svgscale\undefined%
  \makeatother%
  \begin{picture}(1,0.99995649)%
    \lineheight{1}%
    \setlength\tabcolsep{0pt}%
    \put(0,0){\includegraphics[width=\unitlength,page=1]{conn_0.pdf}}%
  \end{picture}%
\endgroup%
}}
    }+b_{
    \vcenter{\hbox{\def\svgwidth{8pt}%% Creator: Inkscape inkscape 0.92.5, www.inkscape.org
%% Accompanies image file 'conn_x.pdf' (pdf, eps, ps)
\begingroup%
  \makeatletter%
  \providecommand\color[2][]{%
    \errmessage{(Inkscape) Color is used for the text in Inkscape, but the package 'color.sty' is not loaded}%
    \renewcommand\color[2][]{}%
  }%
  \providecommand\transparent[1]{%
    \errmessage{(Inkscape) Transparency is used (non-zero) for the text in Inkscape, but the package 'transparent.sty' is not loaded}%
    \renewcommand\transparent[1]{}%
  }%
  \providecommand\rotatebox[2]{#2}%
  \newcommand*\fsize{\dimexpr\f@size pt\relax}%
  \newcommand*\lineheight[1]{\fontsize{\fsize}{#1\fsize}\selectfont}%
  \ifx\svgwidth\undefined%
    \setlength{\unitlength}{11.61026645bp}%
    \ifx\svgscale\undefined%
      \relax%
    \else%
      \setlength{\unitlength}{\unitlength * \real{\svgscale}}%
    \fi%
  \else%
    \setlength{\unitlength}{\svgwidth}%
  \fi%
  \global\let\svgwidth\undefined%
  \global\let\svgscale\undefined%
  \makeatother%
  \begin{picture}(1,0.99940785)%
    \lineheight{1}%
    \setlength\tabcolsep{0pt}%
    \put(0,0){\includegraphics[width=\unitlength,page=1]{conn_x.pdf}}%
  \end{picture}%
\endgroup%
}}}$ and, for all~$k\geq 0$, one has
    $G^k=b^k_{
    \vcenter{\hbox{\def\svgwidth{8pt}}}
    }+b^k_{
    \vcenter{\hbox{\def\svgwidth{8pt}}}
    }+b^k_{
    \vcenter{\hbox{\def\svgwidth{8pt}}}}$.
    Let $f$ and $g$ be as in the statement of the lemma. 
    By \cref{thm:pairing_thm}, 
    \[
    X_i
    \coloneqq 
    \Mor(-\DD(T),\DD(T_i))
    \simeq
    \QGrad{q^{-1}} \BNr(T_i\cup T)
    \quad
    \text{for }
    i=1,2.
    \] 
	We now define chain maps
	\(f'\coloneqq f\circ-\colon X_1 \to X_2\) and  \(g'\coloneqq g\circ-\colon X_2 \to X_1\). 
	Then, by \cref{lem:nullhomotopic_basepoint_action_from_connectivity},
        and using that $b_x$ is an element of the center $Z(\mathcal{B})$ of $\mathcal{B}$ for all $x$, it follows for all \(n\geq1\) that
	\begin{align*}
		G^n\cdot \id_{X_1}
		&\simeq
		(G^n\cdot \id_{\DD(T_1)})\circ -
		\\
		&\simeq
		((G^n-b^n_{\conn{T_1}})\cdot \id_{\DD(T_1)})\circ -
		=
		-\circ ((G^n-b^n_{\conn{T_1}})\cdot \id_{\DD(T)})
		\\
		&\simeq
		-\circ((G^n-b^n_{\conn{T_1}}-b^n_{\conn{T}})\cdot \id_{\DD(T)})
		\simeq 
		-\circ (b_{x_a}^n\cdot \id_{\DD(T)})
		\simeq 
		\pm a^n\cdot \id_{X_1}
		\\
		&\simeq 
		\pm g'\circ f'. 
	\end{align*}
	Analogously, we can prove that for \(n\geq 1\), 
	\(G^n\cdot \id_{X_2}\simeq \pm f'\circ g'\). 
	The sign \(\pm\), which is determined by the sign in \eqref{eq:lem:maps-tangles-to-links} 
	and the parity of \(n\), agrees in both identities. 
	Thus, by replacing \(f'\) by \(\pm f'\), we obtain 
	\[
	G^n\cdot \id_{X_1}\simeq g'\circ f'
	\quad
	\text{and}
	\quad
	G^n\cdot \id_{X_2}\simeq f'\circ g'
	\] 
	as desired. 
	For \(n=0\), these identities obviously hold.
\end{proof}

\begin{proof}[Proof of \cref{prop:lambda-tangles-to-links}]
	Let \(n=\lambdau_a(T_1,T_2)\), and consider maps $f$ and $g$ realising it. 
    By \cref{lem:maps-tangles-to-links}, they induce homogeneous maps 
	\[
	\begin{tikzcd}
		\BNr(T_1\cup T)
		\arrow[bend left=5]{r}{f'}
		\arrow[bend right=5,leftarrow,swap]{r}{g'}
		&
		\BNr(T_2\cup T)
	\end{tikzcd}
	\]
	with 
  \(
	f'\circ g'\simeq G^n\cdot \id
	\), 
	\(
	g'\circ f'\simeq G^n\cdot \id
	\)
    and \(\QGrad{q}(f'), \QGrad{q}(g')\leq 0\). 
    This shows that $\lambdau({T_1\cup T}, {T_2\cup T}) \leq n$. The proof that $\lambda(T_1\cup T, T_2\cup T) \leq \lambda_{a}(T_1,T_2)$ is analogous.
\end{proof}

\begin{corollary}\label{cor:lambda-tangles-to-links}
	For any Conway tangle \(T\) with connectivity \(\conn{T}=\No\), 
	\[
	\lambdau(\textnormal{$\vcenter{\hbox{\def\svgwidth{11.6pt}}}$}\cup T,\textnormal{$\vcenter{\hbox{\def\svgwidth{11.6pt}}}$}\cup T)
	\leq 1.
	\]
\end{corollary}

\begin{proof}
	It is immediate from \cref{prop:lambda-tangles-to-links} with \(T_1=\textnormal{$\vcenter{\hbox{\def\svgwidth{11.6pt}}}$}\) and \(T_2=\textnormal{$\vcenter{\hbox{\def\svgwidth{11.6pt}}}$}\) 
    and from \cref{lem:lambda_crossing_change} that
    $\lambdau(\textnormal{$\vcenter{\hbox{\def\svgwidth{11.6pt}}}$}\cup T,\textnormal{$\vcenter{\hbox{\def\svgwidth{11.6pt}}}$}\cup T)
	\leq
    \lambdau_{D_\bullet}(\textnormal{$\vcenter{\hbox{\def\svgwidth{11.6pt}}}$},\textnormal{$\vcenter{\hbox{\def\svgwidth{11.6pt}}}$})=1$. 
\end{proof}

We are ready to prove \cref{thm:lambda-gordian}.

\begin{proof}[Proof of \cref{thm:lambda-gordian}]
	Let \(u=u(K_1,K_2)\) be the Gordian distance between the two knots \(K_1\) and \(K_2\). 
	Then there exists a sequence of knots \(K_1=J_0,J_1,\dots,J_u=K_2\), such that \(J_i\) is obtained from \(J_{i-1}\) by a single crossing change. 
	Then indeed
	\[
	\lambdau(K_1,K_2)
	\leq
	\sum_{i=1}^{u} \lambdau(J_{i-1},J_i)
	\leq u
	\]
	where the first inequality follows from \cref{item:properties_lambda:2} of \cref{prop:properties_lambda} and the second from \cref{cor:lambda-tangles-to-links}. Indeed, each $J_i$ can be written as $J_i=\textnormal{$\vcenter{\hbox{\def\svgwidth{11.6pt}}}$}\cup T$ or $\textnormal{$\vcenter{\hbox{\def\svgwidth{11.6pt}}}$}\cup T$ for some tangle~$T$. Since the $J_i$ are knots (i.e.\ they only have one component), the connectivity of $T$ must be either $\Ni$ or~$\No$. However, because of the orientation of $\textnormal{$\vcenter{\hbox{\def\svgwidth{11.6pt}}}$}$ and~$\textnormal{$\vcenter{\hbox{\def\svgwidth{11.6pt}}}$}$, one can only have~$\conn{T}=\No$.
\end{proof}

In order to prove \cref{thm:lambda_proper_rational_gordian}, we need the following lemma.
For a bigraded chain complex $C$ of right $\mathcal{B}$-modules,
let us write $\textnormal{End}(C)$ for the graded $\Z[G]$-module of homotopy classes of chain maps $C\to C$, and $\textnormal{End}_n(C)$ for the abelian group arising as homogeneous part of $\textnormal{End}(C)$ consisting of maps of quantum degree \(-2n\).

\begin{lemma} \label{lem:iso_endomorphisms}
    Given \(p\) and \(q\) coprime odd integers, the chain complex \(\DD(Q_{\nicefrac{p}{q}})\) has a \(\circ\)-end and a \(\bullet\)-end. Denote these ends by \(X\) and \(Y\) respectively:
    \begin{equation}\label{eq:lem:iso_endomorphisms}
    \begin{tikzcd}[row sep=-0.2cm, font=\small]
        (X=\circ)
        \arrow[dash]{r}{S \textrm{ or } S^2}
        &
        \cdots
        \arrow[dash]{r}{S \textrm{ or } S^2}
        &
        (\bullet=Y)
    \end{tikzcd}
    \end{equation}
    Let \(\psi\) be a representative of an element in
    \(\textnormal{End}_{n}(\DD(Q_{\nicefrac{p}{q}}))\). 
    The restriction of $\psi$ to $X$ can be written uniquely as 
    \(
    \left.\psi\right|_{X\to X} 
    = 
    a_{\psi}D_{\circ}^n+b_{\psi}S^{2n}
    \) 
    for \(a_{\psi},b_{\psi}\in \Z\) 
    and \(b_{\psi}=0\) when \(n=0\). 
    Similarly, 
    \(
    \left.\psi\right|_{Y\to Y} 
    = 
    c_{\psi}D_{\bullet}^n+d_{\psi}S^{2n}
    \) 
    with \(c_{\psi},d_{\psi}\in \Z\) 
    and \(d_{\psi}=0\) when \(n=0\).
    Define maps
    \[
    \varphi_0\co \textnormal{End}_0(\DD(Q_{\nicefrac{p}{q}})) \longrightarrow \Z,
    \quad
    [\psi] \longmapsto a_{\psi},
    \]
    \[
    \varphi_n\co \textnormal{End}_n(\DD(Q_{\nicefrac{p}{q}}))  \longrightarrow \Z^2, 
    \quad
    [\psi] \longmapsto (a_{\psi},c_{\psi}) \qquad \text{ for } n>0.
    \]
    The maps \(\varphi_n\) are well-defined. 
    Moreover, they are isomorphisms for all \(n\geq 0\).
\end{lemma}

\begin{proof}
    By \cref{lem:goodzigzag},
    the ends $X$ and $Y$ are odd.
    Recall from \ref{eq:consecutived} of \cref{def:zigzag} that this means that the differentials attached to each of these ends are indeed either given by $S$ or~$S^2$, as shown in \eqref{eq:lem:iso_endomorphisms}. 
    It follows that the $D_{\circ}^n$ and $D_{\bullet}^n$ components of $\left.\psi\right|_{X\to X}$ and $\left.\psi\right|_{Y\to Y}$ cannot be homotoped away; 
    in other words, for any $\psi'$ homotopic to~$\psi$, one finds $a_\psi=a_{\psi'}$ and~$c_\psi=c_{\psi'}$. This shows that the maps $\varphi_n$ only depend on the homotopy class of~$\psi$, and so they are well-defined.
    The maps $\varphi_n$ are surjective, since \(\varphi_0([\id]) = 1\) and 
    \[
        \varphi_n([X \overset{D^n_{\circ}}{\longrightarrow} X]) = (1,0), 
        \qquad
        \varphi_n([Y \overset{D^n_{\bullet}}{\longrightarrow} Y]) = (0,1) 
        \qquad
        \text{for } n>0.
    \]
    Furthermore, \cref{thm:pairing_thm} yields
    \[
    \text{End}(\DD(Q_{\nicefrac{p}{q}})) \cong \QGrad{q^{-1}}H_0(\BNr(-Q_{\nicefrac{p}{q}}\cup Q_{\nicefrac{p}{q}})) \cong \QGrad{q^{-1}}H_0(\BNr(U\sqcup U)) \cong \Z[G]\oplus \QGrad{q^{-2}}\Z[G].
    \]
    Therefore $\text{End}_0(\DD(Q_{\nicefrac{p}{q}}))\simeq \Z$ and $\text{End}_n(\DD(Q_{\nicefrac{p}{q}}))\simeq \Z^2$ for~$n>0$.
    So for~$n\geq1$, $\varphi_n$ is a surjective morphism of free abelian groups of the same rank. 
    It follows that these maps are isomorphisms.
\end{proof}

\begin{lemma}\label{lem:lambda_rational_tangles}
	Let 
	\(
	Q_{\nicefrac{p}{q}}
	\neq
	\textnormal{$\vcenter{\hbox{\def\svgwidth{11.6pt}}}$}
	\) 
	be a rational tangle with the same connectivity and orientation as 
	\(
	\textnormal{$\vcenter{\hbox{\def\svgwidth{11.6pt}}}$}.
	\)
	Then,  \(
	\lambda_{D_{\bullet}}
	(
	\textnormal{$\vcenter{\hbox{\def\svgwidth{11.6pt}}}$},
	Q_{\nicefrac{p}{q}}
	)
	=
	1.
	\) 
\end{lemma}

\begin{proof}
    By \cref{lem:parity},  $p$ and $q$ are odd coprime integers. 
    Since $\DD(Q_{\nicefrac{p}{q}}) \not\simeq \QGrad{q^n} \DD(\textnormal{$\vcenter{\hbox{\def\svgwidth{11.6pt}}}$})$ for all~$n\in \Z$, we know by \cref{prop:properties_lambda}\cref{item:properties_lambda:3} that $\lambda_{D_\bullet}(\textnormal{$\vcenter{\hbox{\def\svgwidth{11.6pt}}}$},Q_{\nicefrac{p}{q}})\geq 1$.
    We look for chain maps
    \[
    \begin{tikzcd}
    	\DD(\textnormal{$\vcenter{\hbox{\def\svgwidth{11.6pt}}}$})
    	\arrow[bend left=5]{r}{f}
    	\arrow[bend right=5,leftarrow,swap]{r}{g}
    	&
    	\DD(Q_{\nicefrac{p}{q}})
    \end{tikzcd}
    \]
    realising 
    $\lambda_{D_\bullet}(\textnormal{$\vcenter{\hbox{\def\svgwidth{11.6pt}}}$},Q_{\nicefrac{p}{q}})=1$. 
    The $\bullet$-end of the zigzag complex $\DD(Q_{\nicefrac{p}{q}})$ 
    can be determined from \cref{lem:mirror_complex}, \cref{lem:end_zz} and \cref{lem:grading_end_zz}. 
    We distinguish four cases depending on the slope \(\nicefrac{p}{q}\neq -1\), 
    which are shown in 
    \cref{tab:lambda_rational_tangles:maps_for_cases_p/q}.
    In each case, 
    the upper row shows the complex \(\DD(\textnormal{$\vcenter{\hbox{\def\svgwidth{11.6pt}}}$})\)
    and
    the lower row shows the $\bullet$-end 
    of $\DD(Q_{\nicefrac{p}{q}})$. 
    The vertical arrows define homogeneous chain maps $f$ and~$g$. 
    Since $D_\bullet \cdot \id_{\circ} = 0$, one has 
    \[
    g\circ f = D_\bullet \cdot \id_{\DD(\textnormal{$\vcenter{\hbox{\def\svgwidth{11.6pt}}}$})}.
    \] 
    Now consider the endomorphism $f \circ g$ of $\DD(Q_{\nicefrac{p}{q}})$. 
    The image of $f \circ g$ under the isomorphism 
    $\varphi_1$ of \cref{lem:iso_endomorphisms} 
    is equal to the image of the map 
    $D_\bullet \cdot \id_{\DD(Q_{\nicefrac{p}{q}})}$. 
    This implies that they represent the same element of
    $\text{End}_1(\DD(Q_{\nicefrac{p}{q}}))$, 
    i.e.\ 
    $f\circ g \simeq D_\bullet \cdot \id_{\DD(Q_{\nicefrac{p}{q}})}$.
\end{proof}

\begin{figure}[h]
	\centering
	\begin{subfigure}{0.5\textwidth}
		\centering
		\(
		\begin{tikzcd}
			&
			\GGzqh{\bullet}{0}{1}{0}
			\arrow[leftarrow, bend left=10]{d}{D_\bullet}
			\arrow[bend right=10,swap]{d}{1}
			\arrow{r}{S}
			&
			\GGzqh{\circ}{0}{2}{1}
			\\
			\cdots
			\arrow{r}{S \text{ or } S^2}
			&
			\GGzqh{\bullet}{0}{\ell}{0}
		\end{tikzcd}
		\) 
		\caption{$\nicefrac{p}{q}>0$}
	\end{subfigure}%
	\begin{subfigure}{0.5\textwidth}
		\centering
		\(
		\begin{tikzcd}
			\GGzqh{\bullet}{0}{1}{0}
			\arrow[bend right=10, swap]{d}{1}
			\arrow[leftarrow,bend left=10]{d}{D_\bullet}
			\arrow{r}{S}
			&
			\GGzqh{\circ}{0}{2}{1}
			\arrow[bend right=10, swap]{d}{S}			
			\\
			\GGzqh{\bullet}{0}{\ell}{0}
			\arrow{r}{S^2}
			&
			\GGzqh{\bullet}{0}{\ell+2}{1}
			\arrow{r}{D}
			&
			\bullet
			\cdots
		\end{tikzcd}
		\) 
		\caption{$-\nicefrac{1}{2}<\nicefrac{p}{q}<0$}
	\end{subfigure}%
	\\\bigskip
	\begin{subfigure}{0.5\textwidth}
		\centering
		\(
		\begin{tikzcd}
			\GGzqh{\bullet}{0}{1}{0}
			\arrow[bend right=10, swap]{d}{1}
			\arrow[leftarrow, bend left=10]{d}{D_\bullet}
			\arrow{r}{S}
			&
			\GGzqh{\circ}{0}{2}{1}
			\arrow[bend right=10, swap]{d}{1}			
			\\
			\GGzqh{\bullet}{0}{\ell}{0}
			\arrow{r}{S}
			&
			\GGzqh{\circ}{0}{\ell+1}{1}
			\arrow[leftarrow]{r}{D}
			&
			\circ
			\cdots
		\end{tikzcd}
		\)
		\caption{$-1<\nicefrac{p}{q}<-\nicefrac{1}{2}$}
	\end{subfigure}%
	\begin{subfigure}{0.5\textwidth}
		\centering
		\(
		\begin{tikzcd}
			\GGzqh{\bullet}{0}{1}{0}
			\arrow[leftarrow, bend left=10]{d}{1}
			\arrow[bend right=10,swap]{d}{D_\bullet}
			\arrow{r}{S}
			&
			\GGzqh{\circ}{0}{2}{1}
			\arrow[leftarrow, bend left=10]{d}{1}			
			\\
			\GGzqh{\bullet}{0}{\ell}{0}
			\arrow{r}{S}
			&
			\GGzqh{\circ}{0}{\ell+1}{1}
			\arrow{r}{D}
			&
			\circ
			\cdots
		\end{tikzcd}
		\)
		\caption{$\nicefrac{p}{q}<-1$}
	\end{subfigure}%
	\caption{%
		The maps $f$ and $g$ realising 
		$\lambda_{D_\bullet}(\protect\textnormal{$\vcenter{\hbox{\def\svgwidth{11.6pt}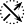}}$},Q_{\nicefrac{p}{q}})=1$ 
		depending on the slope~$\nicefrac{p}{q}$.
	}
	\label{tab:lambda_rational_tangles:maps_for_cases_p/q}
\end{figure}

\begin{remark}
	\cref{lem:lambda_rational_tangles} remains true for any other rational tangle with the same connectivity and orientation in lieu of \(\textnormal{$\vcenter{\hbox{\def\svgwidth{11.6pt}}}$}\).
	The relevant morphisms required for the proof can be constructed by analysing the slopes of the rational tangles and applying the interpretation of morphisms as intersection points of the corresponding curves \cite[Section~5]{KWZ}; compare also \cite[Proof of Lemma~6.2]{LZ}. 
	Moreover, the lemma remains true if we change the orientation of \emph{both} strands, since this does not affect the tangle invariants \cite[Proposition~4.8]{KWZ}.
	However, if we change the orientation of only one strand the statement becomes false. 
\end{remark}

\begin{proof}[Proof of \cref{thm:lambda_proper_rational_gordian}]  	
    Let $K_1$ and $K_2$ be two knots related by a proper rational tangle replacement. 
    One can write $K_1=T_1 \cup T$ and $K_2=T_2 \cup T$ with $T_1,\, T_2$ rational tangles. 
    By applying the same boundary-changing tangle-orientation-preserving homeomorphism to $T_1$ and~$T_2$, we may assume that  
    \(T_1=\textnormal{$\vcenter{\hbox{\def\svgwidth{11.6pt}}}$}\), see \cref{remark:slopechange}. 
    The tangle $T_2$ then has the same orientation and connectivity as \(\textnormal{$\vcenter{\hbox{\def\svgwidth{11.6pt}}}$}\).
    Moreover,~$\conn{T}=\No$. 
    The statement now follows from combining \cref{lem:lambda_rational_tangles} and \cref{prop:lambda-tangles-to-links}. 
\end{proof}

\subsection{Relationship with non-rational proper tangle replacements}
	There are non-rational proper tangle replacements \(T\rightsquigarrow T'\) with \(\lambda_G(T,T')=1\) to which we can apply the same argument as in the proofs of \cref{thm:lambda-gordian,thm:lambda_proper_rational_gordian}.  Let us consider two examples.

\begin{example}
	Let \(T_1\) be the \((3,-3)\)-pretzel tangle, i.e.\ the tangle sum \(Q_{\nicefrac{1}{3}}+Q_{\nicefrac{-1}{3}}\).
	Then a short computation shows that \(\DD(T_1)\) consists of a summand isomorphic to~\(\DD(\No)\), plus a component isomorphic to \(s_{4}(0)\) up to a grading shift, with notation from \cite[Definition~2.1]{KWZ_strong_inversions}. 
	(Note that the absolute grading of \(\DD(T_1)\) is independent of the orientation.)
	Since \(s_{4}(0)\) is a mapping cone of \(G\cdot \id_X\) on some complex \(X\), \(G\cdot\id_{s_4(0)}\simeq0\). 
	Hence \(\lambdau_G(T_1,\No)=1\). 
	Therefore, \(\lambdau(K,J)\leq 1\) for any two knots \(K\) and \(J\) that are related by a tangle replacement \(T=T_1\rightsquigarrow T'=\No\). 
	The same should be true for the \((2n+1,-2n-1)\)-pretzel tangle \(T_{n}\), \(n\geq1\), in place of \(T_1\).
	Indeed, computations suggest \cite{khtpp} that \(\DD(T_n)\) consists of a summand isomorphic to \(\DD(\No)\) and, for each \(i=1,\dots,n\), a component isomorphic to \(s_{4i}(0)\) up to a grading shift, for which \(G\cdot\id_{s_{4i}(0)}\simeq0\).
	
	Greene proved that the \((2n-1,-2n+1,m)\)-pretzel knots \(K_{n,m}\) with $n, m \geq 2$ are not quasi-alternating \cite{zbMATH05937353}, but their rational Khovanov homology was shown to be thin by Starkston and Qazaqzeh \cite{zbMATH06017041,MR2847230}. 
	From thinness, it follows that \(\BNr(K_{n,m};\Q[G])\) does not contain any \(G^k\)-knight piece for \(k>1\).
	We can now recover this absence of higher knights as a consequence of the tangle replacement \(T_n\rightsquigarrow \No\) which turns \(K_{n,m}\) into an unknot (see~\cref{fig:pretzel}).
\end{example}
\begin{figure}[h]
\includegraphics{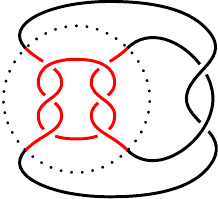}\hspace{3cm}%
\includegraphics{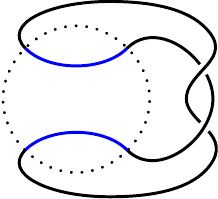}
\caption{Transforming the $(3,-3,2)$-pretzel knot ($8_{20}$ in the tables) into the unknot
by replacing the \textcolor{red}{$(3,-3)$-pretzel tangle} by \textcolor{blue}{$Q_0$ = \(\protect\Nob\)}.}
\label{fig:pretzel}
\end{figure}

\begin{figure}[t]\centering
	\begin{subfigure}{0.5\textwidth}\centering
		\(\TorusknotFromPT\)
		\caption{\(K=Q_{\nicefrac{-1}{3}}\cup \textcolor{red}{T}\)}\label{fig:pretzel-replacement:torusknot}
	\end{subfigure}%
	\begin{subfigure}{0.5\textwidth}\centering
		\(\UnknotFromPT\)
		\caption{\(U=Q_{\nicefrac{-1}{3}}\cup \textcolor{blue}{Q_{\nicefrac{1}{2}}}\)}\label{fig:pretzel-replacement:unknot}
	\end{subfigure}
	\caption{%
		A tangle replacement 
		\(\textcolor{red}{T}\rightsquigarrow\textcolor{blue}{Q_{\nicefrac{1}{2}}}\) 
		from (a)~the torus knot \(K=T(3,4)\) 
		to (b)~the unknot \(U\). 
		As explained in \cref{exa:pretzel-replacement}, 
		the orientation on the tangles 
		\(\textcolor{red}{T}\) and \(\textcolor{blue}{Q_{\nicefrac{1}{2}}}\) 
		is crucial. 
		There is no such tangle replacement between \(K\) and \(U\) 
		if we reverse the orientation on one of the strands.
}\label{fig:pretzel-replacement}
\end{figure}

\begin{example}\label{exa:pretzel-replacement}
	Consider the torus knot \(K=T(3,4)\). 
	Its invariant is computed in \cref{prop:homology_torus} 
	and is equal to
	\[
		\BNr(K;\F_2[G])
		\simeq
			\F_2[G]
			\oplus
			\left(
			\F_2[G]
			\overset{G}{\longrightarrow}
			\F_2[G]
			\right)
			\oplus
			\left(
			\F_2[G]
			\overset{G^2}{\longrightarrow}
			\F_2[G]
			\right).
	\]
	From this, 
	we can compute \(\lambda(K,U;\F_2)=2\); 
	see also \cref{prop:upper_bds_on_lambda-_fields}, 
	part \cref{enu4:prop:upper_bds_on_lambda-_fields}. 
	
	\cref{fig:pretzel-replacement} shows a tangle replacement 
	which turns \(K\) into the unknot: 
	The highlighted tangle \(\textcolor{red}{T}\) 
	on the right of \cref{fig:pretzel-replacement:torusknot} 
	is replaced by the rational tangle \(\textcolor{blue}{Q_{\nicefrac{1}{2}}}\). 
	The tangle invariant for \(\textcolor{red}{T}\) 
	was computed in \cite[Example~4.28]{KWZ} 
	and is equal to
	\[
	\textcolor{red}{\DD(T)}
	\simeq
	\textcolor{red}{\Big[
		\begin{tikzcd}[column sep=20pt,ampersand replacement=\&,red]
			\GGzqh{\bullet}{0}{-7}{-2}
			\&
			\bullet
			\arrow[swap]{l}{D}
			\&
			\circ
			\arrow[swap]{l}{S}
			\&
			\circ
			\arrow[swap]{l}{D}
			\arrow{r}{S}
			\&
			\bullet
			\arrow{r}{D}
			\&
			\bullet
			\arrow{r}{S^2}
			\&
			\bullet
			\arrow{r}{D}
			\&
			\bullet
			\arrow{r}{S}
			\&
			\circ
		\end{tikzcd}
	\Big]}.
	\]
	By \cref{prop:lambda-tangles-to-links}, 
	\(
	\lambda_{D_\circ}(
	\textcolor{red}{T},
	\textcolor{blue}{Q_{\nicefrac{1}{2}}}
	)
	\geq
	2
	\). 
	In fact, 
	one can show that 
	\(
	\lambda_{D_\circ}(
	\textcolor{red}{T},
	\textcolor{blue}{Q_{\nicefrac{1}{2}}}
	)
	=
	2
	\).
	
	Let 
	\(\textcolor{red}{T'}\) 
	be the tangle obtained from 
	\(\textcolor{red}{T}\) 
	by reversing the orientation of one of the two strands. 
	Then 
	\(
	\textcolor{red}{\DD(T')}
	\simeq 
	t^2\QGrad{q^6}\textcolor{red}{\DD(T)}
	\); 
	see for example \cite[Proposition~4.8]{KWZ}. 
	Let \(\textcolor{blue}{Q'_{\nicefrac{1}{2}}}\) 
	be the rational tangle \(\textcolor{blue}{Q_{\nicefrac{1}{2}}}\) 
	equipped with the same orientation at the tangle ends 
	as~\(\textcolor{red}{T'}\). 
	We claim that 
	\(
	\lambda_{S^2}(
	\textcolor{red}{T'},
	\textcolor{blue}{Q'_{\nicefrac{1}{2}}}
	)
	=
	1
	\). 
	Indeed, the Bar-Natan invariants of these two tangles are not isomorphic, and maps \(f\) and \(g\) realizing this value, 
	along with a homotopy \(g\circ f\simeq_h S^2\cdot \id\), 
	can be easily found:
	\begin{align*}
		\tikzmarknode{s0}{\textcolor{red}{\DD(T')}}
		&=
		\textcolor{red}{\Big[
			\begin{tikzcd}[column sep=20pt,ampersand replacement=\&,red]
				\tikzmarknode{s1}{\GGzqh{\bullet}{0}{-1}{0}}
				\&
				\tikzmarknode{s2}{\bullet}
				\arrow[swap]{l}{D}
				\&
				\tikzmarknode{s3}{\circ}
				\arrow[swap]{l}{S}
				\&
				\circ
				\arrow[swap]{l}{D}
				\arrow{r}{S}
				\&
				\bullet
				\arrow{r}{D}
				\arrow[black,dashed,bend left]{l}{S}
				\&
				\bullet
				\arrow{r}{S^2}
				\&
				\bullet
				\arrow{r}{D}
				\arrow[black,dashed,bend left]{l}{1}
				\&
				\bullet
				\arrow{r}{S}
				\&
				\circ
				\arrow[black,dashed,bend left]{l}{S}
			\end{tikzcd}
		\Big]}
		\\[10pt]
		\tikzmarknode{e0}{\textcolor{blue}{\DD(Q'_{\nicefrac{1}{2}})}}
		&=
		\textcolor{blue}{\Big[
		\begin{tikzcd}[column sep=20pt,blue,ampersand replacement=\&]
			\tikzmarknode{e1}{\GGzqh{\bullet}{0}{-1}{0}}
			\&
			\tikzmarknode{e2}{\bullet}
			\arrow[swap]{l}{D}
			\&
			\tikzmarknode{e3}{\circ}
			\arrow[swap]{l}{S}
		\end{tikzcd}
		\Big]}
	\end{align*}
	\begin{tikzpicture}[overlay,remember picture]
		\draw[mor,dashed] (s0.north west) to [in=-160,out=160,looseness=5] node[left] {h} (s0.south west);
		\draw[mor] ($(s0.south east)+(-20pt,0)$) to [in=105,out=-105] node[left]  {f} ($(e0.north east)+(-20pt,0)$);
		\draw[mor] ($(e0.north east)+(-20pt,0)$) to [in=-75,out=75]   node[right] {g} ($(s0.south east)+(-20pt,0)$);
		\draw[mor] (s1.south) to [in=105,out=-105] node[left]  {S^2} (e1.north);
		\draw[mor] (e1.north) to [in=-75,out=75]   node[right] {1}   (s1.south);
		\draw[mor] (s2.south) to [in=105,out=-105] node[left]  {S^2} (e2.north);
		\draw[mor] (e2.north) to [in=-75,out=75]   node[right] {1}   (s2.south);
		\draw[mor] (s3.south) to [in=105,out=-105] node[left]  {S^2} (e3.north);
		\draw[mor] (e3.north) to [in=-75,out=75]   node[right] {1}   (s3.south);
	\end{tikzpicture}%
	In fact, these maps are witnesses for
	\(
	\lambdau_{S^2}(
	\textcolor{red}{T'},
	\textcolor{blue}{Q'_{\nicefrac{1}{2}}}
	)
	=
	1
	\).
	Applying \cref{prop:lambda-tangles-to-links}, 
	we conclude that 
	there is no tangle replacement	\(\textcolor{red}{T'}\rightsquigarrow\textcolor{blue}{Q'_{\nicefrac{1}{2}}}\)
	between \(K\) and \(U\). 
\end{example}

\begin{remark}\label{ref:simple}
	\cref{exa:pretzel-replacement} 
	underlines the importance of orientations in tangle replacements. 
	Another takeaway from this example is the following:  
	\(\lambda(K)\) is not only a lower bound on the proper rational unknotting number. 
	By the same arguments as in the proofs of 
	\cref{thm:lambda-gordian,thm:lambda_proper_rational_gordian}, 
	\(\lambda(K,J)\) 
	is a lower bound on the minimum length of sequences of knots 
	\[
	(K=K_0,\dots,K_\ell=J)
	\]
	where for all
	\(i=1,\dots,\ell\),
	the knots 
	\(K_{i-1}\)
	and 
	\(K_i\)
	are related by certain ``simple'' tangle replacements; 
	namely,
	a tangle replacement \(T_1\rightsquigarrow T_2\) could be called simple
	if 
	\(\lambda_{S^2}(T_1,T_2)=1\), 
	\(\conn{T_1}=\conn{T_2}=\Ni\), 
	and the orientations of opposite tangle ends in \(T_1\) and \(T_2\) agree 
	(thus forcing \(\conn{T}=\No\) for the third tangle \(T\) 
	in \cref{lem:maps-tangles-to-links}).
	Naturally, tangle replacements of the form
	\(Q_{\nicefrac{p}{q}}\rightsquigarrow Q_{\nicefrac{p'}{q'}}\) 
	for \(q,q'\) even and \(p,p'\) odd 
	are simple in this sense 
	(by a computation analogous to \cref{lem:lambda_rational_tangles}).
	However, there are many other simple tangle replacements involving non-rational tangles,
	such as  \(\textcolor{red}{T'}\rightsquigarrow\textcolor{blue}{Q'_{\nicefrac{1}{2}}}\) 
	from \cref{exa:pretzel-replacement}. 
	In fact, we might call the tangle replacement
	\(\textcolor{red}{T'}\rightsquigarrow\textcolor{blue}{Q'_{\nicefrac{1}{2}}}\) 
	``very simple'', 
	since not only 
	\(
	\lambda_{S^2}(
	\textcolor{red}{T'},
	\textcolor{blue}{Q'_{\nicefrac{1}{2}}}
	)
	=
	1
	\), 
	but in fact
	\(
	\lambdau_{S^2}(
	\textcolor{red}{T'},
	\textcolor{blue}{Q'_{\nicefrac{1}{2}}}
	)
	=
	1
	\).
	Many rational tangle replacements are very simple in this sense,
        as we will further discuss in \cref{def:u0} and \cref{cor:strongerthm1}.
	This explains why the integer \(\lambdau(K,J)\), 
	though being an improvement on many other Gordian distance bounds from Khovanov homology, 
	is still relatively small for many examples. 
\end{remark}

\subsection{\texorpdfstring{Relationship with the $s$-invariant}{Relationship with the s-invariant}}

Given a field $\F$ and two knots~$K_1,\,K_2$, define 
\[
s_{\F}(K_1,K_2)\coloneqq s_{\F}(K_1)-s_{\F}(K_2).
\]

\begin{proposition}\label{thm:s_invariant_and_graded_lambda}
	Let \(K_1\) and \(K_2\) be knots. Then for any field \(\F\)
		\begin{equation} \label{eq:rasmussen_u_bound1}
		|s_{\F}(K_1,K_2)| 
		\leq 
		2\lambdau(K_1,K_2;\F).
		\end{equation}
	Moreover, 
		\begin{equation} \label{eq:rasmussen_u_bound2}
		\max \{0\, ,\, \max_{\F} s_{\F}(K_1,K_2)\} 
		+ 
		\max \{0\, ,\, \max_{\F} s_{\F}(K_2,K_1)\} 
		\leq 
		2\lambdau(K_1,K_2).
		\end{equation}
	Similarly,
		\[
		\max_{\F} s_{\F}(K_1,K_2) 
		+ 
		\max_{\F} s_{\F}(K_2,K_1) 
		\leq 
		2\lambda(K_1,K_2).
		\]
\end{proposition}
\begin{remark}
As a direct consequence of \cref{thm:s_invariant_and_graded_lambda}, $\lambdau$ is sharper than the $s$-invariant as a bound for the unknotting number: For any knot $K$ and any field $\mathbb{F}$,
\[
\tfrac{1}{2}|s_{\mathbb{F}}(K)| 
\leq 
\lambdau(K;\F) 
\leq 
\lambdau(K).
\]
In fact, the inequality \eqref{eq:rasmussen_u_bound1} of \cref{thm:s_invariant_and_graded_lambda} says that $\lambdau$ is sharper than the $s$-invariant as a bound for the Gordian distance.
Note that $\tfrac{1}{2}|s_{\mathbb{F}}(K_1, K_2)|$ is indeed the optimal lower bound for $u(K_1, K_2)$ given by $s_{\mathbb{F}}$, meaning that
for all $a_1,a_2\in\mathbb{Z}$ there exist knots $K_1, K_2$ with $s_{\mathbb{F}}(K_i) = 2a_i$ such that $u(K_1, K_2) = \tfrac{1}{2}|s_{\mathbb{F}}(K_1, K_2)|$. For example, simply take $K_i = T(2,2a_i+1)$ for $a_i \geq 0$ and $K_i = T(2,2a_i-1)$ for $a_i < 0$.

The left-hand side of inequality \eqref{eq:rasmussen_u_bound2}
is a lower bound for $u(K_1, K_2)$, since $\lambdau(K_1, K_2)$ is.
But this can also be proven more directly, using that $s_{\mathbb{F}}$ does not decrease under a negative-to-positive crossing change~\cite{LivingstonComputations}. For example, a knot $K$ with $s_{\mathbb{F}}(K) = 2$ and $s_{\mathbb{F}'}(K) = -2$ has $u(K) \geq 2$, since at least one crossing change of each sign is required to make $K$ trivial.

We conjecture that the left-hand side of \eqref{eq:rasmussen_u_bound2}
is the optimal lower bound for $u(K_1, K_2)$ provided by the set of all Rasmussen invariants.
Let us make this precise. Let $C = \{0,2,3,\ldots\}$ be the set of zero and all primes.
Let $a_1, a_2\colon C \to \mathbb{Z}$ be given such that $a_i(c) = a_i(0)$ for almost all $c\in C$.
Then we conjecture that there exist knots $K_1, K_2$ such that
$s_{\mathbb{F}}(K_i) = 2a_i(\Char \mathbb{F})$, and $u(K_1, K_2)$ equals the left-hand side of \eqref{eq:rasmussen_u_bound2}.

In fact, we can prove this---provided we accept the following second conjecture as true:
For $n \geq 2$, let $W_n$ be the $(n^2-1)$-twisted positive Whitehead double $W_n$ of the torus knot $T(n,n+1)$. 
Then $s_{\mathbb{F}}(W_n)$ is expected to be equal to $2$ if $\Char\mathbb{F}$ divides~$n$, and $0$ otherwise (cf.~\cite[Conj.~6.2 and 6.3]{schuetz},
\cite[Conj.~6.9]{LZ}, \cite[Conj.~1.3]{arXiv2401.08480}).
Indeed, if this is true, and $a_1$ and $a_2$ as above are given,
one may build knots $K_1$ and $K_2$ with the desired properties as connected sums of copies of the trefoil and various $W_n$ and their mirror images. 
The proof goes by induction over the left hand side of \eqref{eq:rasmussen_u_bound2}, using that $u(W_n) = 1$. 
We omit the details.
\end{remark}

\begin{corollary} \label{cor:s_invariant_and_ungr_lambda}
	For any knot \(K\) and any fields \(\F\) and \(\F'\) 
		\[
		|s_{\F}(K)-s_{\F'}(K)| \leq 2\lambda(K).
		\]    
\end{corollary}

\begin{proof}
	Applying the last inequality of \cref{thm:s_invariant_and_graded_lambda} to $K_1=K$ and $K_2=U$ yields
	\[
	\max_{\F} s_{\F}(K) - \min_{\F} s_{\F}(K) \leq 2\lambda(K).
	\]
	Thus for any fields $\F,\,\F'$
	\[
	|s_{\F}(K) - s_{\F'}(K)| \leq \max_{\F} s_{\F}(K) - \min_{\F} s_{\F}(K) \leq 2\lambda(K).
	\myqed
	\]
\end{proof}

The proof of \cref{thm:s_invariant_and_graded_lambda} requires the following lemma.

\begin{lemma}\label{lem:qdeg_f_g_and_s_invariant}
	Let \(K_1\) and \(K_2\) be two knots and \(\F\) a field. Consider homogeneous maps 
	\[
	\begin{tikzcd}
		\BNr(K_1;\F[G])
		\arrow[bend left=5]{r}{f}
		\arrow[bend right=5,leftarrow,swap]{r}{g}
		&
		\BNr(K_2;\F[G])
	\end{tikzcd}
	\]
	such that
	\(
	f\circ g\simeq G^n\cdot \id
	\) and
	\(
	g\circ f\simeq G^n\cdot \id
	\)
	for some~\(n\geq 0\).
	Then 
	\begin{equation}\label{eq:qdeg_f_and_g_pawn}
			\QGrad{q}(f) \leq s_{\F}(K_2,K_1),
			\quad
			\text{and}
			\quad
			\QGrad{q}(g) \leq s_{\F}(K_1,K_2).
	\end{equation}
\end{lemma}

\begin{proof}
    Let $f$ and $g$ be as in the statement. By \cref{rem:pawn_knight_pieces}, the complexes~$\BNr(K_i;\F[G])$, for~$i=1,2$, decompose as the direct sum of a pawn piece $t^0\QGrad{q^{s_{\F}(K_i)}}\F[G]$ and some $G^k$-knight pieces, with~$k \geq 1$. 
    The pawn piece of each complex clearly has no non-trivial homotopies, and so $G^n \cdot \id \not\simeq 0$. In particular, since when setting $G=1$ the chain complexes $\BNr(K_i;\F[G]/(G-1))$ for $i=1,2$ reduce to just the pawn pieces, this implies that the restrictions of $f$ and $g$ to the pawn pieces are non-trivial.
    Up to multiplication by units, they are given by
    \begin{align*}
    	f_{\sympawn}\co& 
    	t^0\QGrad{q^{s_{\F}(K_1)}}\F[G]
    	\overset{G^{a}}{\longrightarrow} 
    	t^0\QGrad{q^{s_{\F}(K_2)}}\F[G]
    	\\
    	g_{\sympawn}\co& 
    	t^0\QGrad{q^{s_{\F}(K_2)}}\F[G]
    	\overset{G^{b}}{\longrightarrow} 
    	t^0\QGrad{q^{s_{\F}(K_1)}}\F[G]
    \end{align*}
    for some $a,\, b \geq 0$. One finds
    \begin{equation*}
    \begin{split}
    \QGrad{q}(f) &= \QGrad{q}(f_{\sympawn})=s_{\F}(K_2)-s_{\F}(K_1)-2a \leq s_{\F}(K_2)-s_{\F}(K_1), \\
    \QGrad{q}(g) &= \QGrad{q}(g_{\sympawn})=s_{\F}(K_1)-s_{\F}(K_2)-2b \leq s_{\F}(K_1)-s_{\F}(K_2).
    \myqed
    \end{split}
    \end{equation*}
\end{proof}

\begin{proof}[Proof of \cref{thm:s_invariant_and_graded_lambda}]
We prove the first inequality. 
By \cref{prop:properties_lambda}\cref{item:properties_lambda:1}, we may assume without loss of generality that~$s_{\F}(K_1) \geq s_{\F}(K_2)$. 
Let $n=\lambdau(K_1,K_2;\F)$, and consider a pair of homogeneous chain maps $f$ and $g$ witnessing this. 
Then, in particular, $\QGrad{q}(g) \leq 0$ and
\(
\QGrad{q}(f) \leq s_{\F}(K_2)-s_{\F}(K_1)
\)
by \cref{lem:qdeg_f_g_and_s_invariant}.
Thus 
\[
2n 
= 
-\QGrad{q}(f)-\QGrad{q}(g)
\geq 
s_{\F}(K_1)-s_{\F}(K_2),
\]
proving the first inequality.
For the second inequality, let $n=\lambdau(K_1,K_2)$ and consider maps $f$ and $g$ realising it. For any field~$\F$, $f$ and $g$ induce homogeneous maps 
\[
\begin{tikzcd}
    \BNr(K_1;\F[G])
    \arrow[bend left=5]{r}{f_{\F}}
    \arrow[bend right=5,leftarrow,swap]{r}{g_{\F}}
    &
    \BNr(K_2;\F[G])
\end{tikzcd}
\]
with $\QGrad{q}(f_{\F}),\, \QGrad{q}(g_{\F}) \leq 0$ and $f_{\F} \circ g_{\F} \simeq G^{n} \cdot \id,\, g_{\F} \circ f_{\F} \simeq G^{n} \cdot \id$. Then $\QGrad{q}(f_{\F})$ and $\QGrad{q}(g_{\F})$ satisfy the inequalities of \cref{eq:qdeg_f_and_g_pawn}. Since $\QGrad{q}(f)=\QGrad{q}(f_{\F})$ and $\QGrad{q}(g)=\QGrad{q}(g_{\F})$ for all fields~$\F$, one gets
\begin{equation*}
\begin{split}
\QGrad{q}(f) &\leq \min_{\F} (s_{\F}(K_2)-s_{\F}(K_1))= - \max_{\F} (s_{\F}(K_1)-s_{\F}(K_2)), \\
\QGrad{q}(g) &\leq \min_{\F} (s_{\F}(K_1)-s_{\F}(K_2))= - \max_{\F} (s_{\F}(K_2)-s_{\F}(K_1)).
\end{split}
\end{equation*}
Combining this with the fact that $\QGrad{q}(f),\, \QGrad{q}(g) \leq 0$ yields
\begin{equation*}
\begin{split}
\QGrad{q}(f) &\leq  - \max \{0\, ,\, \max_{\F} (s_{\F}(K_1)-s_{\F}(K_2)) \}, \\
\QGrad{q}(g) &\leq - \max \{0\, ,\, \max_{\F} (s_{\F}(K_2)-s_{\F}(K_1)) \}.
\end{split}
\end{equation*}
Then 
\[
2n = -\QGrad{q}(f)-\QGrad{q}(g)\geq \max \{0\, ,\, \max_{\F} s_{\F}(K_1,K_2) \} + \max \{0\, ,\, \max_{\F} s_{\F}(K_2,K_1) \}.
\]
The third inequality is proved analogously by removing the condition $\QGrad{q}(f),\, \QGrad{q}(g) \leq 0$. 
\end{proof}

An integral version of the Rasmussen invariant has been introduced by Schütz~\cite{schuetz},
and other versions of the Rasmussen invariant, consisting of chain complexes up to an equivalence relation, have been proposed independently by Dunfield--Lipshitz--Schütz \cite{arXiv2312.09114} and the first author~\cite{arXiv2401.08480}.
The latter comes with a lower bound $d(K_1, K_2)$ for the smooth oriented cobordism distance (and thus also for the Gordian distance) between knots $K_1$ and $K_2$, which may be defined as follows:
$d(K_1, K_2)$ is the minimal $n \geq 0$ such that there exist homogeneous chain maps
$f\colon \BNr(K_1) \to \BNr(K_2)$,
$g\colon \BNr(K_2) \to \BNr(K_1)$
with $\QGrad{q}(f) = \QGrad{q}(g) = -2n$ such that $f$ and $g$ induce isomorphisms on
the (ungraded) homology of $\BNr(\ \cdot\ ;\Z[G]/(G-1))$.
\begin{proposition}\label{prop:sz}
For any two knots \(K_1\) and \(K_2\), we have \(d(K_1, K_2) \leq \lambdau(K_1, K_2)\).
\end{proposition}
\begin{proof}
Let $n = \lambdau(K_1, K_2)$.
By definition of $\lambdau$,
there exist homogeneous chain maps
$f'\colon \BNr(K_1) \to \BNr(K_2)$, $g'\colon \BNr(K_2) \to \BNr(K_1)$
with $f'\circ g' \simeq G^n\cdot\id$, $g'\circ f' \simeq G^n\cdot\id$
and $\QGrad{q}(f'), \QGrad{q}(g') \leq 0$.
Since $\QGrad{q}(f') + \QGrad{q}(g') = \QGrad{q}(G^n) = -2n$, 
it follows that $\QGrad{q}(f'), \QGrad{q}(g') \geq -2n$.
So we may set $f = G^{n + \QGrad{q}(f')/2} \cdot f'$ and 
$g = G^{n + \QGrad{q}(g')/2} \cdot g'$.
We claim that $f$ and $g$ are chain maps witnessing $d(K_1, K_2) \leq n$.
Indeed, $\QGrad{q}(f) = \QGrad{q}(g) = -2n$.
Moreover, it follows from $f\circ g \simeq G^{2n}\cdot\id$ and $g\circ f \simeq G^{2n}\cdot\id$
that $f$ and $g$ induce mutually inverse isomorphisms on $\BNr(\ \cdot\ ;\Z[G]/(G-1))$.
\end{proof}
\cref{prop:sz} combined with the inequality
$\tfrac{1}{2}|s_{\mathbb{F}}(K_1, K_2)| \leq d(K_1, K_2)$ shown in~\cite[Proposition~8.5]{arXiv2401.08480}
provides an alternative proof of \eqref{eq:rasmussen_u_bound1}.

\subsection{\texorpdfstring{Relationship with $G$-torsion}{Relationship with G-torsion}}

Let $K$ be a knot and $R$ a commutative ring.
Consider the chain complex $\BNr(K;R[G])$ in the category $\mathcal{M}_{R[G]}$ of graded, finitely generated, free $R[G]$-modules. Since $\mathcal{M}_{R[G]} \subset R[G]\textrm{-Mod}$, the category of modules over~$R[G]$, we can take the homology $H(K;R[G]) \coloneqq H_*(\BNr(K;R[G]))$ in the latter category. 

\begin{definition}\label{def:G-torsion}
	The \emph{\(G\)-torsion order} of an element $a$ of an $R[G]$-module is defined by
	\[
	\text{ord}_G(a)
	\coloneqq
	\min\{n\in\Z^{\geq0}\mid G^n \cdot a = 0\}.
	\]
	If $\text{ord}_G(a)<\infty$, we say that $a$ is $G$\emph{-torsion}.
	For~$i\in\Z$, we denote the $R[G]$-module of $G$-torsion elements of $H_i(K;R[G])$ by $T_i(K;R[G])$ and define
	\[
	\mathfrak{u}_i(K;R)\coloneqq \max_{a \in T_i(K;R[G])}\text{ord}_G(a).
	\]
\end{definition}
The following was stated as \cref{prop:localtorsion} in the introduction for $R = \mathbb{Z}$.

\begin{proposition}\label{prop:G_torsion_and_lambda}
	For all knots \(K_1\) and \(K_2\) and integers~\(i\),
	\[
	|\mathfrak{u}_i(K_1;R)-\mathfrak{u}_i(K_2;R)| \leq \lambda(K_1,K_2;R).
	\]
\end{proposition}

\begin{proof}
	Let $n=\lambda(K_1,K_2;R)$ and $f$, $g$ be homogeneous chain maps realising it, i.e.
	\[
	\begin{tikzcd}
		\BNr(K_1;R[G])
		\arrow[bend left=5]{r}{f}
		\arrow[bend right=5,leftarrow,swap]{r}{g}
		&
		\BNr(K_2;R[G])
	\end{tikzcd}
	\]
	such that $g \circ f \simeq G^n \cdot \id$ and $f \circ g \simeq G^n \cdot \id$.
	Given~$i\in\Z$, let $f_i$ and $g_i$ be the components of $f$ and $g$ between the $i^\text{th}$ chain modules.
	Then, for every $a \in H_i(K_1;R[G])$: 
	\[
	\text{ord}_G \left( f_{i,*}(a) \right) \geq \text{ord}_G \left( g_{i,*} \circ f_{i,*}(a) \right) =\text{ord}_G \left( G^n \cdot a \right) \geq \text{ord}_G \left( a \right)  - n.
	\]
	Taking the maximum over $T_i(K_1;R[G])$ we get:
	\[
	\begin{split}
		\mathfrak{u}_i(K_2;R) & = \max_{b \in T_i(K_2;R[G])} \text{ord}_G \left( b \right) \geq \max_{a \in T_i(K_1;R[G])} \text{ord}_G \left( f_{i,*}(a) \right)
		\\
		& \geq \max_{a \in T_i(K_1;R[G])} \text{ord}_G \left( a \right) - n = \mathfrak{u}_i (K_1;R) - n.
	\end{split}
	\]
	This shows that $\mathfrak{u}_i (K_1;R) - \mathfrak{u}_i (K_2;R) \leq n = \lambda (K_1,K_2;R)$. The proof that $\mathfrak{u}_i (K_2;R) - \mathfrak{u}_i (K_1;R) \leq n$ is analogous. 
\end{proof}

Combined with \cref{thm:lambda_proper_rational_gordian}, \cref{prop:G_torsion_and_lambda} implies that if $K_1$ and $K_2$ are related by a single proper rational tangle replacement, then, in all homological degrees~$i$, 
$|\mathfrak{u}_i(K_1;R)-\mathfrak{u}_i(K_2;R)| \leq 1.$
Moreover, if one defines $\mathfrak{u}(K_1,K_2;R)$ to be
\[
	\mathfrak{u}(K_1,K_2;R)\coloneqq\max_{i\in\Z} |\mathfrak{u}_i(K_1;R)-\mathfrak{u}_i(K_2;R)|,
\]
it is clear that $\mathfrak{u}(K_1,K_2;R) \leq \lambda(K_1,K_2;R)$ also holds.

\begin{definition}
	Given a knot~$K$, let $\mathfrak{u}(K;R)$ be the maximal $G$-torsion order in~$\BNr(K;R[G])$, i.e.
	\[
	\mathfrak{u}(K;R)
	\coloneqq 
	\max_{i\in\Z} \mathfrak{u}_i(K;R).
	\]
\end{definition}

For~$R=\Z$, this is the same invariant as $\mathfrak{u}_G(K)$ of \cite[Definition 3.21]{lambda1}.
Over fields~$\F$, the structure theorem for $\BNr(K;\F[G])$ simplifies the computation of the invariants $\lambda$ and~$\lambdau$. We note the following result:

\begin{proposition} \label{prop:upper_bds_on_lambda-_fields}
	Let \(K_1\) and \(K_2\) be knots and \(\F\) a field. Then
	\begin{enumerate}[label=(\roman*)]
		\item \label{item:upper_bds_on_lambda-_fields:1} \(\lambdau(K_1,K_2;\F) \leq \max \{\nicefrac{|s_{\F}(K_1,K_2)|}{2},\mathfrak{u}(K_1;\F),\mathfrak{u}(K_2;\F)\}\).
		\item \(\lambdau(K;\F)=\max \{\nicefrac{|s_{\F}(K)|}{2}, \mathfrak{u}(K;\F)\}\).
		\item \(\lambda(K_1,K_2;\F) \leq \max \{\mathfrak{u}(K_1;\F),\mathfrak{u}(K_2;\F)\}\).
		\item \label{enu4:prop:upper_bds_on_lambda-_fields} \(\lambda(K;\F)=\mathfrak{u}(K;\F)\).
	\end{enumerate}
\end{proposition}

\begin{remark}\label{rem:torsion_integers}
	There is no analogue to \cref{prop:upper_bds_on_lambda-_fields} over the integers, in general. However, the lemma holds over $\Z$ if the complexes $\BNr(K_1)$ and $\BNr(K_2)$ split as a sum of a pawn and some $G^k$-knights, with~$k\geq 1$. 
\end{remark}

\begin{lemma} \label{lem:f_g_torsion}
	Given knots \(K_1\) and~\(K_2\), consider two integers \(a\) and \(b\) such that\break \(2a\geq s_{\F}(K_1,K_2)\), \(2b\geq s_{\F}(K_2,K_1)\) and \(a+b\geq \max \{\mathfrak{u}(K_1;\F),\mathfrak{u}(K_2;\F)\}\).
	Then there exist homogeneous chain maps 
	\[
	\begin{tikzcd}
		\BNr(K_1;\F[G])
		\arrow[bend left=5]{r}{f}
		\arrow[bend right=5,leftarrow,swap]{r}{g}
		&
		\BNr(K_2;\F[G])
	\end{tikzcd}
	\]
	such that \(\QGrad{q}(f)=-2a\), \(\QGrad{q}(g)=-2b\) and \(f\circ g \simeq G^{a+b}\cdot \id,\, g\circ f \simeq G^{a+b}\cdot \id\).
\end{lemma}

\begin{proof}
	Let $a,b\in\Z$ be as in the statement. For~$i=1,2$, the complex $\BNr(K_i;\F[G])$ splits as the sum of a pawn and some $G^k$-knights, with~$k\geq 1$. 
	We define $f$ and $g$ to be the zero maps everywhere, except between the pawn pieces. There, we define
	\begin{align*}
		f_\sympawn
		=
		G^{a-\frac{1}{2}s_\F(K_1,K_2)}\cdot\id
		\co 
		t^0\QGrad{q^{s_\F(K_1)}}\F[G] 
		\to 
		t^0\QGrad{q^{s_\F(K_2)}}\F[G]
		\\
		g_\sympawn
		=
		G^{b+\frac{1}{2}s_\F(K_1,K_2)}\cdot\id
		\co 
		t^0\QGrad{q^{s_\F(K_2)}}\F[G] 
		\to
		t^0\QGrad{q^{s_\F(K_1)}}\F[G]
	\end{align*}
	Clearly $\QGrad{q}(f)=-2a$ and~$\QGrad{q}(g)=-2b$. The fact that $f\circ g \simeq G^{a+b}\cdot \id$ and $g\circ f \simeq G^{a+b}\cdot \id$ follows from the observation that, on all $G^k$-knight pieces of $\BNr(K_1;\F[G])$ and $\BNr(K_2;\F[G])$, $G^{a+b}\cdot\id$ is homotopic to 0, since $k\leq  \max \{\mathfrak{u}(K_1;\F),\mathfrak{u}(K_2;\F)\} \leq a+b$. 
\end{proof}

\begin{proof}[Proof of \cref{prop:upper_bds_on_lambda-_fields}]
	Let $n=\max \{\nicefrac{|s_{\F}(K_1,K_2)|}{2},\mathfrak{u}(K_1;\F),\mathfrak{u}(K_2;\F)\}$. The first statement follows directly from \cref{lem:f_g_torsion}, by taking $a=n, b=0$ if $s_{\F}(K_1,K_2)\geq 0$, and $a=0, b=n$ if~$s_{\F}(K_1,K_2)< 0$.
	
	For the second statement, let $n=\lambdau(K;\F)$ and consider chain maps $f$ and $g$ witnessing this. By point \cref{item:upper_bds_on_lambda-_fields:1}, with $K_1=K$ and~$K_2=U$, one has $n \leq \max \{\nicefrac{|s_{\F}(K)|}{2},\mathfrak{u}(K;\F)\}$. 
	The complex $\BNr(U;\F[G])$ consists of a pawn piece in homological degree 0, and does not contain any (generalised) knight pieces. 
	Hence the restriction of $g\circ f\simeq G^n\cdot \id$ to every $G^k$-knight piece of $\BNr(K;\F[G])$ must be null-homotopic.
	Consequently,~$n \geq \mathfrak{u}(K;\F)$. 
	\cref{thm:s_invariant_and_graded_lambda} yields $n \geq \nicefrac{|s_{\F}(K)|}{2}$.
	
	As for the last two statements, they follow from \cref{lem:f_g_torsion} by setting $a=\nicefrac{s_\F(K_1,K_2)}{2}+\max\{\mathfrak{u}(K_1;\F),\mathfrak{u}(K_2;\F)\}$ and~$b=\nicefrac{s_\F(K_2,K_1)}{2}$. 
\end{proof}

\section{Applications to torus knots}\label{sec:applications}

Having established the main results of the paper in the previous section, we now turn to examples and applications.
We focus on torus knots (in particular those on two and three strands),
which form a family with exciting and relatively well-understood Khovanov homologies.

\subsection{\texorpdfstring{Knots at $u_q$-distance 1 from 3-stranded torus knots}{Knots at u\_q-distance 1 from 3-stranded torus knots}}

This subsection contains the proof of the following theorem,
which obstructs the existence of a proper rational tangle replacement (\cref{def:proper_rational_gordian}) relating a knot to a positive torus knot on 3 strands.

\thmtorusPRTR*

The optimality of the bound on the number of positive crossings of a diagram of $K$ follows directly from the following lemma.

\begin{lemma}\label{lem:optimal_torus_PRTR}
	For any \(n\in\Z_{\geq 0}\) there exists a knot \(K_n\) that admits a diagram \(D_n\) with \(c_+(D_n)=4n+1\) and that is related to both \(T(3,3n+1)\) and \(T(3,3n+2)\) by a proper rational tangle replacement.
\end{lemma}

\begin{proof}
    Let $B_3=\langle \sigma_1,\sigma_2 \;|\; \sigma_1\sigma_2\sigma_1 = \sigma_2\sigma_1\sigma_2 \rangle$ be the braid group on three strands.
    Let $D_n$ be the knot diagram given by the closure of the braid $\beta_n=\sigma_1^{-1}(\sigma_2\sigma_1^2\sigma_2)^n \sigma_2$. The diagram $D_n$ has $4n+1$ positive crossings.
    
    We now construct a proper rational tangle replacement to $T(3,3n+1)$. For this we need the following identity: 
    \begin{equation}\label{eq:2n+1_twists}
        (\sigma_1 \sigma_2)^{3n+1}=\sigma_1^{2n+2}\beta_n.
    \end{equation}
    We prove this identity by induction on $n\in\Z_{\geq 0}$. 
    For $n=0$ the claim is obvious.
    Now suppose \cref{eq:2n+1_twists} holds for $n$; we prove it for $n+1$. 
    Let $\Delta^2$ denote the full twist $\Delta^2=(\sigma_1\sigma_2)^3$. 
    Then 
    \[
    (\sigma_1 \sigma_2)^{3n+4}
    =
    \Delta^2(\sigma_1 \sigma_2)^{3n+1}
    =
    \Delta^2\sigma_1^{2n+2}\beta_n
    =
    \sigma_1^{2n+2}\Delta^2\beta_n.
    \]
    The last equality comes from the well-known fact that the full twist $\Delta^2$ is in the centre of the braid group.
    Moreover, observe that 
    \[
    \Delta^2
    =
    \sigma_1\sigma_2\sigma_1(\sigma_2\sigma_1\sigma_2
)    =
    \sigma_1\sigma_2\sigma_1(\sigma_1\sigma_2\sigma_1)    
    =
    \sigma_1(\sigma_2\sigma_1^2\sigma_2)\sigma_1.
    \]
    Hence
    \[
    (\sigma_1 \sigma_2)^{3n+4}
    =
    \sigma_1^{2n+2}\Delta^2\beta_n
    =
    \sigma_1^{2n+3}(\sigma_2\sigma_1^2\sigma_2)\sigma_1\beta_n
    =
    \sigma_1^{2n+4}\beta_{n+1}.
    \]

    The knot $T(3,3n+1)$ is given by the closure of the braid $(\sigma_1 \sigma_2)^{3n+1}\in B_3$.
    Observe that the operation $(\sigma_1 \sigma_2)^{3n+1}=\sigma_1^{2n+2}\beta_n \leadsto \beta_n$ corresponds to a proper rational tangle replacement $Q_{-2n-2} \leadsto Q_{0}$ from a diagram of $T(3,3n+1)$ to $D_n$. 

    A proper rational replacement for $T(3,3n+2)$ can be constructed similarly.
		We have
    \[
    (\sigma_1 \sigma_2)^{3n+2}
    =
    \sigma_1(\sigma_2 \sigma_1)^{3n} \sigma_2\sigma_1\sigma_2
    =
    \sigma_1(\Delta^2)^{n} \sigma_1\sigma_2\sigma_1
    =
    \sigma_1(\sigma_1 \sigma_2)^{3n+1}\sigma_1
    =
    \sigma_1^{2n+3}\beta_n\sigma_1.
    \]
    The knot $T(3,3n+2)$ is given by the closure of the braid $(\sigma_1 \sigma_2)^{3n+2}\in B_3$.
    Furthermore, the closures of $\sigma_1^{2n+3}\beta_n\sigma_1$ and $\sigma_1^{2n+4}\beta_n$ coincide.
    It follows that the operation ${\sigma_1^{2n+4}\beta_n \leadsto \beta_n}$ corresponds to a proper rational tangle replacement $Q_{-2n-4} \leadsto Q_{0}$ from a diagram of $T(3,3n+2)$ to $D_n$. 
\end{proof}

The proof of \cref{thm:torus_PRTR} relies on the $\lambda$-invariant. In particular, it requires computing the Bar-Natan chain complexes of positive torus knots on 3 strands.

\begin{proposition}\label{prop:homology_torus}
    For any \(n\in\Z_{\geq 0}\), the chain complexes \(\BNr(T(3,3n+1);\F_2[G])\) and \(\BNr(T(3,3n+2);\F_2[G])\) are given in \cref{fig:homology_torus}. Each colored box stands for a shifted copy of \(\F_2[G]\) and each
    \(\resizebox{0.379\width}{0.3755\height}{\usebox\OrangeBox}\), 
    \(\raisebox{-0.014\textwidth}{\resizebox{0.379\width}{0.3755\height}{\usebox\bluepiece}}\), and
    \(\raisebox{-0.02\textwidth}{\resizebox{0.379\width}{0.3755\height}{\usebox\greenpiece}}\)
    corresponds respectively to a pawn piece, a \(G\)\nobreakdash-knight, and a \(G^2\)-knight. More precisely, for \(i\in\{1,2\}\), \(\BNr(T(3,3n+i);\F_2[G])\) is chain homotopy equivalent to the sum of the following three subcomplexes:
    \begin{gather*}
        t^0\QGrad{q^{6n+2(i-1)}}\F_2[G],
        \\
        \bigoplus_{k=0}^{n+i-2} \left(
        t^{4k+2}\QGrad{q^{6(n+k)+2i+2}}\F_2[G]
        \overset{G}{\longrightarrow}
        t^{4k+3}\QGrad{q^{6(n+k)+2i+4}}\F_2[G]
        \right),
        \\
        \bigoplus_{h=0}^{n-1} \left(
        t^{4h+4}\QGrad{q^{6(n+h)+2i+4}}\F_2[G]
        \overset{G^2}{\longrightarrow}
        t^{4h+5}\QGrad{q^{6(n+h)+2i+8}}\F_2[G]
        \right).
    \end{gather*}   
\end{proposition}

\begin{proof}
        The reduced $\mathbb{F}_2$-Khovanov homologies
        $\BNr(T(3,3n+i);\F_2[G]/G)$ were computed by Benheddi \cite[Theorem~4.12]{benheddi}. 
	(When comparing \cref{fig:homology_torus} with \cite[Figure~24]{benheddi}, note that his torus knots are negative, therefore the sign of all gradings is switched.
	Also note that he denotes the homological grading by $i$ and uses the $\delta$-grading, which is related to our quantum grading \(\QGrad{q}\) by~\(\delta=q-2i\).) 
	Since the vector space $\BNr(T(3,3n+i);\F_2[G]/G)$ is at most one-dimensional in each homological grading, this completely determines the action of~\(G\) and thus $\BNr(T(3,3n+i);\F_2[G])$.
\end{proof}

\begin{remark}\label{rem:homology_torus:coeffs}
	\cref{prop:homology_torus} remains true if we replace \(\F_2\) by~\(\Q\). 
	This follows from the universal coefficient theorem applied to reduced integral Khovanov homology, and the computation of the unreduced integral Khovanov homology of these knots
	\cite{homology_torus_knots_integers}.
	Computer calculations and computations by Kelomäki for the setting $G=0$ \cite{arXiv2306.11186} suggest that \cref{prop:homology_torus} also holds if we replace \(\F_2\) by~\(\Z\) 
	(and thus for any coefficients). 
        This is confirmed by computer calculations~\cite{khoca,khtpp} for all~$n \leq 50$.
\end{remark}

\begin{figure}[t]
    \centering
    \hspace{-0.7cm}
    \begin{tikzpicture}
        \node at (-0.2,0){$\BNr(T(3,3n+1);\F_2[G])$};
        \node[scale=.61*.93] at (-0.5,3.3)
         {\usebox\HomologyTorusThreeNPlusOne};
        \node at (7.3,0){$\BNr(T(3,3n+2);\F_2[G])$};
        \node[scale=.60*.93] at (7,3.3)
         {\usebox\HomologyTorusThreeNPlusTwo};
    \end{tikzpicture}
    \caption{The chain complexes of the positive torus knots $T(3,3n+1)$ and $T(3,3n+2)$. Each colored box represents a copy of $\F_2[G]$. The differentials of the complexes are given by multiplication by $G$ on the blue pairs and multiplication by $G^2$ on the green pairs. These tables are inspired by those of \cite{homology_torus_knots_integers}.}
    \label{fig:homology_torus}
\end{figure}

\begin{proof}[Proof of \cref{thm:torus_PRTR}]
    Let $i\in\{1,2\}$, and consider a knot $K$ at $u_q$-distance 1 from $T(3,3n+i)$.
    By \cref{thm:lambda_proper_rational_gordian}, this implies that $\lambda=\lambda(T(3,3n+i),K;\F_2)\leq 1$.
    In particular, for any diagram $D$ of $K$, there exist chain maps 
    $$
    \begin{tikzcd}
    \BNr(T(3,3n+i);\F_2[G])
    \arrow[bend left=3, start anchor={3}]{r}{f}
		\arrow[bend right=3,leftarrow,swap, start anchor={-3}]{r}{g}
		&
    \BNr(D;\F_2[G])
    \end{tikzcd}
  	$$
    such that $g\circ f \simeq G^\lambda \cdot \id$ and $f\circ g \simeq G^\lambda \cdot \id$. 
    The complex $\BNr(T(3,3n+i);\F_2[G])$ is shown in \cref{fig:homology_torus}. Observe that it contains a $G^2$-knight piece supported in homological degrees $4n$ and $4n+1$. 
    On this piece, the map $G^\lambda \cdot \id \not\simeq 0$, therefore the restrictions of $f$ and $g$ to it must be non-zero. This implies that $\BNr_{4n}(D;\F_2[G])\ne 0$ and $\BNr_{4n+1}(D;\F_2[G])\ne 0$, where $\BNr_k(D;\F_2[G])$ is the $k^\text{th}$ chain module of $\BNr(D;\F_2[G])$, for all $k\in\Z$.
    By construction, the complex $\BNr(D;\F_2[G])$ (not considered up to chain homotopy equivalence) is supported in homological degrees $k$ such that $-c_-(D)\leq k \leq c_+(D)$, where $c_-(D)$ is the number of negative crossings of $D$. 
    One concludes that $4n+1 \leq c_+(D)$, for all diagrams $D$ representing $K$.
\end{proof}

\begin{remark}
An alternative proof for \cref{thm:torus_PRTR} is provided by
using the `local torsion order' bound of \cref{prop:G_torsion_and_lambda}.
\end{remark}

\subsection{\texorpdfstring{$\lambda$- and $\lambdau$-distance between torus knots}{λ- and λ⁻-distance between torus knots}}

\begin{proposition}\label{prop:homology_torus_two}
    For any \(m\in\Z_{\geq 0}\), the chain complex \(\BNr(T(2,2m+1))\) is given in \cref{fig:homology_torus_two}.
    Each colored box stands for a shifted copy of \(\Z[G]\) and each
    \(\resizebox{0.379\width}{0.3755\height}{\usebox\OrangeBox}\) and 
    \(\raisebox{-0.014\textwidth}{\resizebox{0.379\width}{0.3755\height}{\usebox\bluepiece}}\)
    corresponds respectively to a pawn piece and a \(G\)-knight. More precisely, \(\BNr(T(2,2m+1))\) is chain homotopy equivalent to the sum of the following two subcomplexes:

    \begin{gather*}
        t^0\QGrad{q^{2m}}\Z[G],
        \\
        \bigoplus_{k=0}^{m-1} (t^{2k+2}\QGrad{q^{2m+4k+4}}\Z[G]\overset{G}{\longrightarrow}t^{2k+3}\QGrad{q^{2m+4k+6}}\Z[G]).
    \end{gather*}       
\end{proposition}
\begin{figure}[b]
    \centering
    \begin{tikzpicture}
        \node at (0.3,0){$\BNr(T(2,2m+1))$};
        \node[scale=.61*.93] at (0,2.3)
         {\usebox\HomologyTorusTwoNPlusOne};
    \end{tikzpicture}
    \caption{The chain complex of the positive torus knot $T(2,2m+1)$. Each colored box represents a copy of $\Z[G]$. The differentials are given by multiplication by $G$ on the blue pairs, that is in homological degrees $2k+2$ and $2k+3$, for $0\leq k\leq m-1$.}
    \label{fig:homology_torus_two}
\end{figure}

\begin{proof}
	This result is well-known; see for instance \cite[Section~2.1]{StableKhTorusKnots}.
	Regarding \(T(2,2m+1)\) as the 0-closure of the rational tangle \(Q_{-2m-1}\), 
	the complex $\BNr(T(2,2m+1))$ can be easily read off the complex \(\DD(Q_{-2m-1})\); see for instance \cite[Example~4.27]{KWZ}.
	Alternatively, one can argue that the knots \(T(2,2m+1)\) are alternating, so their Khovanov homology $C_{\Z}(T(2,2m+1))$ is thin and hence determined by the Jones polynomial and the signature. 
\end{proof}

\begin{proposition} \label{thm:pairs_torus_knots}
    Let \(K\) and \(J\) be two non-isotopic knots such that 
    \[
    \{ K,J \} 
    \subset 
    \{ T(2,2m+1)  \left|\, m \geq 0 \right.\} 
    \uplus 
    \{ T(3,3n+i)  \left|\, n \geq 1, i=1,2 \right.\}.
    \]
    Then
    \(\lambdau(K,J;\F_2)=1\) if and only if
    \begin{enumerate}[label=(\roman*)]
    	\item \label{item:pairs_torus_knots:1} \(\{ K,J \} = \{ T(2,2m+1), T(2,2m+3) \}\) for some \(m\geq 0\), or
    	\item \label{item:pairs_torus_knots:2} \(\{ K,J \} = \{ T(3,3n+1), T(3,3n+2) \}\) for some \(n\geq 1\), or
    	\item \label{item:pairs_torus_knots:3} \(\{ K,J \} \in \{\{ T(3,4), T(2,5) \}, \{ T(3,4), T(2,7) \},\)\\
    	\(\phantom{\{ K,J \} \in \{}	\{ T(3,5), T(2,7) \}, \{ T(3,5), T(2,9) \}\}\). 
    \end{enumerate}   
    Moreover, 
    \(\lambda(K,J;\F_2)=1\) if and only if
    \begin{enumerate}[label=(\roman*')]
    	\item \label{item:pairs_torus_knots:1'} \(\{ K,J \} = \{ T(2,2m+1), T(2,2m'+1) \}\) for some \(m'>m\geq 0\), or
    	\item \label{item:pairs_torus_knots:2'} \(\{ K,J \} = \{ T(3,3n+1), T(3,3n+2) \}\) for some \(n\geq 1\), or
    	\item \label{item:pairs_torus_knots:3'} \(\{ K,J \} \in \{\{ T(3,4), T(2,2m+1) \}, \{ T(3,5), T(2,2m+1) \}\}\)
    	for some \(m\geq 2\).
    \end{enumerate} 
\end{proposition}

\begin{proof}
    Let $K$ and $J$ be as in the statement of the proposition.
    Let us write 
    \[
    \mathcal{K}\coloneqq\BNr(K;\F_2[G])
    \quad
    \text{and}
    \quad
    \mathcal{J}\coloneqq\BNr(J;\F_2[G])
    \] 
    and likewise
    \[
    \lambda\coloneqq\lambda(K,J;\F_2)
    \quad
    \text{and}
    \quad
    \lambdau\coloneqq\lambdau(K,J;\F_2).
    \] 
    It is clear from \cref{prop:homology_torus,prop:homology_torus_two} that 
    \(\mathcal{K}\not\simeq\QGrad{q^\ell}\mathcal{J}\)
    for all $\ell\in\Z$, 
    and therefore $1\leq\lambda\leq\lambdau$. 
    Moreover, note that $\mathcal{K}$ and $\mathcal{J}$ decompose as a sum of pawns and generalised knights. 
    The former sit in homological degree 0. 
    Each of the latter is supported in homological degrees \(2j\) and \(2j+1\) for some \(j\geq1\), and in each degree, there is at most one piece. 

    Suppose first that both $K$ and $J$ are two-stranded torus knots: $K=T(2,2m+1)$ and $J=T(2,2m'+1)$, for some $m\ne m'$. The complexes $\mathcal{K}$ and $\mathcal{J}$ only contain pawns and $G$-knights.
    This implies that, $\mathfrak{u}(K;\F_2),\mathfrak{u}(J;\F_2)\leq 1$, and, by \cref{prop:upper_bds_on_lambda-_fields}, $\lambda\leq 1$.
    Furthermore, by \cref{thm:s_invariant_and_graded_lambda,prop:upper_bds_on_lambda-_fields}, since $\tfrac{1}{2}|s_{\F_2}(K,J)|=|m-m'|\geq 1$, one has
    \[
    \tfrac{1}{2}|s_{\F_2}(K,J)| 
    \leq 
    \lambdau 
    \leq 
    \max \{\tfrac{1}{2}|s_{\F_2}(K,J)|,\mathfrak{u}(K;\F_2),\mathfrak{u}(J;\F_2)\}
    =
    \tfrac{1}{2}|s_{\F_2}(K,J)|.
    \]
    Therefore, $\lambdau=\tfrac{1}{2}|s_{\F_2}(K,J)|=|m-m'|= 1$ if and only if $m'\in\{m-1,m+1\}$.

    Next, suppose that $K=T(3,3n+i)$ and $J=T(3,3n'+i')$, for some $n\geq n'\geq1$ and $i,i'\in\{1,2\}$. 
    If $n > n'$, $\mathcal{K}$ has a $G^2$-knight piece in homological degrees $4n$ and $4n+1$, where $\mathcal{J}$ is trivial. It follows from \cref{prop:G_torsion_and_lambda} that 
    \begin{equation}\label{eq:torsion_difference_G^2_knight}
    2-0=|\mathfrak{u}_{4n+1}(K;\F_2)-\mathfrak{u}_{4n+1}(J;\F_2)|\leq \lambda \leq \lambdau.
    \end{equation}
    Now let $n=n'$. In that case $i\ne i'$, since $K$ and $J$ are not isotopic. 
    Without loss of generality, we set $i=2$ and $i'=1$. 
    Let us find maps $f\colon \mathcal{K}\to \mathcal{J}$ and $g\colon \mathcal{J}\to \mathcal{K}$ realising $\lambda\leq 1$ and $\lambdau\leq 1$. 
    The complexes $\mathcal{K}$ and $\mathcal{J}$ decompose as a sum of pawns, $G$-knights and $G^2$-knights, as shown in \cref{fig:homology_torus}. 
    On the pawn pieces, we set $f=1$ and $g=G$.    
    The $G$-knight pieces of $\mathcal{K}$ and $\mathcal{J}$ are supported in homological degrees $4k+2$ and $4k+3$, for $0\leq k\leq n$ and $0\leq k\leq n-1$ respectively. We set $f$ and $g$ to be the zero maps on these pieces.
    Lastly, the $G^2$-knights of both complexes are supported in homological degrees $4k+4$ and $4k+5$, with $0\leq k\leq n-1$. On these pieces, we define $f$ and $g$ as follows:
    \[
    \begin{tikzcd}
        t^{4k+4}\QGrad{q^{6(n+k)+8}}\F_2[G]
        \arrow{r}{G^2}
        \arrow[bend left=10]{d}{f=1}
    	\arrow[bend right=10,leftarrow,swap]{d}{g=G}
        &
        t^{4k+5}\QGrad{q^{6(n+k)+12}}\F_2[G]
         \arrow[bend left=10]{d}{f=1}
    	\arrow[bend right=10,leftarrow,swap]{d}{g=G}
        \\
        t^{4k+4}\QGrad{q^{6(n+k)+6}}\F_2[G]
        \arrow{r}{G^2}
        &
        t^{4k+5}\QGrad{q^{6(n+k)+10}}\F_2[G]
    \end{tikzcd}
    \]    
    The maps $f$ and $g$ are homogeneous chain maps whose compositions are homotopic to multiplication by $G$. Moreover, $\QGrad{q}(f)=-2<0$ and $\QGrad{q}(g)=0$, therefore $\lambda\leq \lambdau\leq 1$.

    The last case we need to consider is $K=T(3,3n+i)$ and $J=(2,2m+1)$, for some $n\geq 1$, $m\geq 0$ and $i\in\{1,2\}$.
    Suppose that $\lambda=1$ and consider maps $f$ and $g$ realising it. 
    Note that $m\geq 2n$; 
    otherwise $\lambda\geq 2$ 
    by \cref{eq:torsion_difference_G^2_knight}.
    The complex $\mathcal{J}$ decomposes as the sum of a pawn and some $G$-knights, 
    and the latter are supported in homological degrees $2h+2$ and $2h+3$, 
    with $0\leq h \leq m-1$. 
    On the other hand, as mentioned above, 
    the $G^2$-knights of $\mathcal{K}$ are in homological degrees $4k+4$ and $4k+5$, 
    with $0\leq k\leq n-1$. 
    These degrees coincide for \(h=2k+1\).
    Let $f_j$ and $g_j$ denote the components of $f$ and $g$ between the $j^\text{th}$ chain modules. 
    For \(0\leq k\leq n-1\), we have:
    \[
    \begin{tikzcd}
        t^{4k+4}\QGrad{q^{6(n+k)+4+2i}}\F_2[G]
        \arrow{r}{G^2}
        \arrow[bend left=10]{d}{f_{4k+4}}
    	\arrow[bend right=10,leftarrow,swap]{d}{g_{4k+4}}
        &
        t^{4k+5}\QGrad{q^{6(n+k)+8+2i}}\F_2[G]
         \arrow[bend left=10]{d}{f_{4k+5}}
    	\arrow[bend right=10,leftarrow,swap]{d}{g_{4k+5}}
        \\
        t^{4k+4}\QGrad{q^{2m+8k+8}}\F_2[G]
        \arrow{r}{G}
        &
        t^{4k+5}\QGrad{q^{2m+8k+10}}\F_2[G]
    \end{tikzcd}
    \]
    Since \(f\) and \(g\) are homogeneous, the components must be of the form \(a\cdot G^\ell\) for some \(a\in\F_2\) and \(\ell\in\Z^{\geq0}\). 
    Moreover, \(G^2 f_{4k+5}=G f_{4k+4}\) and \(G g_{4k+5}=G^2 g_{4k+4}\), since \(f\) and \(g\) are chain maps. 
    Also, $G\cdot \id$ on the $G^2$-knight piece is the only morphism in its (\(\QGrad{q}\)-homogeneous) chain homotopy class, so the assumption \(g\circ f\simeq G\cdot \id\) in fact forces \(g\circ f= G\cdot \id\). 
    These conditions combined imply that 
    \[
    g_{4k+4}=1,\quad
    f_{4k+4}=G,\quad
    g_{4k+5}=G,\quad
    f_{4k+5}=1.
    \]
    Note that $f\circ g\simeq G\cdot \id$ is then still satisfied.
    Since $f$ is homogeneous, it holds for all $0\leq k\leq n-1$ that
    \begin{equation}\label{eq:qdeg_f_4k+4}
    \QGrad{q}(f)=\QGrad{q}(f_{4k+4})=-6n+2(m+k-i+1).
    \end{equation}
    This means that $n=1$, i.e.\ $K=T(3,4)$ or $K=T(3,5)$. 
    Moreover, one can check that, in this case, the maps $f$ and $g$ given by 
    \begin{gather} \label{eq:f_torus_2_and_3}
        f_0=1, \quad f_4=G, \quad f_5=1, \quad f_j=0 \text{ for } j \ne 0,4,5,
        \\ \label{eq:g_torus_2_and_3}
        g_0=G, \quad g_4=1, \quad g_5=G, \quad g_j=0 \text{ for } j \ne 0,4,5
    \end{gather}
    realise $\lambda=1$.

    Let us turn to $\lambdau$. By the discussion above, if $\lambdau=1$ and $K$ and $J$ are respectively a 3-stranded torus knot and a 2-stranded torus knot, then $K\in\{T(3,4),T(3,5)\}$. 
    Consider first $K=T(3,4)$. We observe that $\tfrac{1}{2}|s_{\F_2}(K,J)|=|3-m|$, where $J=T(2,2m+1)$ and $m\geq 0$. \cref{thm:s_invariant_and_graded_lambda} implies that if $\lambdau=1$, then $m\in \{2,3,4\}$. 
    For any maps $f$ and $g$ realising $\lambdau=1$, the quantum degree of $f$ is equal to $2m-6$; this is determined by setting $n=1,\, k=0$ and $i=1$ in \cref{eq:qdeg_f_4k+4}.
    If $m=4$, one finds $\QGrad{q}(f) > 0$, which means that $\lambdau(T(3,4),T(2,9))\geq 2$. One can check that the maps $f$ and $g$ of \cref{eq:f_torus_2_and_3} and \cref{eq:g_torus_2_and_3} realise $\lambdau(T(3,4),T(2,5))=1=\lambdau(T(3,4),T(2,7))$.
    The fact that $\lambdau(T(3,5),T(2,2m+1))=1$ iff $m\in\{3,4\}$ follows by a similar argument.    
\end{proof}

\begin{remark}
	In view of \cref{rem:homology_torus:coeffs}, 
	\cref{thm:pairs_torus_knots} also holds, 
	with identical proof, 
	if we replace \(\F_2\) by \(\Q\). 
	Similarly, 
	if \cref{prop:homology_torus} is true over \(\Z\), 
	then so is \cref{thm:pairs_torus_knots}. 
\end{remark}

The pairs of torus knots at Gordian distance 1 have been completely determined in \cite[Corollary 1.3]{gordian_distance_1}. These pairs correspond exactly to those given in \cref{item:pairs_torus_knots:1}, \cref{item:pairs_torus_knots:2}, and \cref{item:pairs_torus_knots:3} of \cref{thm:pairs_torus_knots}, minus the pairs $\{T(3,4),T(2,7)\}$ and $\{T(3,5),T(2,9)\}$, which are not related by a crossing change. 
This shows that $\lambdau$ detects all pairs of two-stranded torus knots at Gordian distance~1, and all pairs of three-stranded torus knots at Gordian distance~1.

On the other hand, the $u_q$-distance has not been completely determined for 2-stranded and 3-stranded torus knots.
All knots in \cref{item:pairs_torus_knots:1'} of \cref{thm:pairs_torus_knots} are related by a proper rational tangle replacement $Q_a \leadsto Q_b$, for $a,b\in\Z$. The pairs of knots in \cref{item:pairs_torus_knots:2'} are even related by a crossing change. But we do not know which pairs of knots in \cref{item:pairs_torus_knots:3'} are related by a proper rational tangle replacement. For example, 
we observe that, for $\{K,J\}=\{T(3,5),T(2,5)\}$,
\[
1=u_q(K,J)=\lambda(K,J)<\lambdau(K,J)=u(K,J)=2.
\]

\begin{example}\label{exa:lambda-coefficients-matter}
For general torus knots $T(p,q)$, with $p$ and $q$ coprime integers, the complex $\BNr(T(p,q))$ does not split as the sum of pawns and generalised knights. 
Instead, more complicated pieces can appear, as shown in \cite[Example 3.18]{lambda1}
for the complex $\BNr(T(5,6))$.
As a consequence, in general, the $\lambda$- and $\lambdau$-distances of torus knots vary depending on the base ring considered. For instance, simple calculations show that
\begin{equation} \label{eq:small_lambda-_different_rings}
\begin{split}
1&=\lambdau(T(5,6),T(3,11);\Q)\\&
<\lambdau(T(5,6),T(3,11);\F_5)=\lambdau(T(5,6),T(3,11))=2.
\end{split}
\end{equation}
Observe that $s_{\F}(T(5,6),T(3,11))=0$ for all fields $\F$. Therefore, \cref{eq:small_lambda-_different_rings} shows that, for the knots $T(5,6)$ and $T(3,11)$, the $\lambda$-distance over any $\F$ is strictly sharper than $s_\F$ as a lower bound for the Gordian distance.
In \cref{tab:lambda_and_s_torus} below, we give a (non-exhaustive) list of torus knots $K,J$ for which $\tfrac{1}{2}s_{\Q}(K,J)<\lambdau(K,J;\Q)$ and $\lambdau(K,J;\Q)>1$.
The corresponding complexes $\BNr$ were computed using the programs~\cite{khoca,homca};
given the complexes, $\lambdau$ was computed manually, similarly as 
in the proof of \cref{thm:pairs_torus_knots}.

\begin{table}[!ht]
    \centering
    \begin{tabular}{cccc}
        \toprule
        $K$ & $J$ & $\tfrac{1}{2}s_{\Q}(K,J)$ & $\lambdau(K,J;\Q)$ \\
        \hline
        $T(3,16)$ & $T(6,7)$ & 0 & 2 \\
        $T(3,17)$ & $T(5,9)$ & 0 & 2 \\
        $T(2,11)$ & $T(3,7)$ & 1 & 2 \\
        $T(2,11)$ & $T(4,5)$ & 1 & 2 \\
        $T(2,13)$ & $T(3,8)$ & 1 & 2 \\
        $T(3,14)$ & $T(5,8)$ & 1 & 2 \\
        $T(3,16)$ & $T(5,8)$ & 1 & 2 \\
        \bottomrule
    \end{tabular}
    \caption{Some pairs of torus knots $K,J$ for which $\tfrac{1}{2}|s_{\Q}(K,J)|<\lambdau(K,J;\Q)$.}
    \label{tab:lambda_and_s_torus}
\end{table}
\end{example}
\begin{remark}
This section has provided us with pairs of knots $K$, $J$ such that $\lambdau(K,J)$
is strictly greater than $\tfrac{1}{2}|s_{\mathbb{F}}(K,J)|$ and
than $\mathfrak{u}(K, J;\mathbb{F})$ for all $\mathbb{F}$:
Take for example $K = T(3,3a+1)$, $J = T(2,6a+1)$ for any $a \in \{2,\ldots, 50\}$ (conjecturally any $a\geq 2$),
then $s_{\mathbb{F}}(K,J) = 0$, $\mathfrak{u}(K, J;\mathbb{F}) = 1$ and
$\lambdau(K,J) = \lambdau(K,J;\mathbb{F}) = 2$ for all fields $\mathbb{F}$.

The question whether there exists such a pair of knots with $J$ the unknot remains open.
More succinctly,
is there a knot $K$ with $\lambdau(K) > \tfrac{1}{2}|s_{\mathbb{F}}(K)|$ and $\lambdau(K) > \mathfrak{u}(K;\mathbb{F})$ for all fields~$\mathbb{F}$?
\cref{prop:upper_bds_on_lambda-_fields}(ii) tells us that such a knot $K$ does not exist with $\lambdau(K; \mathbb{F})$ in lieu of $\lambdau(K)$.
However, working with $\lambdau(K) = \lambdau(K; \mathbb{Z})$, the authors are not aware of an algebraic reason that such a knot $K$ should not exist.
\end{remark}

\section{\texorpdfstring{A generalisation of $\lambda$ and $\lambdau$}{A generalisation of λ and λ⁻}}\label{sec:Lambda_set}

In this section, we push the techniques from \cref{sec:small_lambda} to a limit
and introduce a new invariant $\Lambda$,
which takes values in subsets of $\mathbb{Z}^2$ and subsumes $\lambda$ and $\lambdau$.
Unlike the latter, $\Lambda$ is sensitive to the signs of crossing changes,
or more generally the slopes of proper rational tangle replacements.

\subsection{Definition and basic properties}

\begin{definition}
	Given a unital commutative ring $R$ and two knots \(K_1\) and \(K_2\), define
	\[
	\Lambda(K_1,K_2;R)
	=
	\left\{%
	\begin{minipage}{2cm}\centering
		\((q_1,q_2)\in\Z^2\)
		\\
		\text{ with }
		\\
		\(q_1+q_2\geq0\)
	\end{minipage}
	\left|\,
    \begin{minipage}{7cm}
    \centering
	\(
	\exists
    \)
    homogeneous chain maps
    \\
    \(
		\begin{tikzcd}
			\BNr(K_1;R[G])
			\arrow[bend left=5]{r}{f}
			\arrow[bend right=5,leftarrow,swap]{r}{g}
			&
			\BNr(K_2;R[G]):
		\end{tikzcd}
	\)
    \\
    \(
	f\circ g\simeq G^{q_1+q_2}\cdot \id, \,
	g\circ f\simeq G^{q_1+q_2}\cdot \id
	\)
    \\
	and
    \(\QGrad{q}(f)=-2q_1, \QGrad{q}(g)=-2q_2\)
	\end{minipage}
	\right.
	\right\}
	\]
\end{definition}

As in \cref{def:lambda}, we set $\Lambda(K_1,K_2)=\Lambda(K_1,K_2;\Z)$ and $\Lambda(K;R)=\Lambda(K,U;R)$ where $K$ is a knot. Furthermore, observe that, for any ring~$R$, 
\[
\Lambda(K_1,K_2) \subseteq \Lambda(K_1,K_2;R).
\]

\begin{definition}\label{def:V}
	For \(n\geq0\) and \((n_1,n_2)\in\Z^2\), let 
	\begin{align*}
		\mathsf{V}_n
		&\coloneqq
		\{(q_1,q_2)\mid q_1+q_2\geq n\},
		\\
		\mathsf{V}(n_1,n_2)
		&\coloneqq
		\{(q_1,q_2)\mid q_1\geq n_1,q_2\geq n_2\}.
	\end{align*}
\end{definition}
These two types of subsets of $\Z^2$ will be useful to describe $\Lambda(K_1,K_2)$.
For example, note that
for all $K_1, K_2$ we have
$\Lambda(K_1,K_2)\subset \mathsf{V}_0$ 
by definition.

We like to illustrate subsets of \(\Z^2\) by shading areas of the Cartesian plane~$\mathbb{R}^2$. 
More precisely, given a subset $\Lambda\subset \Z^2$ and a point~$(x,y)\in\Z^2$, we shade the unit square centred around \((x,y)\) if and only if~$(x,y)\in\Lambda$.
For an example, see \cref{fig:Lambda_dummy}. 

 \begin{figure}[hb]
	\centering
	\begin{tikzpicture}[scale=.80]
		\fill[lightred] (0,2) |- (1,1) |- (2,0) |- (3,-1) |- (4,-2) |- (5,-3) |- (-1,3) |- cycle ;
		\draw[step=1cm,gray,very thin,xshift=0.5cm,yshift=0.5cm] (-1.25,-3.25) grid (4.25,2.25);    
		\draw[thick,->] (-1,0.5) -- (5.5,0.5);
		\draw[thick,->] (0.5,-3) -- (0.5,3.5);
		\draw[thick,red] (-1,2) -- (0,2) |- (1,1) |- (2,0) |- (3,-1) |- (4,-2) |- (4,-3);
		\node at (-0.5,2.5)[red,circle,fill,inner sep=1.25pt]{};
		\node at (0.5,2.5)[red,circle,fill, inner sep=1.25pt]{};
		\node at (1.5,2.5)[red,circle,fill, inner sep=1.25pt]{};
		\node at (2.5,2.5)[red,circle,fill, inner sep=1.25pt]{};
		\node at (3.5,2.5)[red,circle,fill, inner sep=1.25pt]{};
		\node at (4.5,2.5)[red,circle,fill, inner sep=1.25pt]{};
		\node at (0.5,1.5)[red,circle,fill, inner sep=1.25pt]{};
		\node at (1.5,1.5)[red,circle,fill, inner sep=1.25pt]{};
		\node at (2.5,1.5)[red,circle,fill, inner sep=1.25pt]{};
		\node at (3.5,1.5)[red,circle,fill, inner sep=1.25pt]{};
		\node at (4.5,1.5)[red,circle,fill, inner sep=1.25pt]{};
		\node at (1.5,0.5)[red,circle,fill, inner sep=1.25pt]{};
		\node at (2.5,0.5)[red,circle,fill, inner sep=1.25pt]{};
		\node at (3.5,0.5)[red,circle,fill, inner sep=1.25pt]{};
		\node at (4.5,0.5)[red,circle,fill, inner sep=1.25pt]{};
		\node at (2.5,-0.5)[red,circle,fill,inner sep=1.25pt]{};
		\node at (3.5,-0.5)[red,circle,fill,inner sep=1.25pt]{};
		\node at (4.5,-0.5)[red,circle,fill,inner sep=1.25pt]{};
		\node at (3.5,-1.5)[red,circle,fill,inner sep=1.25pt]{};
		\node at (4.5,-1.5)[red,circle,fill,inner sep=1.25pt]{};
		\node at (4.5,-2.5)[red,circle,fill,inner sep=1.25pt]{};
		\node [scale=1] at (3,1) [fill,red!30!white] {\textcolor{red}{$\mathsf{V}_{1}$}};
	\end{tikzpicture}
	\qquad
	\begin{tikzpicture}[scale=.80]
		\fill[lightblue] (1,-1) |- (5,3) |- cycle ;
		\draw[step=1cm,gray,very thin,xshift=0.5cm,yshift=0.5cm] (-1.25,-3.25) grid (4.25,2.25);    
		\draw[thick,->] (-1,0.5) -- (5.5,0.5);
		\draw[thick,->] (0.5,-3) -- (0.5,3.5);
		\draw[thick,blue] (1,3) |- (5,-1);
		\node at (1.5,-0.5)[blue,circle,fill, inner sep=1.25pt]{};
		\node at (2.5,-0.5)[blue,circle,fill, inner sep=1.25pt]{};
		\node at (3.5,-0.5)[blue,circle,fill, inner sep=1.25pt]{};
		\node at (4.5,-0.5)[blue,circle,fill, inner sep=1.25pt]{};
		\node at (1.5,0.5)[blue,circle,fill, inner sep=1.25pt]{};
		\node at (2.5,0.5)[blue,circle,fill, inner sep=1.25pt]{};
		\node at (3.5,0.5)[blue,circle,fill, inner sep=1.25pt]{};
		\node at (4.5,0.5)[blue,circle,fill, inner sep=1.25pt]{};
		\node at (1.5,1.5)[blue,circle,fill, inner sep=1.25pt]{};
		\node at (2.5,1.5)[blue,circle,fill, inner sep=1.25pt]{};
		\node at (3.5,1.5)[blue,circle,fill, inner sep=1.25pt]{};
		\node at (4.5,1.5)[blue,circle,fill, inner sep=1.25pt]{};
		\node at (1.5,2.5)[blue,circle,fill, inner sep=1.25pt]{};
		\node at (2.5,2.5)[blue,circle,fill, inner sep=1.25pt]{};
		\node at (3.5,2.5)[blue,circle,fill, inner sep=1.25pt]{};
		\node at (4.5,2.5)[blue,circle,fill, inner sep=1.25pt]{};
		\node [scale=1] at (3,1) [fill,blue!30!white] {\textcolor{blue}{$\mathsf{V}(1,-1)$}};
	\end{tikzpicture}
	\caption{Illustration of \cref{def:V}.}
	\label{fig:Lambda_dummy}
\end{figure}

\begin{remark}\label{rem:lambda_from_Lambda}
    The set $\Lambda(K_1,K_2;R)$ contains all the information provided by $\lambda$ and~$\lambdau$. Indeed, it follows directly from the definition that
    \begin{enumerate}[label=(\roman*)]
        \item \label{item:rem:lambda_from_Lambda:lambda}
        \(
		\lambda(K_1,K_2;R)
		=
		\min\{q_1+q_2\mid (q_1,q_2)\in\Lambda(K_1,K_2;R)\}
		\).
        \item \label{item:rem:lambda_from_Lambda:lambda-}
		\(
		\lambdau(K_1,K_2;R)
		=
		\min\{q_1+q_2\mid (q_1,q_2)\in\Lambda(K_1,K_2;R)\cap\mathsf{V}(0,0)\}
		\).
    \end{enumerate}
\end{remark}

Below, we state some properties of~$\Lambda$.

\begin{proposition}\label{prop:Lambda-basic-properties}
	For \(K_1,K_2\), and \(K_3\) three knots, the following is true:
	\begin{enumerate}[label=(\roman*)]
	\item \label{item:Lambda-basic-properties:1}
		\(
		\Lambda(K_2,K_1;R)
		=
		\{(q_2,q_1)\mid (q_1,q_2)\in\Lambda(K_1,K_2;R)\}
		\).
        \item \label{item:Lambda-basic-properties:2} For any \(q_1,q_2\in\Z\),
		\(
		(q_1,q_2)\in\Lambda(K_1,K_2;R)
		\Rightarrow
		\mathsf{V}(q_1,q_2)\subseteq\Lambda(K_1,K_2;R)
		\). 
        \item \label{item:Lambda-basic-properties:3} For any~\(\ell\in \Z\),
        \[
        \Lambda(K_1,K_2;R) \ni (\ell,-\ell) 
        \iff 
        \BNr(K_1;R[G]) \simeq \QGrad{q^{2\ell}}\BNr(K_2;R[G]).
        \]
        \item \label{item:Lambda-basic-properties:4}
        \(\Lambda(K_1,K_2;R) + \Lambda(K_2,K_3;R) \subseteq \Lambda(K_1,K_3;R)\).
	\end{enumerate}
\end{proposition}

\begin{proof}
    The first property follows immediately from the definitions.
	For the second, observe that 
	if \(f\) and \(g\) are maps realising \((q_1,q_2)\), then for any \(a,b\geq0\),
	\(G^a f\) and \(G^b g\) realise \((q_1+a,q_2+b)\).
 
	The implication $\Leftarrow$ of property \cref{item:Lambda-basic-properties:3} is obvious, as $(\ell,-\ell)$ is realised by the maps $f=\id$ and~$g=\id$. The other direction can be seen as follows. Suppose $\Lambda(K_1,K_2;R) \ni (\ell,-\ell)$. Then there exist homogeneous chain maps $f$ and $g$ with $\QGrad{q}(f) = -2\ell,\, \QGrad{q}(g)=2\ell$ such that $f \circ g \simeq \id$ and $g \circ f \simeq \id$, which proves the statement.

    For property \cref{item:Lambda-basic-properties:4}, let $(a_1,b_1)\in \Lambda(K_1,K_2;R)$ and $(a_2,b_2)\in \Lambda(K_2,K_3;R)$, and consider maps $f_i$ and $g_i$ realising $(a_i,b_i)$ for~$i=1,2$. Then the maps $f_2\circ f_1$ and $g_1\circ g_2$ realise $(a_1+a_2,b_1+b_2)\in \Lambda(K_1,K_3;R)$.
\end{proof}

\begin{proposition} \label{prop:Lambda_mirror}
    Let \(K_1\) and \(K_2\) be two knots. 
    For \(-K_1\) and \(-K_2\) the mirror images of \(K_1\) and~\(K_2\), one has 
    \[
    \Lambda(-K_2,-K_1;R)=\Lambda(K_1,K_2;R).
    \]
\end{proposition}

\begin{proof}
    The complex $\BNr(-K_i;R[G])$, for~$i=1,2$, is chain homotopy equivalent to $-\BNr(K_i;R[G])$, by \cref{lem:mirror_complex}. Let $(a,b)\in \Lambda(K_1,K_2;R)$ and consider maps
    \[
	\begin{tikzcd}
		\BNr(K_1;R[G])
		\arrow[bend left=5]{r}{f}
		\arrow[leftarrow, bend right=5, swap]{r}{g}
		&
		\BNr(K_2;R[G])
		\end{tikzcd}
		\]
    realising it.
    The maps $f$ and $g$ induce maps
    \[
		\begin{tikzcd}
		-\BNr(K_1;R[G])
		\arrow[leftarrow, bend left=5]{r}{\overline{f}}
		\arrow[bend right=5, swap]{r}{\overline{g}}
		&
		-\BNr(K_2;R[G]).
		\end{tikzcd}
		\]
    The homological and quantum gradings in $-\BNr(K_i;R[G])$ are the opposites of the gradings in $\BNr(K_i;R[G])$. But the maps $\overline{f}$ and $\overline{g}$ also go in the opposite direction to $f$ and~$g$. Hence, $\QGrad{q}(\overline{f})=\QGrad{q}(f)$ and $\QGrad{q}(\overline{g})=\QGrad{q}(g)$. It follows that $(a,b)\in \Lambda(-K_2,-K_1;R)$.
\end{proof}

\subsection{Relationship with proper rational tangle replacements}

The main result of this subsection is the following.

\begin{theorem}\label{thm:obstruction_rational_repl_other_cr}
	Let \(K_1\) and \(K_2\) be knots and suppose that \(K_2\) can be obtained from \(K_1\) by the proper rational tangle replacement
	\[
	\textnormal{$\vcenter{\hbox{\def\svgwidth{11.6pt}}}$} \leadsto Q_{\nicefrac{p}{q}},
	\]
	where \(p\) and \(q\) are odd coprime integers and \(\nicefrac{p}{q} \ne -1\).
	Then \((1-\alpha,\alpha) \in \Lambda(K_1,K_2)\), where
	\[
	\alpha=\alpha(p,q) \coloneqq \tfrac{1}{2} s_{\Q}(Q_{\nicefrac{p}{q}}(0)) + \begin{cases}
		1 & \textrm{if }\; -1<\nicefrac{p}{q}<0,
		\\
		0 & \textrm{else}.
	\end{cases}
	\]
\end{theorem}

See \cref{rem:s_of_2-bridge_knots} for a formula expressing \(s_{\Q}(Q_{\nicefrac{p}{q}}(0))\) in terms of $p$ and~$q$.
Before proving \cref{thm:obstruction_rational_repl_other_cr}, 
we first observe some of its immediate consequences and discuss some examples. 

\begin{corollary}\label{cor:obstruction_rational_repl}
	Let \(K_1\) and \(K_2\) be knots and suppose that \(K_2\) can be obtained from \(K_1\) by the proper rational tangle replacement
	\[
	\textnormal{$\vcenter{\hbox{\def\svgwidth{11.6pt}}}$} \leadsto Q_{\nicefrac{p}{q}}
	\]
	where \(p\) and \(q\) are odd coprime integers and \(\nicefrac{p}{q} \ne 1\).
	Then \((\beta,1-\beta) \in \Lambda(K_1,K_2)\), where
	\[
	\beta 
	= 
	-\tfrac{1}{2} s_{\Q}(Q_{\nicefrac{p}{q}}(0)) 
	+ 
	\begin{cases}
		1 & \textrm{if }\; 0<\nicefrac{p}{q}<1,
		\\
		0 & \textrm{else}.
	\end{cases}
	\]
\end{corollary}

\begin{proof}
    If the knots $K_1$ and $K_2$ are related by the tangle replacement $\textnormal{$\vcenter{\hbox{\def\svgwidth{11.6pt}}}$} \leadsto Q_{\nicefrac{p}{q}}$, then their mirror images $-K_1$ and $-K_2$ are related by the tangle replacement $\textnormal{$\vcenter{\hbox{\def\svgwidth{11.6pt}}}$} \leadsto -Q_{\nicefrac{p}{q}}=Q_{\nicefrac{-p}{q}}$. 
    By \cref{thm:obstruction_rational_repl_other_cr}, $(1-\alpha(-p,q),\alpha(-p,q))\in \Lambda(-K_1,-K_2)$. One computes
    \begin{align*}
    \alpha(-p,q) & 
    = 
    \tfrac{1}{2} s_{\Q}(Q_{\nicefrac{-p}{q}}(0)) 
    + 
    \begin{cases}
        1 & \textrm{if }\; -1<-\nicefrac{p}{q}<0
        \\
        0 & \textrm{else}
    \end{cases} \\
    & =
    -\tfrac{1}{2} s_{\Q}(Q_{\nicefrac{p}{q}}(0)) 
    + 
    \begin{cases}
        1 & \textrm{if }\; 0<\nicefrac{p}{q}<1
        \\
        0 & \textrm{else}
    \end{cases} = \beta.
    \end{align*}
    Using \cref{prop:Lambda_mirror}, one has that $(1-\beta,\beta)\in \Lambda(K_2,K_1)$, which implies that $(\beta,1-\beta)\in \Lambda(K_1,K_2)$ by \cref{prop:Lambda-basic-properties}.
\end{proof}

\begin{corollary}\label{cor:obstruction_cr_change}
    Let \(K_-\) and \(K_+\) be knots and suppose that \(K_+\) is obtained from \(K_-\) by changing a negative crossing to a positive crossing. Then \((0,1) \in \Lambda(K_-,K_+)\).
    \qed
\end{corollary}

\cref{thm:obstruction_rational_repl_other_cr} also implies a stronger version of \cref{thm:lambda-gordian}. To formulate it, we come back to the ``very simple'' rational tangle replacements discussed in \cref{ref:simple}.
\begin{definition}\label{def:u0}
Let $u^0(K, J)$ be the minimal number of proper rational tangle replacements relating 
two knots $K, J$, where each replacement is of the form $\textnormal{$\vcenter{\hbox{\def\svgwidth{11.6pt}}}$} \rightsquigarrow Q_{\nicefrac{p}{q}}$
such that~$\alpha(p,q)\in\{0,1\}$, where \(\alpha\) is the function defined in \cref{thm:obstruction_rational_repl_other_cr}.
\end{definition}
Clearly, we have
\[
u_q(K, J) \leq u^0(K, J) \leq u(K,J).
\]%
The replacements allowed in \cref{def:u0}  include:
\begin{itemize}
\item Crossing changes.
\item For $k\in\Z\setminus\{0\}$, tying $2k$ crossings into two oppositely oriented strands, which is equivalent to a replacement $\textnormal{$\vcenter{\hbox{\def\svgwidth{11.6pt}}}$} \rightsquigarrow Q_{\nicefrac{1}{2k-1}}$. This is called a \(\overline{t}_{2k}\)-move by Przytycki~\cite{zbMATH04088518}.
\item Many further replacements like $Q_{-1} \rightsquigarrow
Q_{\nicefrac{-3}{5}}, 
Q_{\nicefrac{3}{5}}, 
Q_{\nicefrac{-7}{9}}, 
Q_{\nicefrac{-7}{5}}, 
Q_{\nicefrac{-7}{3}}, 
Q_{\nicefrac{7}{9}}$, to mention just a few. 
\end{itemize}
\begin{corollary}\label{cor:strongerthm1}
For any two knots \(K\) and \(J\),
\[
\lambdau(K,J) \leq u^0(K,J).
\]
Combining this with \cref{thm:s_invariant_and_graded_lambda} implies for any field \(\mathbb{F}\) that
\[
\tfrac{1}{2}|s_{\mathbb{F}}(K,J)|
\leq 
u^0(K,J).
\myqed
\]
\end{corollary}

\begin{example}\label{ex:Lambda_trefoil}
    For illustration, we compute $\Lambda(T(2,3))=\Lambda(T(2,3),U;\Z)$.
    The chain complex $\BNr(T(2,3))$ splits as the sum of a pawn piece and a $G$-knight piece, while $\BNr(U)$ consists of a single pawn piece: 
    \[
    \BNr(T(2,3)) \simeq \;t^0\QGrad{q^2}\Z[G]\; \oplus \;t^2\QGrad{q^6}\Z[G] \overset{G}{\longrightarrow} t^3\QGrad{q^8}\Z[G]\;,
    \qquad
    \BNr(U) \simeq \;t^0\QGrad{q^0}\Z[G].
    \]
    We can split the computation of $\Lambda$ by homological degree:
    \[
    \Lambda(T(2,3),U)=\Lambda(t^0\QGrad{q^2}\Z[G] \;,\; t^0\QGrad{q^0}\Z[G]) \;\cap\; \Lambda(t^2\QGrad{q^6}\Z[G] \overset{G}{\longrightarrow} t^3\QGrad{q^8}\Z[G] \;,\; 0).
    \]
We claim that $\Lambda(t^0\QGrad{q^2}\Z[G] \;,\; t^0\QGrad{q^0}\Z[G]) = \mathsf{V}(1,-1)$. 
        This can be easily checked by looking at the quantum degrees of any maps $f$ and $g$ between $t^0\QGrad{q^2}\Z[G]$ and $t^0\QGrad{q^0}\Z[G]$: One finds
        \[
        \QGrad{q}(f) \leq 0-2=-2 \qquad \text{and} \qquad \QGrad{q}(g) \leq 2-0=2,
        \]
        which implies that $\Lambda(t^0\QGrad{q^2}\Z[G] \;,\;t^0\QGrad{q^0}\Z[G]) \subseteq \mathsf{V}(1,-1)$. 
        Moreover, the identity maps realise the fact that $(1,-1)\in\Lambda(t^0\QGrad{q^2}\Z[G] \;,\;t^0\QGrad{q^0}\Z[G])$.

Next, we claim that \[\Lambda(t^2\QGrad{q^6}\Z[G] \overset{G}{\longrightarrow} t^3\QGrad{q^8}\Z[G] \;,\; 0)=\mathsf{V}_1.\]
        Indeed, any maps $f$ and $g$ between $t^2\QGrad{q^6}\Z[G] \overset{G}{\longrightarrow} t^3\QGrad{q^8}\Z[G]$ and $0$ satisfy~$g \circ f =0$. In the first complex, $0 \simeq G^k \cdot \id$ if and only if~$k\geq 1$. Therefore one must have 
        \[
        \QGrad{q}(g)+\QGrad{q}(f) = \QGrad{q}(g \circ f) = \QGrad{q}(G^k \cdot \id) = -2k 
        \qquad \text{for some } k\geq 1.
        \]
        Hence $\Lambda(t^2\QGrad{q^6}\Z[G] \overset{G}{\longrightarrow} t^3\QGrad{q^8}\Z[G] \;,\; 0)\subseteq\mathsf{V}_1$. 
        The opposite inclusion comes from the fact that the maps $f=0$ and $g=0$ satisfy both conditions above. 
    Overall, we have shown $\Lambda(T(2,3))=\mathsf{V}(1,-1) \cap \mathsf{V}_1$, as illustrated in \cref{fig:Lambda_trefoil}. 
    In the second point of \cref{prop:torsion_and_biglambda}, we will generalise this result. 

    \begin{figure}[t]
        \centering
        \begin{tikzpicture}[scale=.80]
        \fill[lightviolet] (1,2.9) |- (2,0) |- (4.9,-1) |- (1,2.9);
				\draw[step=1cm,gray,very thin,xshift=0.5cm,yshift=0.5cm] (-1.25,-3.25) grid (4.25,2.25);    
        \draw[thick,->] (-1,0.5) -- (5.5,0.5);
        \draw[thick,->] (0.5,-3) -- (0.5,3.5);
        \draw[thick,blue] (1,2.9) |- (4.9,-1);
        \draw[thick,red] (-1,2) -- (0,2) |- (1,1) |- (2,0) |- (3,-1) |- (4,-2) |- (4,-2.9);
        \node [scale=1] at (3,1) [fill,lightviolet] {\textcolor{violet}{$\Lambda(T(2,3))$}};
        \node [scale=1, left, red] at (-1,2) {$\mathsf{V}_{1}$};
        \node [scale=1, right, blue] at (4.9,-1) {$\mathsf{V}(1,-1)$};
    \end{tikzpicture}
        \caption{The shaded region illustrates the set~$\Lambda(T(2,3))$. It is given by the intersection of $\mathsf{V}_{1}$ (bounded by the red path) and $\mathsf{V}(1,-1)$ (bounded by the blue path) from \cref{fig:Lambda_dummy}.}
        \label{fig:Lambda_trefoil}
    \end{figure}
    The geometric significance of $\Lambda(T(2,3))$ has already been discussed in the introduction. Let us now provide more details:
    \begin{itemize}
        \item $(0,1)\not\in \Lambda(T(2,3))$. By \cref{cor:obstruction_cr_change} this implies that there is no negative to positive crossing change relating $T(2,3)$ to the unknot.
        \item $(1,0)\in \Lambda(T(2,3))$. This agrees with the fact that there exists a positive to negative crossing change relating $T(2,3)$ to the unknot.
        \item $(2,-1)\in \Lambda(T(2,3))$. This agrees with the fact that there is a proper rational tangle replacement $\textnormal{$\vcenter{\hbox{\def\svgwidth{11.6pt}}}$} \leadsto Q_3$ relating $T(2,3)$ to the unknot (use \cref{thm:obstruction_rational_repl_other_cr} and the fact that $s_{\Q}(Q_{3}(0))=-2$, as $Q_3(0)$ is the negative trefoil knot).
    \end{itemize}

\end{example}
In \cref{exa:lambda-coefficients-matter}, we pointed out a pair of knots $K$ and $J$ with different $\lambdau$-distance over $\Q$ and~$\F_5$. The following example shows that the shape of $\Lambda$ over different rings $R$ and $R'$ may vary even when $\lambdau(K,J;R)=\lambdau(K,J;R')$.

\begin{example}
    Consider the knots $K=T(5,6)$ and~$J=T(3,11)$. 
    For all rings~$R$, the complex $\BNr(J;R[G])$ can be obtained from the right-hand side of \cref{fig:homology_torus}, by setting~$n=3$. 
    The complex $\BNr(K)$ is shown in \cite[Example 3.18]{lambda1} (there, the bigraded module $t^i\QGrad{q^\ell}\Z[G]$ is denoted by $_i\Z[G]\{\ell\}$).
    As described in \cref{fig:chain_complexes_related}, one can obtain from it, among others, the complexes~$\BNr(K;\F[G])$, for $\F=\Q$ and $\F_p$ with $p$ prime.
    Simple calculations for each of these complexes yield the sets $\Lambda(K,J)$ and $\Lambda(K,J;\F)$ of \cref{fig:big_Lambda_different_rings}.
    Note that, in this example, $\Lambda(K,J)$ is strictly contained in $\Lambda(K,J;\F)$ for all~$\F$, and one has
    \[
    \Lambda(K,J)=\bigcap_{\F} \Lambda(K,J;\F).
    \]
  \end{example}

    \begin{figure}[t]
    \centering
    \begin{tikzpicture}
        \node at (0,0){$\Lambda(K,J;\F')$};
        \node at (0,1.7)
         {
         \begin{tikzpicture}[scale=.50]
            \fill[lightviolet] (0,3.9) |- (2,1) |- (4.9,0) |- (0,3.9);
            \draw[step=1cm,gray,very thin,xshift=0.5cm,yshift=0.5cm] (-1.25,-1.25) grid (4.25,3.25);    
            \draw[thick,->] (-1,0.5) -- (5.5,0.5);
            \draw[thick,->] (0.5,-1) -- (0.5,4.5);
            \draw[thick,violet] (0,3.9) |- (2,1) |- (4.9,0);
        \end{tikzpicture}
         };
        \node at (3.7,0){$\Lambda(K,J;\F_2) = \Lambda(K,J;\F_3)$};
        \node at (3.7,1.7)
         {
         \begin{tikzpicture}[scale=.50]
            \fill[lightviolet] (0,3.9) |- (3,1) |- (4.9,0) |- (0,3.9);
            \draw[step=1cm,gray,very thin,xshift=0.5cm,yshift=0.5cm] (-1.25,-1.25) grid (4.25,3.25);
            \draw[thick,->] (-1,0.5) -- (5.5,0.5);
            \draw[thick,->] (0.5,-1) -- (0.5,4.5);
            \draw[thick,violet] (0,3.9) |- (3,1) |- (4.9,0);
        \end{tikzpicture}
         };
        \node at (7.4,0){$\Lambda(K,J;\F_5)$};
        \node at (7.4,1.7)
         {
         \begin{tikzpicture}[scale=.50]
            \fill[lightviolet] (0,3.9) |- (1,2) |- (2,1) |- (4.9,0) |- (0,3.9);
            \draw[step=1cm,gray,very thin,xshift=0.5cm,yshift=0.5cm] (-1.25,-1.25) grid (4.25,3.25);
            \draw[thick,->] (-1,0.5) -- (5.5,0.5);
            \draw[thick,->] (0.5,-1) -- (0.5,4.5);
            \draw[thick,violet] (0,3.9) |- (1,2) |- (2,1) |- (4.9,0);
        \end{tikzpicture}
         };
        \node at (11.1,0){$\Lambda(K,J)$};
        \node at (11.1,1.7)
         {
         \begin{tikzpicture}[scale=.50]
            \fill[lightviolet] (0,3.9) |- (1,2) |- (3,1) |- (4.9,0) |- (0,3.9);
            \draw[step=1cm,gray,very thin,xshift=0.5cm,yshift=0.5cm] (-1.25,-1.25) grid (4.25,3.25);
            \draw[thick,->] (-1,0.5) -- (5.5,0.5);
            \draw[thick,->] (0.5,-1) -- (0.5,4.5);
            \draw[thick,violet] (0,3.9) |- (1,2) |- (3,1) |- (4.9,0);
        \end{tikzpicture}
         };
    \end{tikzpicture}
        \caption{Let $K=T(5,6)$ and~$J=T(3,11)$. The shaded regions indicate $\Lambda(K,J;R)$ over different rings~$R$. Here, $\F'$ denotes the fields $\Q$ and~$\F_{p'}$, where $p'$ is a prime different from 2, 3 and 5.}
        \label{fig:big_Lambda_different_rings}
    \end{figure}

We now turn to the proof of \cref{thm:obstruction_rational_repl_other_cr}. 
Imitating the proof of \cref{thm:lambda-gordian,thm:lambda_proper_rational_gordian} for \(\lambda\) in \cref{sec:small_lambda}, we first generalise \(\Lambda\) to Conway tangles. 

\begin{definition}
	Given \(a\in\{D_\bullet,S^2,D_\circ,G\}\subset\mathcal{B}\) and two Conway tangles \(T_1\) and \(T_2\), define
	\[
	\Lambda_a(T_1,T_2)
	=
	\left\{%
	\begin{minipage}{2cm}\centering
		\((q_1,q_2)\in\Z^2\)
		\\
		\text{ with }
		\\
		\(q_1+q_2\geq0\)
	\end{minipage}
	\left|\,
	\begin{minipage}{6.5cm}\centering
		\(
		\exists
        \)
		homogeneous chain maps 
        \\
		\(
		\begin{tikzcd}
			\DD(T_1)
			\arrow[bend left=5]{r}{f}
			\arrow[bend right=5,leftarrow,swap]{r}{g}
			&
			\DD(T_2):
		\end{tikzcd}
		\)
		\\
		\(f\circ g\simeq a^{q_1+q_2}\cdot \id\),
		\(g\circ f\simeq a^{q_1+q_2}\cdot \id\),
		\\
        and
		\(\QGrad{q}(f)=-2q_1, \QGrad{q}(g)=-2q_2\)
	\end{minipage}
	\right.
	\right\}
	\]
\end{definition}

\begin{remark}\label{rem:Lambda-basic-props-generalize}
    The properties of \cref{rem:lambda_from_Lambda} and \cref{prop:Lambda-basic-properties} still hold with tangles $T_1,\, T_2,\, T_3$ in lieu of knots $K_1,\, K_2,\,K_3$, and $\Lambda_a,\lambda_a,\lambdau_a$ in lieu of $\Lambda, \lambda,\lambdau$, for \(a\in\{D_\bullet,S^2,D_\circ,G\}\subset\mathcal{B}\).
\end{remark}

\begin{lemma}\label{lem:Lambda-very-silly-case}
	Let \(T_1\) and \(T_2\) be two Conway tangles with \(x\coloneqq \conn{T_1}=\conn{T_2}\). Then
    \[
	\Lambda_{a_x}(T_1,T_2)
	=
	\begin{cases*}
	\mathsf{V}_1\cup \mathsf{V}(\ell,-\ell) 
	& 
	if \(\DD(T_1) \simeq \QGrad{q^{2\ell}}\DD(T_2)\) for some~$\ell\in \Z$,
	\\
	\mathsf{V}_1 
	& 
	else.
	\end{cases*}
	\]
\end{lemma}

\begin{proof}
	By \cref{lem:nullhomotopic_basepoint_action_from_connectivity}, the maps $f=0$ and $g=0$ witness the fact that \(\Lambda_{a_x}(T_1,T_2)\supseteq\mathsf{V}_1\).
	By definition, \(\Lambda_{a_x}(T_1,T_2)\subseteq\mathsf{V}_0\). 
	Hence, it just remains to study for which \(\ell\in\Z\) the point \((\ell,-\ell)\) lies in \(\Lambda_{a_x}(T_1,T_2)\). 
	In view of \cref{rem:Lambda-basic-props-generalize}, for any \(\ell\in\Z\),
	\((\ell,-\ell)\in\Lambda_{a_x}(T_1,T_2)\) is equivalent to \(\DD(T_1)\simeq \QGrad{q^{2\ell}}\DD(T_2)\) by \cref{item:Lambda-basic-properties:3} of \cref{prop:Lambda-basic-properties}. 
	Since \(\DD(T_1)\not\simeq0\), there is at most one such \(\ell\). 
\end{proof}

\begin{lemma}\label{lem:Lamda-crossings}
	We have
	\[
	\Lambda_{a}(\textnormal{$\vcenter{\hbox{\def\svgwidth{11.6pt}}}$},\textnormal{$\vcenter{\hbox{\def\svgwidth{11.6pt}}}$})
	=
	\begin{cases*}
		\mathsf{V}(0,1) & if \(a=D_\bullet\),
		\\
		\mathsf{V}_1 & if \(a=S^2\),
		\\
		\varnothing & if \(a=D_\circ\).
	\end{cases*}
	\]
\end{lemma}

\begin{proof}
	The following maps \(f\) and \(g\) witness the fact that \((0,1)\in\Lambda_{D_\bullet}(\textnormal{$\vcenter{\hbox{\def\svgwidth{11.6pt}}}$},\textnormal{$\vcenter{\hbox{\def\svgwidth{11.6pt}}}$})\):
	\[
	\begin{tikzcd}
		\DD(\textnormal{$\vcenter{\hbox{\def\svgwidth{11.6pt}}}$})
		\arrow[bend left=10]{d}{f}
		\arrow[bend right=10,leftarrow,swap]{d}{g}
		&
		\GGzqh{\circ}{0}{-2}{-1}
		\arrow{r}{S}
		&
		\GGzqh{\bullet}{0}{-1}{0}
		\arrow[bend left=10]{d}{D_\bullet}
		\arrow[bend right=10,leftarrow,swap]{d}{1}
        &
		\\
		\DD(\textnormal{$\vcenter{\hbox{\def\svgwidth{11.6pt}}}$})
		&
		&
		\GGzqh{\bullet}{0}{1}{0}
        \arrow{r}{S}
        &
        \GGzqh{\circ}{0}{2}{1}
	\end{tikzcd}
	\]
	Up to multiplication by constants, the only other homogeneous homomorphisms that preserve homological grading are multiples of \(f\) by powers of \(D_\bullet\) and multiples of \(g\) by powers of \(D_\bullet\) or \(S^2\). 
	This proves the claim for \(a=D_\bullet\). 
 
	For \(a=S^2\), we can apply \cref{lem:Lambda-very-silly-case}. 
    The only potentially non-zero components of (homological grading preserving) morphisms between \(\DD(\textnormal{$\vcenter{\hbox{\def\svgwidth{11.6pt}}}$})\) and \(\DD(\textnormal{$\vcenter{\hbox{\def\svgwidth{11.6pt}}}$})\) are in homological degree~$0$. However, \(D_\circ^n\cdot \id_{\DD(\textnormal{$\vcenter{\hbox{\def\svgwidth{10pt}}}$})}\) and \(D_\circ^n\cdot \id_{\DD(\textnormal{$\vcenter{\hbox{\def\svgwidth{10pt}}}$})}\) are not null-homotopic for any \(n\geq0\). Therefore \(\Lambda_{D_\circ}(\textnormal{$\vcenter{\hbox{\def\svgwidth{11.6pt}}}$},\textnormal{$\vcenter{\hbox{\def\svgwidth{11.6pt}}}$})\) is the empty set. 
\end{proof}

\begin{lemma}\label{lem:Lambda_rational_repl}
	Let \(p\) and \(q\) be odd coprime integers such that \(\nicefrac{p}{q}\neq-1\).
	Then 
	\[
	\Lambda_{D_\bullet}(\textnormal{$\vcenter{\hbox{\def\svgwidth{11.6pt}}}$},Q_{\nicefrac{p}{q}})
	=
	\mathsf{V}(1-\alpha,\alpha),
	\]
	for $\alpha = \alpha(p,q)$ as in \cref{thm:obstruction_rational_repl_other_cr}, i.e.
	\[
	\alpha=\tfrac{1}{2}s_{\Q}(Q_{\nicefrac{p}{q}}(0))+
	\begin{cases*}
		1
		&
		if \(-1<\nicefrac{p}{q}<0\),
		\\
		0
		&
		else.
	\end{cases*}
	\] 
\end{lemma}

\begin{proof}
	We first show that the maps \(f\) and \(g\) constructed in the proof of \cref{lem:lambda_rational_tangles} and shown in \cref{tab:lambda_rational_tangles:maps_for_cases_p/q}
	are witnesses for 
	$(1-\alpha,\alpha)\in \Lambda_{D_\bullet}(\textnormal{$\vcenter{\hbox{\def\svgwidth{11.6pt}}}$},Q_{\nicefrac{p}{q}})$.
	We already know from the proof of 
	\cref{lem:lambda_rational_tangles}
	that both their compositions are homotopic to \(D_\bullet\cdot\id\). 
	It therefore suffices 
	to calculate the quantum gradings of \(f\) and \(g\). 
	We have 
	\[
	\QGrad{q}(g)
	=
	-\ell+
	\begin{cases*}
		-1
		&
		 if \(\nicefrac{p}{q}>-1\),
		\\
		+1
		&
		otherwise.
	\end{cases*}
	\]
	Moreover, by \cref{lem:grading_end_zz},
	\[
	\ell
	=
	s_\Q(Q_{\nicefrac{p}{q}}(0))
	+
	\begin{cases*}
		-1
		&
		if $\nicefrac{p}{q}>0$,
		\\
		+1
		&
		if $\nicefrac{p}{q}<0$.
	\end{cases*}
	\]
	Thus, 
	\(\QGrad{q}(g)=-2\alpha\). 
  Furthermore, 
  \(\QGrad{q}(f\circ g)=-2\), 
  so
  \(\QGrad{q}(f)=-2\cdot (1-\alpha)\). 
  This shows  
  $(1-\alpha,\alpha)\in \Lambda_{D_\bullet}(\textnormal{$\vcenter{\hbox{\def\svgwidth{11.6pt}}}$},Q_{\nicefrac{p}{q}})$
  and hence 
  \[
  \Lambda_{D_\bullet}(\textnormal{$\vcenter{\hbox{\def\svgwidth{11.6pt}}}$},Q_{\nicefrac{p}{q}})
  \supseteq
  \mathsf{V}(1-\alpha,\alpha),
  \]
  by 
  \cref{prop:Lambda-basic-properties}\cref{item:Lambda-basic-properties:2} and \cref{rem:Lambda-basic-props-generalize}.
	To see the other inclusion, 
	suppose $f'$ and $g'$ are maps of quantum grading 
	$\QGrad{q}(f')=-2q_1$ and 
	$\QGrad{q}(g')=-2q_2$ 
	satisfying 
	$f'\circ g' \simeq D_\bullet^{q_1+q_2} \cdot \id$ and 
	$g'\circ f' \simeq D_\bullet^{q_1+q_2} \cdot \id$. 
	By \cref{lem:iso_endomorphisms}, 
	the condition 
	$f'\circ g' \simeq D_\bullet^{q_1+q_2} \cdot \id$
	implies that the restriction of $f'\circ g'$ 
	to the $\bullet$-end of the $\DD(Q_{\nicefrac{p}{q}})$ 
	must contain a component~$D_\bullet^{q_1+q_2}$. 
	Thus, the restrictions of $f'$ and $g'$ 
	to the $\bullet$-ends of the two complexes 
	must contain components
	$\pm D_\bullet^{q_1+a}$ and 
	$\pm D_\bullet^{q_2-a}$, 
	respectively, where $a=\tfrac{1}{2}(\ell-1)$. 
	Observe that these components must be, up to sign, multiples 
	of the respective components of the maps $f$ and $g$ 
	in \cref{tab:lambda_rational_tangles:maps_for_cases_p/q} 
	by non-negative powers of~$D_\bullet$. 
	This is because there are no chain maps
	which restrict to $\pm \id$ on those components 
	that are labelled $D_\bullet$ 
	in \cref{tab:lambda_rational_tangles:maps_for_cases_p/q}. 
	Therefore, $(q_1,q_2)\in \mathsf{V}(1-\alpha,\alpha)$. 
\end{proof}

\begin{proposition}\label{prop:Lambda-tangles-to-links}
	Let \(T_1\), \(T_2\), \(T\) be three Conway tangles with \(\conn{T_1}=\conn{T_2}\neq \conn{T}\). 
	Let \(a\in\{D_\bullet,S^2,D_\circ\}\subset\mathcal{B}\) be such that 
	\(\{\conn{T_1}, \conn{T},x_a\}=\{\Ni, \ConnectivityX, \No\}\). 
	Then
	\[
	\Lambda(T_1\cup T,T_2\cup T)
	\supseteq
	\Lambda_{a}(T_1,T_2). 
	\]
\end{proposition}

\begin{proof}
    Let \((q_1,q_2)\in \Lambda_a(T_1,T_2)\), and consider maps $f$ and $g$ realising it. 
    By \cref{lem:maps-tangles-to-links}, they induce homogeneous maps 
	\[
	\begin{tikzcd}
		\BNr(T_1\cup T)
		\arrow[bend left=5]{r}{f'}
		\arrow[bend right=5,leftarrow,swap]{r}{g'}
		&
		\BNr(T_2\cup T)
	\end{tikzcd}
	\]
	with 
    \(
	f'\circ g'\simeq G^{q_1+q_2}\cdot \id
	\), 
	\(
	g'\circ f'\simeq G^{q_1+q_2}\cdot \id
	\)
    and \(\QGrad{q}(f')=-2q_1,\, \QGrad{q}(g')=-2q_2\). 
    This shows that $(q_1,q_2)\in \Lambda(T_1\cup T, T_2\cup T)$.
\end{proof}

\begin{proof}[Proof of \cref{thm:obstruction_rational_repl_other_cr}]
	We may write \(K_1=\textnormal{$\vcenter{\hbox{\def\svgwidth{11.6pt}}}$}\cup T\) and \(K_2=Q_{\nicefrac{p}{q}}\cup T\), for a tangle $T$ with connectivity~$\conn{T}=\No$. 
	Now apply \cref{lem:Lambda_rational_repl} and \cref{prop:Lambda-tangles-to-links} with \(T_1=\textnormal{$\vcenter{\hbox{\def\svgwidth{11.6pt}}}$}\) and \(T_2=Q_{\nicefrac{p}{q}}\). 
\end{proof}

\subsection{\texorpdfstring{Relationship with the $s$-invariant and with $G$-torsion}{Relationship with the s-invariant and with G-torsion}}

\begin{proposition} \label{thm:s_and_biglambda}
For any two knots \(K_1\) and \(K_2\) and any field~\(\F\),
    \[
    \Lambda(K_1,K_2;\F) 
    \subseteq 
    \mathsf{V}
    \left(
    \nicefrac{s_{\F}(K_1,K_2)}{2}\, ,\, 
    \nicefrac{s_{\F}(K_2,K_1)}{2}
    \right).
    \]
It follows that
    \[
    \Lambda(K_1,K_2) 
    \subseteq 
    \mathsf{V}
    \left(
    \max_{\F} (\nicefrac{s_{\F}(K_1,K_2)}{2})\, ,\, 
    \max_{\F} (\nicefrac{s_{\F}(K_2,K_1)}{2})
    \right).
    \]
\end{proposition}

\begin{proof}
For the first claim, let $(a,b)\in \Lambda(K_1,K_2;\F)$, with~$a+b \geq 0$, and consider a pair of homogeneous chain maps $f$ and $g$ witnessing this, i.e.\ $f \circ g \simeq G^{a+b} \cdot \id,\, g \circ f \simeq G^{a+b} \cdot \id$ and $\QGrad{q}(f)=-2a,\, \QGrad{q}(g)=-2b$.
By \cref{lem:qdeg_f_g_and_s_invariant}, 
\[
2a \geq s_{\F}(K_1,K_2) 
\qquad 
\text{and} 
\qquad 
2b \geq s_{\F}(K_2,K_1).
\]
The second statement is obtained by applying the first to get
\[
\Lambda(K_1,K_2) 
\subseteq 
\bigcap_{\F} \Lambda(K_1,K_2;\F) 
\subseteq 
\bigcap_{\F} 
\mathsf{V}
\left(
\nicefrac{s_{\F}(K_1,K_2)}{2}\, ,\, 
\nicefrac{s_{\F}(K_2,K_1)}{2}
\right)
\]
and observing that
\[
\bigcap_{\F} 
\mathsf{V}
\left(
\nicefrac{s_{\F}(K_1,K_2)}{2}\, ,\, 
\nicefrac{s_{\F}(K_2,K_1)}{2}
\right) 
= 
\mathsf{V}
\left(
\max_{\F} (\nicefrac{s_{\F}(K_1,K_2)}{2})\, ,\, 
\max_{\F} (\nicefrac{s_{\F}(K_2,K_1)}{2})
\right).\qedhere
\]
\end{proof}

\begin{proposition}\label{prop:torsion_and_biglambda}
   Let \(\F\) be a field. 
   Then for any two knots \(K_1\) and \(K_2\), 
   \[
   \Lambda(K_1,K_2;\F)
   \supseteq
   \mathsf{V}_{k_\F} 
   \cap 
   \mathsf{V}(
   \nicefrac{s_\F(K_1,K_2)}{2},
   \nicefrac{s_\F(K_2,K_1)}{2}
   ),
   \]
   where 
   \(
   k_\F
   =
   \max \{ \mathfrak{u}(K_1;\F), \mathfrak{u}(K_2;\F) \}
   \). 
   In fact, for any knot~\(K\),
   \[
   \Lambda(K;\F)
   =
   \mathsf{V}_{\mathfrak{u}(K;\F)} \cap \mathsf{V}(\nicefrac{s_\F(K)}{2},\nicefrac{-s_\F(K)}{2}).
   \]
\end{proposition}

\begin{proof}
    The first statement is a direct consequence of \cref{lem:f_g_torsion}.    
    The inclusion $\supseteq$ of the second statement follows from the first by setting $K_1=K$ and~$K_2=U$. 
    As for the other inclusion, clearly $\Lambda(K;\F) \subseteq \mathsf{V}(\nicefrac{s_\F(K)}{2},\nicefrac{-s_\F(K)}{2})$ by \cref{thm:s_and_biglambda}. 
    The fact that $\Lambda(K;\F) \subseteq \mathsf{V}_{\mathfrak{u}(K;\F)}$ follows from \cref{prop:upper_bds_on_lambda-_fields}\cref{enu4:prop:upper_bds_on_lambda-_fields} 
    combined with
    \cref{rem:lambda_from_Lambda}\cref{item:rem:lambda_from_Lambda:lambda}.
\end{proof}

Mirroring \cref{rem:torsion_integers}, we note that \cref{prop:torsion_and_biglambda} also holds over the integers, provided the complexes~$\BNr(K_i)$, for~$i=1,2$, split as a sum of pawns and generalised knights.

\begin{remark}
    As observed in \cref{rem:lambda_from_Lambda}, all information provided by $\lambda$ and $\lambdau$ can be deduced from~$\Lambda$. In particular, \cref{thm:s_invariant_and_graded_lambda}, \cref{cor:s_invariant_and_ungr_lambda}, and \cref{prop:upper_bds_on_lambda-_fields} are direct consequences of \cref{thm:s_and_biglambda,prop:torsion_and_biglambda}.
\end{remark}

\begin{appendix}
\section{The Bar-Natan complex of rational tangles}\label{sec:appendix}

We discuss in more detail the structure of the chain complex~\(\DD(Q_r)\), for an oriented rational tangle~\(Q_r\).
As stated in \cref{sec:prelim}, it is enough to consider positive rational numbers~\(r\), as the complexes \(\DD(Q_{-r})\) can be obtained from \(\DD(Q_r)\) via \cref{lem:mirror_complex}.
For any \(r\in\Q_{>0}\), the complex \(\DD(Q_r)\) is chain homotopy equivalent to a zigzag complex (this was shown in \cite{thompson} for Bar-Natan's dotted category and generalised in \cite{lambda1} to our setting). Up to global shifts in homological and quantum degrees, this complex only depends on~\(r\).
We start by associating an ungraded zigzag complex (\cref{def:zigzag}) \(zz(r)\) to any~\(r\in\Q_{>0}\). The chain modules and the directions of the differentials of \(zz(r)\) are determined by the following rules.
\begin{enumerate}[(Z1)]
    \item \label{en:zz1}
    \(zz(1) = \circ \longrightarrow \bullet.\)
  	\item \label{en:zzinv}
    \(zz(\nicefrac{1}{r})\) is obtained from \(zz(r)\) by switching \(\circ\) and~\(\bullet\), and reversing the directions of all arrows.
    \item \label{en:zz+1}
    \(zz(r + 1)\) is obtained from \(zz(r)\) by replacing each arrow as shown in \cref{table:x+1}.
\end{enumerate}
\begin{table}[t]
    \begin{tabular}{cc}
    	\toprule
    replace & by  \\\hline 
    \raisebox{-4ex}{\rule{0em}{7ex}}%
    \begin{tikzcd}{\bullet} \ar[r,"\text{odd}",swap] & {\bullet}\end{tikzcd} &
    \begin{tikzcd}\bullet \ar[r,leftarrow] & \circ\ar[r] & \circ \ar[r] & \bullet \end{tikzcd} \\
    \raisebox{-4ex}{\rule{0em}{4ex}}%
    \makebox[4em][r]{\footnotesize not an end}\begin{tikzcd} {\bullet} \ar[r,"\text{even}",swap] & {\bullet}\end{tikzcd}\makebox[4em][l]{\footnotesize not an end} &
    \begin{tikzcd}\bullet \ar[r] & \bullet \end{tikzcd} \\
    \raisebox{-4ex}{\rule{0em}{4ex}}%
    \makebox[4em][r]{\footnotesize not an end}\begin{tikzcd}  {\bullet} \ar[r,"\text{even}",swap] & {\bullet}\end{tikzcd}\makebox[4em][l]{\footnotesize end} &
    \begin{tikzcd}\phantom{\circ}\ar[r,phantom] & \bullet \ar[r] & \bullet \ar[r,leftarrow] & \circ \end{tikzcd}\\
    \raisebox{-4ex}{\rule{0em}{4ex}}%
    \begin{tikzcd}\circ \ar[r] & \bullet\end{tikzcd} &
    \begin{tikzcd}\circ \ar[r] & \circ\ar[r] &  \bullet \ar[r,phantom] & \phantom{\circ}\end{tikzcd}\\
    \begin{tikzcd}\circ \ar[r] & \circ\end{tikzcd} &
    \begin{tikzcd}\circ \ar[r] & \circ \end{tikzcd}
    \\
    \bottomrule
    \end{tabular}
    \caption{How to obtain \(zz(r+1)\) from \(zz(r)\) (cf.\ \cite[Table 2]{lambda1}).
    Note that on the left column there appears to be one case missing, namely 
    \(\text{\footnotesize end } \bullet \xrightarrow{\text{even}} \bullet \text{ \footnotesize not an end}\).
    However, by \cref{lem:one_saddle_per_segment_of_zz}, this case can never occur.
    }
    \label{table:x+1}
\end{table}%
We remind the reader that
the differentials of any zigzag complex are automatically determined 
by the underlying chain modules and the directions of arrows 
(\cref{rem:zigzag-differentials}).
Let us make sure that the rules \ref{en:zz1}--\ref{en:zz+1}
indeed determine \(zz(r)\) for all \(r\in \Q_{>0}\).
Given \(r\in \Q_{>0}\), we call a maximal sequence of consecutive differentials of \(zz(r)\) oriented in the same direction a \emph{segment} of \(zz(r)\). 

\begin{lemma} \label{lem:one_saddle_per_segment_of_zz}
	For any~\(r\in \Q_{>0}\), each segment of \(zz(r)\) contains exactly one differential~\(S_\circ\). In other words, each segment is of the form
        \[
		\underbrace{\circ \to \circ \to \ldots \to \circ \to}_{m\geq 0} \circ \xrightarrow{S_\circ} \bullet \underbrace{\to \bullet \to \bullet \to \ldots \to \bullet}_{n\geq 0}.
        \]
	In particular, the configurations 
	\[
	\ldots \; \longrightarrow \circ \qquad \text{or} \qquad \ldots \; \longleftarrow \bullet
	\]
	cannot be found at either end of the complex~\(zz(r)\), 
	so the rules in \cref{table:x+1} are complete. 
\end{lemma}

\begin{proof}
	For \(r=1\), the statement is clearly true. 
	It is enough to show that the property is preserved by the rules \ref{en:zz1}--\ref{en:zz+1}. For the first two rules, this is clearly true. The fact that the operation \(zz(r) \leadsto zz(r+1)\) of \ref{en:zz+1} preserves the property can be seen as follows: 
	Firstly, since \(zz(r)\) satisfies the property by assumption, it does not contain \(\text{\footnotesize end } \bullet \xrightarrow{\text{even}} \bullet \text{ \footnotesize not an end}\) nor a differential \(\bullet \to \circ\). 
	Thus, \(zz(r+1)\) is well-defined. 
	Secondly, one may check one-by-one that the rules in \cref{table:x+1} preserve the property. 	
\end{proof}

The reason for the definition of \(zz\) is the following.
\begin{proposition}[{\cite[Theorem 5.6]{lambda1}, \cite{thompson}}]\label{prop:zigzag_equivalent}
    Up to shifts in homological and quantum gradings, \(zz(r) \simeq \DD(Q_r)\) for any \(r\in \Q_{>0}\) and any orientation on~\(Q_r\). 
\end{proposition}

In the rest of the appendix, we state and prove some general facts about \(zz(r)\) and compute the quantum and homological grading shifts of \cref{prop:zigzag_equivalent} when \(Q_r\) has connectivity~\(\ConnectivityX\).
The following is a reformulation of \cref{lem:end_zz}.
\begin{lemma}\label{lem:end_zzAppendix}
    Let \(p\) and \(q\) be odd coprime integers such that \(\nicefrac{p}{q}>0\). The chain complex \(zz(\nicefrac{p}{q})\) has an odd \(\circ\)-end and an odd \(\bullet\)-end. In a neighborhood of the \(\bullet\)-end, the complex \(zz(\nicefrac{p}{q})\) has the following shape
    \begin{equation*}
    zz(\nicefrac{p}{q})=
    \begin{cases*}
        \cdots \; \circ \overset{D}{\longrightarrow} \circ \overset{S}{\longrightarrow} \bullet & if \(\nicefrac{p}{q} > 1\) \\
        \phantom{\cdots \; \circ \overset{D}{\longrightarrow}}\;
        \circ \overset{S}{\longrightarrow} \bullet & if \(\nicefrac{p}{q} = 1\) \\
        \cdots \; \circ \overset{D}{\longleftarrow} \circ \overset{S}{\longrightarrow} \bullet & if \(\nicefrac{1}{2} < \nicefrac{p}{q} < 1\) \\
        \cdots \; \bullet \overset{D}{\longrightarrow} \bullet \overset{S^2}{\longrightarrow} \bullet & if \(0 < \nicefrac{p}{q} < \nicefrac{1}{2}\) \\
    \end{cases*}
    \end{equation*}
\end{lemma}

\begin{proof}
    For \(\nicefrac{p}{q}=1\), the statement is immediate by \ref{en:zz1}.
    Let now \(p\) and \(q\) be odd coprime integers different from 1, such that~\(\nicefrac{p}{q}>0\). Up to inserting or removing two consecutive \(r \mapsto \nicefrac{1}{r}\), \(\nicefrac{p}{q}\) is obtained uniquely from 1 by a sequence \(\mathcal{S}_{\nicefrac{p}{q}}\) of \(r \mapsto r+1\) and \(r \mapsto \nicefrac{1}{r}\). This sequence can be deduced from the continued fraction expansion of \(\nicefrac{p}{q}\) coming from Euclid's algorithm:
    \[
    \frac{p}{q} = a_n + \frac{1}{
    a_{n-1} + \frac{1}{
    a_{n-2} + \frac{1}{
    \cdots + \frac{1}{a_1+1}
    }
    }
    }
    \]
    with \(a_n\in \Z_{\geq 0}\) and \(a_i\in \Z_{>0}\) for~\(i<n\).
    If \(\nicefrac{p}{q}>1\), one finds~\(a_n>0\), which implies that the last operation of the sequence \(\mathcal{S}_{\nicefrac{p}{q}}\) is given by~\(r \mapsto r+1\). On the other hand, if~\(\nicefrac{p}{q}<1\), then \(a_n=0\) and the last operation is \(r \mapsto \nicefrac{1}{r}\).

    Let us start by considering the case~\(\nicefrac{p}{q}>1\).  The last operation of \(\mathcal{S}_{\nicefrac{p}{q}}\) is~\(r \mapsto r+1\), therefore the \(\bullet\)-end of \(zz(\nicefrac{p}{q})\) comes from an end of \(zz(\nicefrac{p-q}{q})\) via the rules of \cref{table:x+1}.
    By examining the right column of the table, we see that the odd \(\bullet\)-end of \(zz(\nicefrac{p}{q})\) has to be of one of the following types:
    \begin{gather*}
        \text{\footnotesize end } \bullet \longleftarrow \circ \longrightarrow \circ \longrightarrow \bullet \cdots \\
        \cdots \bullet \longleftarrow \circ \longrightarrow \circ \longrightarrow \bullet \text{ \footnotesize end} \\
        \cdots \circ \longrightarrow \circ \longrightarrow \bullet \text{ \footnotesize end}
    \end{gather*}
    Furthermore, we observe that the first configuration cannot occur: It would have to come from an arrow of the form \(\text{\footnotesize end } \bullet \longrightarrow \bullet \cdots\) in~\(zz(\nicefrac{p-q}{q})\), which is excluded by \cref{lem:one_saddle_per_segment_of_zz}. This concludes the first case.

    We turn to the case~\(\nicefrac{p}{q}<1\). The last operation of \(\mathcal{S}_{\nicefrac{p}{q}}\) is~\(r \mapsto \nicefrac{1}{r}\), and \(zz(\nicefrac{p}{q})\) is obtained from \(zz(\nicefrac{q}{p})\) by rule \ref{en:zzinv}. 
    Thus, we are now interested in the odd \(\circ\)-end of \(zz(\nicefrac{q}{p})\). By a similar argument as above, observing that~\(\nicefrac{q}{p}>1\), we conclude that this end is of one of the two types below: 
    \begin{gather}
        \label{eq:circ-end_zz_qp<2} \cdots \bullet \longrightarrow \bullet \longleftarrow \circ \text{ \footnotesize end} \\
        \label{eq:circ-end_zz_qp>2} \text{\footnotesize end } \circ \longrightarrow \circ \cdots
    \end{gather}
    The last operation in \(\mathcal{S}_{\nicefrac{q}{p}}\) is~\(r\mapsto r+1\), while the one before that is 
    \(r \mapsto \nicefrac{1}{r}\) if~\(\nicefrac{q}{p}<2\), and \(r \mapsto r+1\) if~\(\nicefrac{q}{p}>2\). 
    One can check that, in the first case, only configuration \cref{eq:circ-end_zz_qp<2} can occur, while in the second case the only possible configuration is \cref{eq:circ-end_zz_qp>2}.
    Furthermore, in the latter case, we can deduce one more differential of the complex: By \cref{lem:one_saddle_per_segment_of_zz} and the fact that the \(\circ\)-end is odd, it follows that the differential adjacent to \(\text{\footnotesize end } \circ \longrightarrow \circ\) must be of the form~\(\circ \longrightarrow \circ\). Indeed, it has to be even, and therefore cannot be a~\(S_\circ\). Thus, we get \(\text{\footnotesize end } \circ \longrightarrow \circ \longrightarrow \circ \cdots\).
\end{proof}

Consider a Conway tangle \(T\) with a basepoint. Suppose that~\(\conn{T}=\ConnectivityX\), and that \(T\) is oriented as the tangle \textnormal{$\vcenter{\hbox{\def\svgwidth{11.6pt}}}$}.
By the equivalence of categories of \cref{prop:equiv_cat_4-tangles}, we can view \(\DD(T)\) as a complex \(\llbracket T \rrbracket\) over \(\text{Mat}(\Cob_{/l}(4))\).
By joining with an arc the two upper end-points of every Conway tangle appearing in~\(\llbracket T \rrbracket\), we obtain a chain complex over~\(\text{Mat}(\Cob_{/l}(2))\). 
This complex is chain homotopy equivalent to~\(\llbracket\overline{T}\rrbracket\), where \(\overline{T}\) is the 2-ended tangle obtained by joining the two upper end-points of~\(T\).
Note that this is where the orientation of \(T\) comes into play: The operation of connecting the two upper end-points of \(T\) with an arc must be consistent with the orientation of the tangle. 
By \cref{prop:equiv_cat_2-tangles}, we can view \(\llbracket\overline{T}\rrbracket\) as a chain complex over \(\mathcal{M}_{\Z[G]}\) which is chain homotopy equivalent to~\(\BNr(\overline{T})\).
Now observe that \(T(0)\) is obtained from \(\overline{T}\) by joining the two remaining end-points with an arc. Then \(\BNr(\overline{T})\) extends to the chain complex~\(\BNr(T(0))\). 
This can be summarized as follows:
\begin{equation} \label{eq:corresp_complexes}
\DD(T) \overset{\text{\cref{prop:equiv_cat_4-tangles}}}{\longleftrightarrow}
\llbracket T \rrbracket  \overset{\text{join upper end-points}}{\longrightarrow}
\llbracket\overline{T}\rrbracket \overset{\text{\cref{prop:equiv_cat_2-tangles}}}{\longleftrightarrow}
\BNr(T(0)).
\end{equation}

When \(p\) and \(q\) are odd positive coprime integers, the complex \(zz(\nicefrac{p}{q})\) has an odd \(\bullet\)-end and an odd \(\circ\)-end. For~\(i,\ell\in\Z\), let \(t^i\QGrad{q^\ell}zz(\nicefrac{p}{q})\) be the graded chain complex obtained from \(zz(\nicefrac{p}{q})\) by setting the bigrading of the \(\bullet\)-end to be~\(\GGzqh{\bullet}{0}{\ell}{i}\). 
The grading of all other modules of \(t^i\QGrad{q^\ell}zz(\nicefrac{p}{q})\) is automatically determined by requiring that the differentials increase the homological grading by 1 and preserve the quantum grading. By \cref{prop:zigzag_equivalent}, if the tangle \(Q_{\nicefrac{p}{q}}\) is given an orientation, there exist \(i,\ell\in \Z\) such that \(t^i\QGrad{q^\ell}zz(\nicefrac{p}{q})\simeq \DD(Q_{\nicefrac{p}{q}})\). The gradings are determined by the following lemma, which is a reformulation of \cref{lem:grading_end_zz}.

\begin{lemma}\label{lem:grading_end_zzAppendix}
    Let \(p\) and \(q\) be odd coprime integers. The rational tangle \(Q_{\nicefrac{p}{q}}\) has connectivity~\(\ConnectivityX\). We equip it with the same orientation as the tangle \textnormal{$\vcenter{\hbox{\def\svgwidth{11.6pt}}}$}.
    If \(p\) and \(q\) are positive, then \(t^0\QGrad{q^\ell}zz(\nicefrac{p}{q})\simeq \DD(Q_{\nicefrac{p}{q}})\) for \(\ell = s_\Q(Q_{\nicefrac{p}{q}}(0))-1\).
    As a consequence, the odd \(\bullet\)-end of \(\DD(Q_{\nicefrac{p}{q}})\) has homological degree \(0\) and quantum degree
    \[
    \begin{cases*}
        s_\Q(Q_{\nicefrac{p}{q}}(0))-1 & if \(\nicefrac{p}{q}>0\), \\
        s_\Q(Q_{\nicefrac{p}{q}}(0))+1 & if \(\nicefrac{p}{q}<0\).
    \end{cases*}
    \]
\end{lemma}

\begin{proof}
     Let \(i,\ell\in \Z\) such that \(t^i\QGrad{q^\ell}zz(\nicefrac{p}{q})\simeq \DD(Q_{\nicefrac{p}{q}})\).
     We will use \cref{eq:corresp_complexes} to relate \(t^i\QGrad{q^\ell}zz(\nicefrac{p}{q})\) to the chain complex \(\BNr(Q_{\nicefrac{p}{q}}(0))\).
     In order to reduce the complexity of \(\llbracket\overline{Q_{\nicefrac{p}{q}}}\rrbracket\) and \(\BNr(Q_{\nicefrac{p}{q}}(0))\), we will apply delooping and Gaussian elimination. These tools were first introduced by Bar-Natan in \cite{BN_fast}; the version of delooping used here is described in \cref{fig:delooping} (cf.\ \cite[Figure 5]{lambda1}).

    \begin{figure}[t]
        \centering
        {\def\svgwidth{190pt}
        %% Creator: Inkscape 1.3.2 (091e20ef0f, 2023-11-25, custom), www.inkscape.org
%% Accompanies image file 'deloopingv2.pdf' (pdf, eps, ps)
\begingroup%
  \makeatletter%
  \providecommand\color[2][]{%
    \errmessage{(Inkscape) Color is used for the text in Inkscape, but the package 'color.sty' is not loaded}%
    \renewcommand\color[2][]{}%
  }%
  \providecommand\transparent[1]{%
    \errmessage{(Inkscape) Transparency is used (non-zero) for the text in Inkscape, but the package 'transparent.sty' is not loaded}%
    \renewcommand\transparent[1]{}%
  }%
  \providecommand\rotatebox[2]{#2}%
  \newcommand*\fsize{\dimexpr\f@size pt\relax}%
  \newcommand*\lineheight[1]{\fontsize{\fsize}{#1\fsize}\selectfont}%
  \ifx\svgwidth\undefined%
    \setlength{\unitlength}{394.82471376bp}%
    \ifx\svgscale\undefined%
      \relax%
    \else%
      \setlength{\unitlength}{\unitlength * \real{\svgscale}}%
    \fi%
  \else%
    \setlength{\unitlength}{\svgwidth}%
  \fi%
  \global\let\svgwidth\undefined%
  \global\let\svgscale\undefined%
  \makeatother%
  \begin{picture}(1,0.49019829)%
    \lineheight{1}%
    \setlength\tabcolsep{0pt}%
    \put(0,0){\includegraphics[width=\unitlength,page=1]{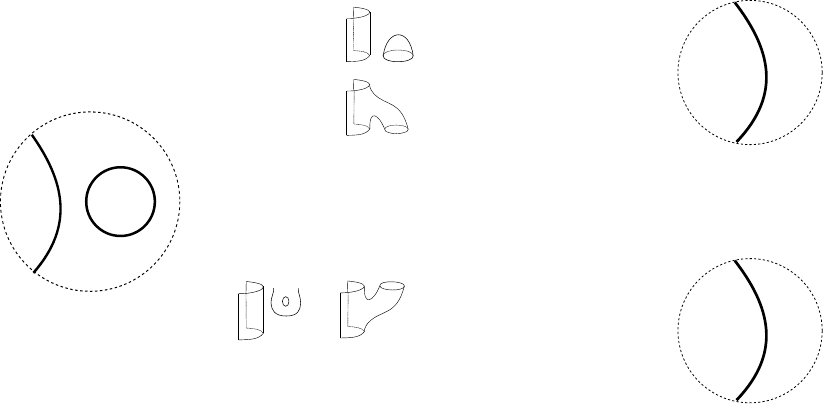}}%
    \put(0.38261318,0.102198){\color[rgb]{0,0,0}\makebox(0,0)[lt]{\lineheight{1.25}\smash{\begin{tabular}[t]{l}-\end{tabular}}}}%
    \put(0,0){\includegraphics[width=\unitlength,page=2]{deloopingv2.pdf}}%
    \put(0.71967516,0.39803026){\color[rgb]{0,0,0}\makebox(0,0)[lt]{\lineheight{1.25}\smash{\begin{tabular}[t]{l}$\QGrad{q^{-1}}$\end{tabular}}}}%
    \put(0.71967516,0.07415242){\color[rgb]{0,0,0}\makebox(0,0)[lt]{\lineheight{1.25}\smash{\begin{tabular}[t]{l}$\QGrad{q^{1}}$\end{tabular}}}}%
  \end{picture}%
\endgroup%
}
        \caption{Delooping (cf.\ \cite[Figure 5]{lambda1}).}
        \label{fig:delooping}
    \end{figure} 
     
     Let \(p\) and \(q\) be odd coprime positive integers.
     By \cref{lem:end_zzAppendix}, if~\(\nicefrac{p}{q}>\nicefrac{1}{2}\), the differential connected to the \(\bullet\)-end of \(t^i\QGrad{q^\ell}zz(\nicefrac{p}{q})\) is given by~\(S_\circ\). By \cref{eq:corresp_complexes}, one has:
     \[
        \llbracket Q_{\nicefrac{p}{q}}\rrbracket \simeq \quad
        \cdots\;
        t^{i-1}\QGrad{q^{\ell-1}} \Ni
        \overset{\vcenter{\hbox{\def\svgwidth{7pt}\includegraphics[width=0.03\textwidth]{figures/Scirc.pdf}}}}{\longrightarrow}
        t^{i}\QGrad{q^{\ell}} \No
        \; , \qquad
        \llbracket \overline{Q_{\nicefrac{p}{q}}} \rrbracket \simeq \quad
        \cdots\;
        t^{i-1}\QGrad{q^{\ell-1}} 
        \vcenter{\hbox{\def\svgwidth{7pt}\includegraphics[width=0.028\textwidth]{figures/2-tangle.pdf}}}
        \overset{\vcenter{\hbox{\def\svgwidth{7pt}\includegraphics[width=0.03\textwidth]{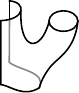}}}}{\longrightarrow}
        t^{i}\QGrad{q^{\ell}} 
        \vcenter{\hbox{\def\svgwidth{7pt}\includegraphics[width=0.028\textwidth]{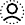}}}
     \]
    Now, we apply delooping to \(\llbracket \overline{Q_{\nicefrac{p}{q}}} \rrbracket\),  and use the correspondence between \(\llbracket \overline{Q_{\nicefrac{p}{q}}} \rrbracket\) and~\(\BNr(Q_{\nicefrac{p}{q}}(0))\). The complex \(\BNr(Q_{\nicefrac{p}{q}}(0))\) is chain homotopy equivalent to the following, where the dashed arrows stand for zero maps:
    \[
    \begin{tikzcd}[row sep=0pt]
        &
        t^{i}\QGrad{q^{\ell-1}} \Z[G]
        \\
        \cdots\; t^{i-1}\QGrad{q^{\ell-1}} \Z[G]
        \arrow{ru}{1}
        \arrow[swap]{rd}{G}
        &
        \oplus
        \\
        &
        t^{i}\QGrad{q^{\ell+1}} \Z[G]
    \end{tikzcd}
    \quad
    \overset{\raisebox{1.5ex}{\parbox{12ex}{\centering\tiny Gaussian elimination}}}{\simeq}
    \quad
    \begin{tikzcd}[row sep=0pt]
        &
        0
        \\
        \cdots\; 0
        \arrow[dashed]{ru}
        \arrow[dashed]{rd}
        &
        \oplus
        \\
        &
        t^{i}\QGrad{q^{\ell+1}} \Z[G]
    \end{tikzcd}
    \]
    This proves that \(\BNr(Q_{\nicefrac{p}{q}}(0))\) decomposes as 
    \[
    \BNr(Q_{\nicefrac{p}{q}}(0)) \simeq t^{i}\QGrad{q^{\ell+1}} \Z[G] \oplus C,
    \]
    for some chain complex~\(C\). On the other hand, by \cref{rem:pawn_knight_pieces}, we know that the complex \(\BNr(Q_{\nicefrac{p}{q}}(0);\Q[G])\) decomposes as a sum of a single pawn piece \(t^0\QGrad{q^{s_\Q(Q_{\nicefrac{p}{q}}(0))}} \Q[G]\) and some \(G^k\)-knight pieces.
    Therefore, when taking the tensor product of \(\BNr(Q_{\nicefrac{p}{q}}(0))\) with~\(\Q[G]\), the summand \(t^{i}\QGrad{q^{\ell+1}} \Z[G]\) must be sent to the pawn piece of \(\BNr(Q_{\nicefrac{p}{q}}(0);\Q[G])\). This means that
\(i=0\) and \(\ell+1=s_\Q(Q_{\nicefrac{p}{q}}(0))\), as desired.
    
    Let now \(0<\nicefrac{p}{q}<\nicefrac{1}{2}\). In this case, the differential connected to the \(\bullet\)-end of \(t^i\QGrad{q^\ell}zz(\nicefrac{p}{q})\) is~\(S^2\). By \cref{eq:corresp_complexes}:
    \[
        \llbracket Q_{\nicefrac{p}{q}}\rrbracket \simeq \quad
        \cdots\;
        t^{i-1}\QGrad{q^{\ell-2}} \No
        \overset{\vcenter{\hbox{\def\svgwidth{7pt}\includegraphics[width=0.03\textwidth]{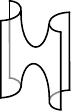}}}}{\longrightarrow}
        t^{i}\QGrad{q^{\ell}} \No
        \; , \quad
        \llbracket \overline{Q_{\nicefrac{p}{q}}} \rrbracket \simeq \quad
        \cdots\;
        t^{i-1}\QGrad{q^{\ell-2}} 
        \vcenter{\hbox{\def\svgwidth{7pt}\includegraphics[width=0.028\textwidth]{figures/2-tangleCircle.pdf}}}
        \overset{\vcenter{\hbox{\def\svgwidth{7pt}\includegraphics[width=0.03\textwidth]{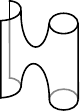}}}}{\longrightarrow}
        t^{i}\QGrad{q^{\ell}} 
        \vcenter{\hbox{\def\svgwidth{7pt}\includegraphics[width=0.028\textwidth]{figures/2-tangleCircle.pdf}}}
    \]
    Using delooping, the complex \(\BNr(Q_{\nicefrac{p}{q}}(0))\) is chain homotopy equivalent to the following complexes (where all dashed arrows are zero maps):
    \[
    \ldots
    \begin{tikzcd}[row sep=0pt]
        t^{i-1}\QGrad{q^{\ell-3}} \Z[G]
        \arrow[dashed]{r}
        \arrow[dashed]{rdd}
        &
        t^{i}\QGrad{q^{\ell-1}} \Z[G]
        \\
        \oplus
        &
        \oplus
        \\
        t^{i-1}\QGrad{q^{\ell-1}} \Z[G]
        \arrow[swap]{r}{G}
        \arrow[{pos=0.67}]{ruu}{1}
        &
        t^{i}\QGrad{q^{\ell+1}} \Z[G]
    \end{tikzcd}
    \hspace{.7em}
    \overset{\raisebox{1.5ex}{\parbox{12ex}{\centering\tiny Gaussian elimination}}}{\simeq}
    \hspace{.7em}
    \ldots
    \begin{tikzcd}[row sep=0pt]
        t^{i-1}\QGrad{q^{\ell-3}} \Z[G]
        \arrow[dashed]{r}
        \arrow[dashed]{rdd}
        &
        0
        \\
        \oplus
        &
        \oplus
        \\
        0
        \arrow[dashed]{r}
        \arrow[dashed]{ruu}
        &
        t^{i}\QGrad{q^{\ell+1}} \Z[G]
    \end{tikzcd}
    \]
    As in the previous case, \(\BNr(Q_{\nicefrac{p}{q}}(0))\) now decomposes as 
    \(
    \BNr(Q_{\nicefrac{p}{q}}(0)) \simeq t^{i}\QGrad{q^{\ell+1}} \Z[G] \oplus C',
    \)
    for some chain complex~\(C'\), and we find again that \(i=0\) and \({\ell+1=s_\Q(Q_{\nicefrac{p}{q}}(0))}\).

    For \(\nicefrac{p}{q}>0\), the statement about the odd $\bullet$-end of $\DD(Q_{\nicefrac{p}{q}})$ follows directly. We turn to the case~\(\nicefrac{p}{q}<0\). 
    By the first part of the proof, we have \(\DD(Q_{\nicefrac{-p}{q}}) \simeq t^0\QGrad{q^\ell}zz(\nicefrac{-p}{q})\) with \(\ell = s_{\mathbb{Q}}(Q_{\nicefrac{-p}{q}}(0)) - 1\).
    Since \(Q_{\nicefrac{-p}{q}} = -Q_{\nicefrac{p}{q}}\), it follows that \(\ell = -s_{\mathbb{Q}}(Q_{\nicefrac{p}{q}}(0)) - 1\). Moreover by \cref{lem:mirror_complex}, \(\DD(Q_{\nicefrac{-p}{q}}) \simeq \DD(-Q_{\nicefrac{p}{q}}) \simeq -\DD(Q_{\nicefrac{p}{q}})\). Thus
    \[ \DD(Q_{\nicefrac{p}{q}}) \simeq t^0 \QGrad{q^{-\ell}} \overline{zz}(\nicefrac{-p}{q}), \]
    where \(\overline{zz}(\nicefrac{-p}{q})\) is the complex obtained by reversing all arrows of~\(zz(\nicefrac{-p}{q})\). So the odd $\bullet$-end of $\DD(Q_{\nicefrac{p}{q}})$ is in homological degree 0 and quantum degree $-\ell = s_{\mathbb{Q}}(Q_{\nicefrac{p}{q}}(0)) + 1$ as claimed.
\end{proof}

\begin{remark}
    For any field \(\F\) and Conway tangle~\(T\), consider the chain complex \(\DD(T;\F)=\DD(T)\otimes_\Z \F\) over the \(\F\)-algebra~\(\BNAlgH \otimes_\Z \F\).
    The complex \(\DD(T;\F)\) admits a geometric interpretation in terms of the technology of \emph{immersed curves} developed in~\cite{KWZ}. 
    In particular, when \(T\) is a rational tangle~\(Q_r\), \(\DD(Q_r;\F)\) corresponds to a straight line segment of slope \(r\) in the covering space \(\mathbb{R}^2\setminus \Z^2\) of \(\mathbb{S}^2_{4,\ast}\) (where \(\mathbb{S}^2_{4,\ast}\) is the 4-punctured sphere with one special puncture). The segment connects the lifts of two non-special punctures of~\(\mathbb{S}^2_{4,\ast}\). 
    In \cite[Theorem~3.6]{KhMutation}, this technology is extended to coefficients in~\(\Z\) for a certain class of tangles, which includes rational tangles. 
    As a consequence, one obtains alternative and more conceptual proofs of \cref{lem:end_zzAppendix,lem:grading_end_zzAppendix}.
\end{remark}

\end{appendix}

\newcommand{\order}[1]{}%
\bibliographystyle{myamsalpha}
\bibliography{main}
\end{document}